\documentclass[a4paper,11pt]{article}

\usepackage[top=3.0cm, bottom=3.0cm, inner=3.0cm, outer=3.0cm,
includefoot]{geometry}

\usepackage{verbatim}
\usepackage{url}

\usepackage{geometry}
\usepackage{amssymb}
\usepackage{amsmath}
\usepackage{graphicx}
\usepackage{amsthm}
\usepackage{bbm}

\usepackage{color,soul}

\usepackage[T1]{fontenc}
\usepackage[utf8]{inputenc}
\usepackage{authblk}

\usepackage{enumerate}
\usepackage{tikz}
\usetikzlibrary{arrows}
\usetikzlibrary{decorations.markings}

\newcommand{\IR}{{\mathbb{R}}}

\newcommand{\eChar}{\begin{enumerate}[(i)]}
\newcommand{\eCharR}{\begin{enumerate}[(a)]}
\newcommand{\eBr}{\begin{enumerate}[(1)]}

\newcommand{\verteq}{\rotatebox{90}{$\,=$}}
\newcommand{\equalto}[2]{\underset{\displaystyle\overset{\mkern4mu\verteq}{#2}}{#1}}

\title
{Rigidity of the Bonnet-Myers inequality for graphs with respect to Ollivier Ricci curvature}

\author[1]{D. Cushing}
\author[1]{S. Kamtue}
\author[2]{J. Koolen}
\author[2]{S. Liu}
\author[3]{F. M\"unch}
\author[1]{N. Peyerimhoff}
\affil[1]{Department of Mathematical Sciences, Durham University}
\affil[2]{School of Mathematical Sciences, University of Science and Technology of China, and Wu Wen-Tsun Key Laboratory of Mathematics of CAS, Hefei} 
\affil[3]{Institute of Mathematics, Universit\"at Potsdam }
\date{\today}

\theoremstyle{plain}
\newtheorem{lemma}{Lemma}[section]
\newtheorem{theorem}[lemma]{Theorem}
\newtheorem{proposition}[lemma]{Proposition}
\newtheorem{corollary}[lemma]{Corollary}

\newtheorem*{theoremlichn}{Theorem \ref{thm:lichn}}
\newtheorem*{theoremdivis}{Theorem \ref{thm:DL_rel}}
\newtheorem*{theorembmbesharp}{Theorem \ref{thm:aBM-BEcurv}}
\newtheorem*{conjectureBM}{Conjecture \ref{conj:BM}}

\theoremstyle{definition}

\newtheorem{conjecture}[lemma]{Conjecture}
\newtheorem{definition}[lemma]{Definition}

\newtheorem{rem}[lemma]{Remark}

\newtheorem*{question}{Question}

\numberwithin{equation}{section}

\begin{document}

\maketitle

\begin{abstract}
  We introduce the notion of Bonnet-Myers and Lichnerowicz sharpness in
  the Ollivier Ricci curvature sense. Our main result is a
  classification of all self-centered Bonnet-Myers sharp graphs
  (hypercubes, cocktail party graphs, even-dimensional demi-cubes,
  Johnson graphs $J(2n,n)$, the Gosset graph and suitable Cartesian
  products). We also present a purely combinatorial reformulation
  of this result. We show that Bonnet-Myers sharpness implies Lichnerowicz
  sharpness. We also relate Bonnet-Myers sharpness to an upper bound
  of Bakry-\'Emery $\infty$-curvature, which motivates a general
  conjecture about Bakry-\'Emery $\infty$-curvature.
\end{abstract}

\section{Introduction and statement of results}

A fundamental question in geometry is in which way local properties
determine the global structure of a space. A famous result of this
kind is the Bonnet-Myers Theorem \cite{My41} for complete $n$-dimensional
Riemannian manifolds $M$ with $K = \inf {\rm Ric}_M(v) > 0$ (a
condition on the local invariant ${\rm Ric}_M ={\rm Tr}(R_M)$), where
the infimum is taken over all unit tangent vectors $v$ of $M$. Under
this condition, $M$ is compact and its diameter satisfies
\begin{equation} \label{eq:BM_RG_ineq}
{\rm diam}(M) \le \pi \sqrt{\frac{n-1}{K}}. 
\end{equation}
Moreover, Cheng's Rigidity Theorem \cite{Cheng75} states that this
diameter estimate is sharp if and only if $M$ is the $n$-dimensional
round sphere. Note that inequality \eqref{eq:BM_RG_ineq} can be
reformulated as an upper bound on the infimum of the Ricci curvature
in terms of the diameter, and this reformulation is the viewpoint we
will assume in this paper.

In the discrete setting of graphs there are several analogs of Ricci
curvature notions providing Bonnet-Myers type theorems (see, e.g.,
\cite{FS,HLLY,LMP,LLY11,Ol09},
...). In view of Cheng's rigidity result, it is natural to ask for
which graphs the Bonnet-Myers estimates is sharp. We call such graphs
\emph{Bonnet-Myers sharp} graphs. For example, in the case of
Bakry-\'Emery $\infty$-curvature, Bonnet-Myers sharp graphs have been
fully characterised and are only the hypercubes (see \cite{LMP2}).

The motivation of this paper is to study Bonnet-Myers sharpness with
respect to another curvature notion, namely, \emph{Ollivier Ricci
  curvature}. (In fact, we will consider a modification of Ollivier's
definition introduced in \cite{LLY11}.) Henceforth, all graphs
$G = (V,E)$ with vertex set $V$ and edge set $E$ will be simple
(loopless without multiple edges) and edges can be identified with
$2$-element subsets of $V$. In this paper, we will only formulate and
derive our results for regular graphs, that is, all vertices have the
same valency, even though similar questions can be posed for
non-regular graphs.

Ollivier Ricci curvature $\kappa(x,y)$ is a notion based on optimal
transport and is defined on pairs of different vertices $x,y \in
V$. The precise definition requires a longer introduction and is given
in Subsection \ref{sec:OllivKant}. Generally, $\kappa(x,y)$ is
positive if the average distance between corresponding neighbours of
$x$ and $y$ is smaller than $d(x,y)$. For now, we confine ourselves to
provide a useful connection of this curvature with a particular
combinatorial property, to provide the readers with some understanding
of this notion. Note that this Proposition follows directly from
Proposition \ref{prop:curvcalc0} by choosing $m = \frac{2D}{L} - 2$.

\begin{proposition} \label{prop:curvcalc}
  Let $G=(V,E)$ be a $D$-regular graph of diameter $L$ and
  $e=\{x,y\} \in E$. Assume that $e$ is contained in precisely
  $\frac{2D}{L}-2$ triangles and there is a perfect matching between
  the neighbours of $x$ and the neighbours of $y$ which are not involved
  in these triangles. Then we have
  $$ \kappa(x,y) = \frac{2}{L}. $$
\end{proposition}

Let us now state the discrete Bonnet-Myers Theorem for Ollivier Ricci
curvature and introduce the associated notion of Bonnet-Myers sharpness
for this curvature notion:

\begin{theorem}[Discrete Bonnet-Myers, see \cite{Ol09,LLY11}]
  \label{thm:DBM}
  Let $G= (V,E)$ be a connected $D$-regular graph and $\inf_{x \sim y}
  \kappa(x,y) > 0$. Then $G$ has finite diameter $L = {\rm diam}(G) <
  \infty$ and  
 \begin{equation} \label{eq:BM_est} 
  \inf_{x \sim y} \kappa(x,y) \le \frac{2}{L}. 
  \end{equation}
  We say that such a graph $G$ is
  {\bf{\textit{$\boldsymbol{(D,L)}$-Bonnet-Myers sharp}}} (with respect
    to Ollivier Ricci curvature) if \eqref{eq:BM_est} holds with
    equality.
\end{theorem}

Many of our results require the additional condition of
self-centeredness. Note that a graph $G=(V,E)$ is called
self-centered if, for every vertex $x \in V$, there exists a
vertex $\overline{x} \in V$ such that
$d(x,\overline{x}) = {\rm diam}(G)$ (see Subsection \ref{sec:graph_notation} for its definition).  Let us now state the main results of this paper.

\begin{itemize}
\item[(a)] \emph{Cartesian products:}
  $G_1 \times G_2 \times \cdots \times G_k$ is Bonnet-Myers sharp if and only if
  all factors $G_i$ are Bonnet-Myers sharp and satisfy
  \begin{equation} \label{eq:cartprod_cond0} 
    \frac{D_1}{L_1} = \frac{D_2}{L_2} =\cdots = \frac{D_k}{L_k},
  \end{equation} 
  where $D_i$ and $L_i$ are the vertex degrees and the diameters of the
  graphs $G_i$, respectively (see Theorem \ref{thm:cartprod}).
\item[(b)] Every Bonnet-Myers sharp graph is Lichnerowicz sharp (see
  Theorem \ref{thm:lichn}).
\item[(c)] \emph{Classification of self-centered Bonnet-Myers sharp graphs:}
  Self-centered Bonnet-Myers sharp graphs are precisely the following
  ones: Hypercubes, cocktail party graphs, the Johnson graphs
  $J(2n,n)$, even-dimensional demi-cubes, the Gosset graph and
  Cartesian products of them satisfying \eqref{eq:cartprod_cond0} (see Theorem
  \ref{thm:main}).
\item[(d)] Self-centered $(D,L)$-Bonnet-Myers sharp graphs are
  Bakry-{\'E}mery $\infty$-curvature sharp in all vertices with
  normalized $\infty$-curvature value $\frac{1}{D} + \frac{1}{L}$ (see
  Theorem \ref{thm:aBM-BEcurv}).
\end{itemize}

We provide more detailed information about these results in the next subsection. In particular, result (c) above is based on another result which can be reformulated in purely combinatorial terms. This combinatorial reformulation is derived in Subsection \ref{subsec:comb_res}.

\subsection{Our results on Bonnet-Myers sharp graphs}

It is useful to know that the vertex degree $D$ and the diameter $L$ of a
Bonnet-Myers sharp graph cannot be arbitrary:

\begin{theorem} \label{thm:DL_rel}
  Any $(D,L)$-Bonnet-Myers sharp graph satisfies $L \le D$. Moreover $L$
  must divide $2D$. 
\end{theorem}

This theorem is proved in Section \ref{sec:self-centBMsh}.

Another fundamental result between local and global properties of a
closed $n$-dimensional Riemannian manifold $M$ is Lichnerowicz' Theorem
\cite[p. 135]{Li58} which states that, under the condition
$K = \inf {\rm Ric}_M(v) > 0$, the smallest positive Laplace-Beltrami
eigenvalue $\lambda_1(M)$ satisfies
$$ \frac{n}{n-1}K \le \lambda_1. $$
The associated rigidity result is Obata's Theorem \cite{Ob62}, which
states that this eigenvalue estimate is sharp if and only if $M$ is
the $n$-dimensional round sphere. There is a discrete analogue of
Lichnerowicz' Theorem for Ollivier Ricci Curvature and the normalized
Laplacian $\Delta_G = D^{-1} A_G - {\rm Id}$ of an arbitrary graph
$G=(V,E)$, where $A_G$ denotes the adjacency matrix of $G$ and $D$ is
here a diagonal matrix whose entries are the valencies $d_x$ of the
vertices $x \in V$. In the case of a $D$-regular graph this matrix is
just given by $D \cdot {\rm Id}$. Alternatively, $\Delta_G$ can be
written as an operator acting on functions $f$ defined on the vertices
$V$ via
\begin{equation} \label{eq:Deltanorm} 
\Delta_G f(x) = \frac{1}{d_x} \sum_{y \sim x} (f(y)-f(x)). 
\end{equation}
The discrete Lichnerowicz' Theorem gives naturally rise to the
definition of Lichnerowicz sharpness:

\begin{theorem}[Discrete Lichnerowicz, see \cite{LLY11}] \label{thm:Lich}
  Let $G=(V,E)$ be a finite connected $D$-regular graph. Then we have
  for the smallest positive solution $\lambda_1$ of
  $\Delta_G f + \lambda f = 0$ with $\Delta_G$ given by
  \eqref{eq:Deltanorm},
  \begin{equation} \label{eq:L_est} 
  \inf_{x \sim y} \kappa(x,y) \le \lambda_1. 
  \end{equation}
  We say that the graph $G$ is
  {\bf{\textit{Lichnerowicz sharp}}} (with respect to
  Ollivier Ricci curvature) if \eqref{eq:L_est} holds with
  equality.
\end{theorem}

We have the following relation between these two sharpness properties, proved
in Section \ref{sec:curvandeig}:

\begin{theorem} \label{thm:lichn}
  Every Bonnet-Myers sharp graph is Lichnerowicz sharp.
\end{theorem}

Note, however, that the family of all Lichnerowicz sharp graphs is
much larger than the family of all Bonnet-Myers sharp graphs (see
Remark \ref{rem:lichcartsharp}). In Section \ref{sec:curvandeig}, we
classify a special subclass of Lichnerowicz sharp graphs, namely, all
distance regular Lichnerowicz sharp graphs with an additional spectral
condition.

Our main theorem is the following classification result of
self-centered Bonnet-Myers sharp graphs:

\begin{theorem} \label{thm:main}
  Self-centered Bonnet-Myers sharp graphs are precisely the following
  graphs:
  \begin{enumerate}
  \item hypercubes $Q^n$, $n \ge 1$;
  \item cocktail party graphs $CP(n)$, $n \ge 3$; 
  \item the Johnson graphs $J(2n,n)$, $n \ge 3$;
  \item even-dimensional demi-cubes $Q^{2n}_{(2)}$, $n \ge 3$; 
  \item the Gosset graph;
  \end{enumerate}
  and Cartesian products of 1.-5. satisfying the condition
  \eqref{eq:cartprod_cond0}.
\end{theorem}

Let us explain the proof of Theorem \ref{thm:main}: In Section
\ref{sec:examples} we show that all graphs in the list of Theorem
\ref{thm:main} are self-centered Bonnet-Myers sharp.
In Sections \ref{sec:transpgeod} and
\ref{sec:antBMstrspher} we show that every self-centered Bonnet-Myers sharp
graph is strongly spherical (this is the statement of Theorem
\ref{thm_spherical}; the notion of strongly spherical graphs is
introduced in Definition \ref{def:strspher}). 
Then we use the classification result \cite{Koo} (see Theorem
\ref{thm:strongspher_class} below), which states that strongly
spherical graphs are precisely the Cartesian products in the list
provided in Theorem \ref{thm:main}, thus completing the proof. \qed

Finally, we like to present connections between Bonnet-Myers sharpness
with respect to Ollivier Ricci curvature and normalized
Bakry-{\'E}mery $\infty$-curvature. Generally, relations between
different curvature notions of discrete spaces are a very interesting
and challenging topic. Normalized Bakry-{\'Emery} $\infty$-curvature
is defined on the vertices of a graph $G=(V,E)$ and denoted by
${\mathcal K}^{\rm n}_{G,x}(\infty)$, $x \in V$. For further details
and the precise definition, we refer the readers to Subsection
\ref{sec:BE-curvature}. We have the following results:

\begin{theorem} \label{thm:BMsharp-BEeq}
  Every $(D,L)$-Bonnet-Myers sharp graph $G=(V,E)$ satisfies
  \begin{equation} \label{eq:BM-BakryEm}
  \inf_{x \in V} {\mathcal K}^{\rm n}_{G,x}(\infty) \le 
  \frac{1}{D} + \frac{1}{L}. 
   \end{equation}
\end{theorem}

This result follows immediately from Theorem \ref{thm:curv_relation}
later in the paper.

As a consequence of Theorem \ref{thm:main} we also obtain
  
\begin{theorem} \label{thm:aBM-BEcurv}
  Let $G =(V,E)$ be a self-centered $(D,L)$-Bonnet-Myers sharp graph. Then $G$
  is Bakry-{\'E}mery $\infty$-curvature sharp at all vertices
  $x \in V$ and we have
  $$ 
  {\mathcal K}^{\rm n}_{G,x}(\infty) = \frac{1}{D} + \frac{1}{L}. 
  $$
\end{theorem}

The proof of this theorem is given in Subsection \ref{sec:BEcurvBMsharp}.

These last two results provide a better understanding why Bonnet-Myers
sharpness with respect to Bakry-{\'E}mery $\infty$-curvature is much
more restrictive: the only graphs with this property are the
hypercubes (see \cite{LMP2}). This result classifies all $D$-regular
graphs of diameter $L$ satisfying the condition
$$ 
\inf_{x \in V}{\mathcal K}^{\rm n}_{G,x}(\infty) = \frac{2}{L}, 
$$
which is much stronger than the equality condition of \eqref{eq:BM-BakryEm},
that is
\begin{equation} \label{eq:BM-BakryEmeq}
\inf_{x \in V}{\mathcal K}^{\rm n}_{G,x}(\infty) = \frac{1}{D} + \frac{1}{L},
\end{equation}
since $L \le D$, by Theorem \ref{thm:DL_rel}. By considering the
weaker condition \eqref{eq:BM-BakryEmeq} we encounter other graphs
like the ones given in the list of Theorem \ref{thm:main}. It is an
open question whether these graphs and suitable Cartesian products of
them are the only self-centered examples of $D$-regular graphs of diameter
$L$ satisfying \eqref{eq:BM-BakryEmeq}.

Moreover, we are not aware of any $D$-regular graph $G$ of diameter
$L$ which violates the condition \eqref{eq:BM-BakryEm}. This led us to
formulate the following conjecture:

\begin{conjecture} \label{conj:BM}
  Let $G=(V,E)$ be a finite connected $D$-regular graph of diameter $L$.
  Then we have
  $$ \inf_{x \in V} {\mathcal K}^{\rm n}_{G,x}(\infty) \le 
  \frac{1}{D} + \frac{1}{L}. $$ 
\end{conjecture}

Let us compare this conjecture with the combinatorial Bonnet-Myers
Theorem for normalized Bakry-{\'E}mery $\infty$-curvature proved in
\cite{LMP}, namely,
\begin{equation} \label{eq:BM_BE}
\inf_{x \in V} {\mathcal K}^{\rm n}_{G,x}(\infty) \le \frac{2}{L}. 
\end{equation}
Our conjecture can be viewed as a strengthening of \eqref{eq:BM_BE} in
the case $L \le D$.

\subsection{A combinatorial result}
\label{subsec:comb_res}

For readers interested in combinatorial graph theory, this
subsection provides a combinatorial reformulation of Theorem
\ref{thm_spherical}. The combinatorial version is given in Theorem
\ref{thm:main_comb} below and we think that it is of own interest.

We start with a combinatorial
property closely related to Ollivier Ricci curvature (see Proposition
\ref{prop:curvcalc0} below for the connection).

\begin{definition} \label{def:Lambda}
  Let $G=(V,E)$ be a regular graph. We say $G$ \emph{satisfies
  $\Lambda(m)$} at an edge $e = \{x,y\} \in E$ if the following holds:
  \begin{itemize}
  \item[(i)] $e$ is contained in at least $m$ triangles and
  \item[(ii)] there is a perfect matching between the neighbours of $x$ and
    the neighbours of $y$ not involved in these triangles.
  \end{itemize}
  We say that $G$ satisfies $\Lambda(m)$ if it satisfies $\Lambda(m)$
  at each edge.
\end{definition}

Next, we introduce the notions of antipodal and
strongly spherical graphs, which are based on intervals $[x,y]$ which are the set of all vertices lying on geodesics from $x$ to $y$ (see Subsection \ref{sec:graph_notation} for the precise definition of intervals). 

\begin{definition} \label{def:strspher}
  A graph $G = (V,E)$ is called \emph{antipodal} if, for
  every vertex $x \in V$, there exists another vertex
  $\overline{x} \in V$ satisfying $[x,\overline{x}]=V$.
  
  A graph $G = (V,E)$ is called \emph{strongly spherical} if $G$ and the induced subgraphs of all its intervals are antipodal. That is, the following two properties are necessary and sufficient conditions for strongly spherical graphs
  $G = (V,E)$:
  \begin{itemize}
  \item For every $x \in V$ there exists $\overline{x} \in V$ such
    that $ G = [x,\overline{x}]$.
  \item For every pair $x,y \in V$, $x \neq y$, and every
    $z \in [x,y]$ there exists $\overline{z} \in [x,y]$ such that
    $[x,y] = [z,\overline{z}]$.
  \end{itemize}
\end{definition}
  


To reformulate Theorem \ref{thm_spherical} in purely combinatorial
terms we need the following result which allows us to replace the
curvature condition by a combinatorial condition.

\begin{proposition} \label{prop:antBMcomb}
  Let $G$ be a $D$-regular finite connected graph of diameter $L$. The
  following are equivalent:
  \begin{itemize}
  \item $G$ is self-centered Bonnet-Myers sharp.
  \item $G$ is self-centered and satisfies $\Lambda\left(\frac{2D}{L}-2\right)$.
  \end{itemize}
Moreover, if any of these equivalent properties holds, then every edge of $G$ lies in precisely $\frac{2D}{L}-2$ triangles.
\end{proposition}

This proposition is a consequence of Corollary \ref{cor:numtriangle}
and Proposition \ref{prop:curvcalc0} later in the paper.

Using Proposition \ref{prop:antBMcomb}, Theorem \ref{thm_spherical}
translates then into the following equivalent combinatorial result:

\begin{theorem} \label{thm:main_comb}
  Let $G$ be a $D$-regular finite connected graph of diameter
  $L$. Assume that $G$ is self-centered and satisfies
  $\Lambda\left(\frac{2D}{L}-2\right)$. Then $G$ is strongly
  spherical.
\end{theorem}

It is an interesting question whether the condition of self-centeredness in
Theorem \ref{thm:main_comb} can be removed. If this were possible,
we could view this result as another example where \emph{local}
properties have a strong \emph{global} implication.

Finally, we would like to present the following classification of strongly
spherical graphs.

\begin{theorem}[Classification of strongly spherical graphs, see \cite{Koo}]
  \label{thm:strongspher_class}
  Strongly spherical graphs are precisely the Cartesian products
  $G_1 \times G_2 \times \cdots \times G_k$, where each factor $G_i$ is
  either a hypercube, a cocktail party graph, a Johnson graph
  $J(2n,n)$, an even dimensional demi-cube, or the Gosset graph.
\end{theorem}

\subsection{Outline of the paper}
\label{subsec:outline}

Here is an overview about the content of the following sections:
\begin{itemize}
\item Section \ref{sec:defn}: Basic graph theoretical notation and introduction
  into Ollivier Ricci curvature.
\item Section \ref{sec:cart}: Cartesian products preserve Bonnet-Myers sharpness
  under particular conditions.
\item Section \ref{sec:examples}: Presentation of all known examples of Bonnet-Myers
  sharp graphs.
\item Section \ref{sec:genfacts}: Discussion of various consequences of a useful
  relation between the Laplacian and Ollivier Ricci curvature. For
  example, all vertices of a Bonnet-Myers sharp graph lie on geodesics
  between any pair of antipoles.
\item Section \ref{sec:curvandeig}: Proof that Bonnet-Myers sharp graphs are Lichnerowicz
  sharp and classification of all distance-regular Lichnerowicz sharp
  graphs under an additional spectral condition.
\item Sections \ref{sec:transpgeod} and \ref{sec:antBMstrspher}: Proof of our main result: Classification of all
  self-centered Bonnet-Myers sharp graphs.
\item Section \ref{sec:BMextrdiam}: Classification of all Bonnet-Myers sharp graphs with
  extremal diameters $L=2$ and $L=D$.
\item Section \ref{sec:BM-BE-curvature}: Relations between Bonnet-Myers sharpness and an
  upper bound on the Bakry-\'Emery $\infty$-curvature. 
\end{itemize}

\section{Basic definitions, concepts and notation}
\label{sec:defn}

Throughout this paper, we restrict our graphs $G=(V,E)$ to be
undirected, simple, unweighted, finite, and connected. 

\subsection{Graph theoretical notation} \label{sec:graph_notation}
We write $x \sim y$ if there exists an edge between the vertices $x$ and $y$. The degree of a vertex $x \in V$ is denoted by $d_x$.
For a set of vertices $A\subseteq V$, the induced subgraph ${\rm Ind}_G(A)$ is the subgraph of $G$ whose vertex set is $A$ and whose edge set consists of all edges in $G$ that have both endpoints in $A$. 

For any two vertices $x,y\in V$, the (combinatorial) distance $d(x,y)$ is the length (i.e. the number of edges) in a shortest path from $x$ to $y$. Such paths of minimal length are also called {\it geodesics} from $x$ to $y$. An \emph{interval} $[x,y]$ is the set of all vertices lying on geodesics from $x$ to $y$, that is
$$ [x,y] = \{ z \in V \mid d(x,z) + d(z,y) = d(x,y) \}. $$

The {\it diameter} of $G$ is denoted by ${\rm diam}(G)= \max_{x,y\in V} d(x,y)$. A vertex $x\in V$ is called a {\it pole} if there exists a vertex $y\in V$ such that $d(x,y)={\rm diam}(G)$, in which case $y$ will be called an {\it antipole} of $x$ (with respect to $G$). A graph $G$ is called {\it self-centered} if every vertex is a pole. 

For $k\in\mathbb{N}$ and $x\in V$ we define the {\it $k$-sphere} of
$x$ as $S_k(x) = \{z : d(z,x) = k\}$ and the {\it $k$-ball} of $x$ as
$B_k(x) = \{z: d(z,x) \le k\}$. In particular, $S_1(x)$ is also
denoted as $N_x$, the set of all neighbors of $x$. Denote also
$N_{xy}=S_1(x)\cap S_1(y)$, the common neighbors of $x$ and
$y$. Especially when $d(x,y)=2$, the induced subgraph
${\rm Ind}_G (N_{xy})$ is called $\mu$-graph of $x$ and $y$. In case
that $\mu$-graphs are isomorphic to a graph $H$ for all $x,y$ with
$d(x,y)=2$, we call the graph $H$ the {\it $\mu$-graphs} of $G$.

For a vertex $x\in V$, let $\#_{\Delta}(x)$ denote the number of
triangles containing $x$. Similarly, for an edge $e = \{x,y\} \in E$,
let $\#_{\Delta}(x,y)$ denote the number of triangles containing
$e$. We have the following relation:
\begin{equation} \label{eq:triangles}
\#_{\Delta}(x) = \frac{1}{2} \sum_{y: y \sim x} \#_{\Delta}(x,y). 
\end{equation}

For $x,y\in V$, we also define
\begin{align*}
d_{x}^{-}(y) &= |\{z\sim y : d(x,y) = d(x,z) + 1\}|,\\
d_{x}^{0}(y) &= |\{z\sim y : d(x,y) = d(x,z) \}|,\\
d_{x}^{+}(y) &= |\{z\sim y : d(x,y) = d(x,z) - 1\}|,
\end{align*}
which we call the {\it in degree, spherical degree, out degree} of
$y$, respectively. The averages of these degrees are then defined as:
\begin{align} \label{eq:averdeg}
av_k^-(x) = \frac{1}{|S_k(x)|}\sum_{y \in S_k(x)}d_{x}^{-}(y),\nonumber\\ 
av_k^0(x) = \frac{1}{|S_k(x)|}\sum_{y \in S_k(x)}d_{x}^{0}(y),\\ av_k^+(x) = \frac{1}{|S_k(x)|}\sum_{y \in S_k(x)}d_{x}^{+}(y). \nonumber
\end{align}

By abuse of notation, we sometimes identify $S_k(x)$ and $B_k(x)$ with the induced subgraphs ${\rm Ind}_G(S_k(x))$ and ${\rm Ind}_G(B_k(x))$,
respectively.

\subsection{Ollivier Ricci curvature and Kantorovich duality}
\label{sec:OllivKant}

A fundamental concept in Optimal Transport Theory is the Wasserstein
distance defined on probability measures.

\begin{definition}
  Let $G = (V,E)$ be a graph. Let $\mu_{1},\mu_{2}$ be two probability
  measures on $V$. The \emph{Wasserstein distance}
  $W_1(\mu_{1},\mu_{2})$ between $\mu_{1}$ and $\mu_{2}$ is defined as
  \begin{equation} \label{eq:W1def} 
    W_1(\mu_{1},\mu_{2}):=\inf_{\pi \in \Pi(\mu_1,\mu_2)} {\rm cost}(\pi),
  \end{equation}
  where $\pi$ runs over all transport plans in
  \[
    \Pi(\mu_1,\mu_2) = \left\{ \pi: V \times V \to [0,1] :
      \mu_{1}(x)=\sum_{y\in V}\pi(x,y), \; \mu_{2}(y)=\sum_{x\in
        V}\pi(x,y) \right\}.
  \]
  and the \emph{cost} of $\pi$ is defined as
  $$ {\rm cost}(\pi) := \sum_{x \in V} \sum_{y \in V} d(x,y)
  \pi(x,y). $$
\end{definition}

The transportation plan $\pi$ in the above definition moves a mass
distribution given by $\mu_1$ into a mass distribution given by
$\mu_2$, and $W_1(\mu_1,\mu_2)$ is a measure for the minimal effort
which is required for such a transition. If $\pi$ attains the infimum
in \eqref{eq:W1def} we call it an {\it optimal transport plan}
transporting $\mu_{1}$ to $\mu_{2}$. We define the following
probability measures $\mu_x^p$ for any $x\in V,\: p\in[0,1]$:
$$ \mu_x^p(z) = \begin{cases} p & \text{if $z = x$,} \\
  \frac{1-p}{d_x} & \text{if $z\sim x$,} \\
  0 & \mbox{otherwise.} \end{cases} $$ 
In this paper we will pay particular attention to $D$-regular graphs
and the idleness parameter $p = \frac{1}{D+1}$ and write $\mu_{x}$ for
$\mu^{\frac{1}{D+1}}_{x}$, for simplicity.

\begin{definition}[\cite{Ol09}]
  Let $p\in [0,1]$. The \emph{$p$-Ollivier Ricci curvature} between two
  different vertices $x,y \in V$ is
  $$ \kappa_{ p}(x,y) = 1 - \frac{W_1(\mu^{ p}_x,\mu^{ p}_y)}{d(x,y)}, $$
  where $p$ is called the {\it idleness}.

  If $G$ is $D$-regular then we define the curvature, $\kappa$, as
  \begin{equation} \label{kpklly} 
    \kappa(x,y) = \frac{D+1}{D}\kappa_{\frac{1}{D+1}}(x,y).
  \end{equation}
\end{definition}

\begin{rem}
  The motivation for \eqref{kpklly} is that $\kappa(x,y)$ agrees with
  the modified curvature $\kappa_{LLY}(x,y)$ introduced by Lin/Lu/Yau
  in \cite{LLY11} for neighbours $x \sim y$, due to \cite[Theorem
  1.1]{Idle}. That is, we have for neighbours $x \sim y$ in
  $D$-regular graphs,
  \begin{equation} \label{eq:kappaLLY} \kappa(x,y) =
    \frac{\kappa_{\frac{1}{D+1}}(x,y)}{1-\frac{1}{D+1}} = \lim_{p \to 1}
    \frac{\kappa_p(x,y)}{1-p} =: \kappa_{LLY}(x,y).
  \end{equation}
\end{rem}

A straightforward consequence of the fact that Ollivier Ricci curvature
is defined via a distance function (Wasserstein distance) is the following:
\begin{equation} \label{eq:infeqinf}
\inf_{x \sim y} \kappa(x,y) \le \inf_{z \neq w} \kappa(z,w), 
\end{equation}
namely that it makes no difference to take the infimum of the curvature
over all pairs of different vertices or only over neighbours. Moreover,
the Discrete Bonnet-Myers Theorem \ref{thm:DBM} follows directly from \eqref{eq:infeqinf}
and the following stronger inequality for any pair of different vertices $z,w \in V$ (see \cite{Ol09}):
\begin{equation} \label{eq:kapzw_ineq}
\kappa(z,w) \le \frac{2}{d(z,w)},
\end{equation}
by choosing a pair of vertices $z,w \in V$ of maximal distance, that
is, $d(z,w) = {\rm diam}(G)$.

Another fundamental concept in Optimal Transport Theory is Kantorovich
duality. First we recall the notion of 1--Lipschitz functions and then
state the Kantorovich Duality Theorem.

\begin{definition}
  Let $G=(V,E)$ be a graph and $\phi: V \rightarrow \IR$. We say that
  $\phi$ is \emph{$1$-Lipschitz} if
  $$ |\phi(x) - \phi(y)| \leq d(x,y) \quad \text{for all $x,y\in V$.} $$
  We denote the set of all $1$--Lipschitz functions by
  $\textrm{\rm{1}--{\rm Lip}}(V)$.
\end{definition}

For an arbitrary graph $G = (V,E)$, we denote by $\ell_\infty(V)$ the
space of all bounded function $f: V \to \IR$ and by $C_c(V)$ the
subspace of all functions with finite support.

\begin{theorem}[Kantorovich duality \cite{Vill03}] \label{Kantorovich}
  Let $G = (V,E)$ be a graph. Let $\mu_{1},\mu_{2}$ be
  two probability measures on $V$. Then
  $$ W_1(\mu_{1},\mu_{2}) = 
  \sup_{\phi \in \textrm{\rm{1}--{\rm Lip}}(V)\, \cap\, \ell_\infty(V)}   
  \sum_{x\in V}\phi(x)(\mu_{1}(x)-\mu_{2}(x)).
  $$
  If $\phi$ attains the supremum we
  call it an \emph{optimal Kantorovich potential} transporting
  $\mu_{1}$ to $\mu_{2}$.
\end{theorem}

The following fact from \cite[Theorem 2.1]{MW17} is at the heart of
all the results presented in Section \ref{sec:genfacts}. Moreover, it
provides an alternative definition of Ollivier Ricci curvature as
a notion induced by a given Laplace operator. This alternative
viewpoint allows a more general definition of Ollivier Ricci curvature
for various kinds of Laplacians.

\begin{theorem}[see \cite{MW17}]
  Let $G=(V,E)$ be a connected graph and $f: V \to \IR$ be a function
  on the vertices.  We define the \emph{gradient} of $f$ as
  $$ \nabla_{xy} f = \frac{f(x)-f(y)}{d(x,y)} \quad 
  \text{for all $x,y \in V$, $x \neq y$}. $$ 
  Let $\Delta_G$ be the normalized Laplacian of $G$ defined in
  \eqref{eq:Deltanorm}. Then, for any pair $x,y \in V$ of different
  vertices, we have
  \begin{equation} \label{eq:kappally_alt} 
  \kappa_{LLY}(x,y) = 
  \inf_{\substack{\phi \in \textrm{\rm{1}--{\rm Lip}}(V)\, \cap\, C_c(V)\\ \nabla_{yx}\phi = 1}} 
  \nabla_{xy} (\Delta_G \phi), 
  \end{equation}
  where $\kappa_{LLY}$ was defined in \eqref{eq:kappaLLY}.
\end{theorem}

\subsection{Transport plans based on triangles and a perfect matching}
\label{sec:transportplanstrianglesmatching}

Recall the definition of the combinatorial property $\Lambda(m)$ given
in Definition \ref{def:Lambda}. The following proposition explains its
curvature implication and is useful for our curvature calculations in
Section \ref{sec:examples}. As mentioned before, it also implies
Proposition \ref{prop:curvcalc} directly by choosing
$m = \frac{2D}{L} - 2$.  Besides presenting this proposition and its
proof, we also introduce in this subsection the notion of a
\emph{transport plan based on triangles and a perfect matching}.


\begin{proposition} \label{prop:curvcalc0}
  Let $G=(V,E)$ be a $D$-regular graph and $e = \{x,y\} \in E$ an edge.
  Then we have
  \begin{equation} \label{eq:kappa_est}
  \kappa(x,y) \le \frac{2+\#_\Delta(x,y)}{D} 
  \end{equation}
  Moreover, the following are equivalent:
  \begin{itemize}
  \item[(a)] $\kappa(x,y) = \frac{2+\#_\Delta(x,y)}{D}$;
  \item[(b)] there is a perfect matching between the vertex sets
    $N_x \backslash (N_{xy} \cup \{y\})$ and
    $N_y \backslash (N_{xy} \cup \{x\})$.
  \end{itemize}
\end{proposition}

\begin{proof}
  Assume that $G=(V,E)$ is $D$-regular and $e= \{x,y\} \in E$ is
  contained in $m = \#_\Delta(x,y)$ triangles. Then we have
  $$ s = |N_{x}\setminus (N_{xy}\cup\{y\})| = 
  |N_{y}\setminus (N_{xy}\cup\{x\})| = D-1-m. $$ Since all masses are
  equal to $\frac{1}{D+1}$, we are faced with a Monge Problem and
  there exists an optimal transport plan $\pi$ transporting $\mu_x$ to
  $\mu_y$ with $\pi(z,w) \in \{0, \frac{1}{D+1}\}$ for all
  $z,w \in V$. (That is, $\pi$ is induced by a bijective optimal
  transport map $T: B_1(x) \to B_1(y)$; for more details see
  Subsection \ref{sec:concat_tramap}.) By \cite[Lemma 4.1]{Idle}, we
  can choose $\pi$ to satisfy $\pi(z,z) = \frac{1}{D+1}$ for all
  $z \in N_{xy} \cup \{x,y\}$, that is there is no mass transport on
  these vertices, and enumerate the vertices in
  $N_{x}\setminus (N_{xy}\cup\{y\})$ by $\{x_{i}\}_{i=1}^s$ and in
  $N_{y}\setminus (N_{xy}\cup\{x\})$ by $\{y_{j}\}_{j=1}^s$ such that
  $\pi(x_i,y_j) = \frac{1}{D+1} \delta_{ij}$.  Explicitly, we have
  \begin{equation} \label{eq:traplan}
  \pi(u,v) = \begin{cases} \frac{1}{D+1}, & \text{if $(u,v) = (x_i,y_i)$,} \\
  \frac{1}{D+1}, & \text{if $u=v \in N_{xy} \cup \{x,y\}$,} \\
  0, & \text{otherwise.} \end{cases}
  \end{equation}
  Then
  \begin{equation} \label{eq:W1calc}
  W_{1}(\mu_{x},\mu_{y}) = \sum_{u\in V} \sum_{v\in V} \pi(u,v)d(u,v) \ge 
  \frac{s}{D+1} = \frac{D-1-m}{D+1},
  \end{equation}
  with equality iff $d(x_i,y_i) = 1$, that is, there exists a perfect
  matching between the vertex sets $N_{x}\setminus (N_{xy}\cup\{y\})$
  and $N_{y}\setminus (N_{xy}\cup\{x\})$. Consequently, we have the
  following curvature estimate with the same matching property in case
  of equality:
  \begin{equation} \label{eq:kappa_calc} \kappa(x,y) = \frac{D+1}{D} \left(1 -
      W_1(\mu_x,\mu_y)\right) \le \frac{2+m}{D}.
  \end{equation}
\end{proof}

The transport plan $\pi$ chosen in the proof is of a particular
structure which will also be important later. Therefore we introduce
the following definition:

\begin{definition} \label{def:plan_tpm}
  Let $e = \{x,y\} \in E$ be an edge of a $D$-regular graph $G=(V,E)$ with
  the following property: There is a perfect matching between the sets
  $N_{x}\setminus (N_{xy}\cup\{y\})$ and
  $N_{y}\setminus (N_{xy}\cup\{x\})$ given by $x_i \sim y_i$, where
  $x_i$ and $y_i$ are defined as in the above proof. We say that the
  transport plan $\pi$ defined by \eqref{eq:traplan} is \emph{based on
    triangles and a perfect matching}.
\end{definition}

\section{Cartesian products}
\label{sec:cart}

In this section we show under what conditions Bonnet-Myers sharpness is preserved under taking Cartesian products. First we recall the following result of Lin, Lu and Yau from \cite{LLY11}.
\begin{theorem}[\cite{LLY11}]\label{LLYcartprods}
	Let  $G=(V_{G},E_{G})$ be a $d_{G}$-regular graph and $H=(V_{H},E_{H})$ be a $d_{H}$-regular graph. Let $x_{1},x_{2}\in V_{G}$ with $x_{1}\sim x_{2}$ and $y_{1},y_{2}\in V_{H}$ with $y_{1}\sim y_{2}$. Then
	\begin{align*}
	\kappa^{G\times H}((x_{1},y_{1}),(x_{2},y_{1})) & = \frac{d_{G}}{d_{G}+d_{H}} \kappa^{G}(x_{1},x_{2}),
	\\
	\kappa^{G\times H}((x_{1},y_{1}),(x_{1},y_{2})) & = \frac{d_{H}}{d_{G}+d_{H}} \kappa^{H}(y_{1},y_{2}),
	\end{align*}
\end{theorem}

\begin{theorem} \label{thm:cartprod}
	Let $\{G_{i} = (V_{i}, E_{i})\}_{i=1}^{N}$ be a family of regular graphs where $G_{i}$ has valency $D_{i}$ for each $i.$ Let $L_{i}$ be the diameter of $G_{i}.$ Let $G = G_{1}\times\cdots\times G_{N}.$ The following are equivalent:
	\begin{enumerate}[(i)]
		\item
		$G$ is Bonnet-Myers sharp.
		\item
		Each $G_{i}$ is Bonnet-Myers sharp and $\frac{D_{1}}{L_{1}}=\cdots=\frac{D_{N}}{L_{N}}.$  
	\end{enumerate} 
\end{theorem}
\begin{proof}
	Since $G_{1}\times\cdots\times G_{N} = G_{1}\times(G_{2}\times\cdots\times G_{N})$ we may assume that $N=2$ and use induction.
	
	By Theorem \ref{LLYcartprods},
	\begin{align} \label{eq:cart_expand}
	\inf_{\substack{u,v \in V(G) \\ u\sim v}}\kappa^{G}(u,v) 
	& = \min\left\{\inf_{\substack{x_{1},x_{2}\in V_{1}\\x_{1}\sim x_{2}\\ y \in V_{2}}} \kappa^{G}((x_{1},y),(x_{2},y)), \inf_{\substack{y_{1},y_{2}\in V_{2}\\y_{1}\sim y_{2}\\ x \in V_{1}}} \kappa^{G}((x,y_{1}),(x,y_{2}))\right\}
	\nonumber \\
	& = \min\left\{\inf_{\substack{x_{1},x_{2}\in V_{1}\\x_{1}\sim x_{2}}} \frac{D_{1}}{D_{1}+D_{2}}\kappa^{G_{1}}(x_{1},x_{2}), \inf_{\substack{y_{1},y_{2}\in V_{2}\\y_{1}\sim y_{2}}} \frac{D_{2}}{D_{1}+D_{2}}\kappa^{G_{2}}(y_{1},y_{2})\right\}.
	\end{align}
	
	First we prove that (i) implies (ii). Since $G$ is Bonnet-Myers sharp, we have, by Bonnet-Myers theorem applied on $G_1$: $$\frac{2}{L_1+L_2}=\inf_{\substack{u,v \in V(G) \\ u\sim v}}\kappa^{G}(u,v) 
	\le \inf_{\substack{x_{1},x_{2}\in V_{1}\\x_{1}\sim x_{2}}} \frac{D_{1}}{D_{1}+D_{2}}\kappa^{G_{1}}(x_{1},x_{2}) \le \frac{D_1}{D_1+D_2}\cdot\frac{2}{L_1},$$
	which is equivalent to $\frac{D_2}{L_2} \le \frac{D_1}{L_1}$.
	
	On the other hand, Bonnet-Myers on $G_2$ gives
	$$\frac{2}{L_1+L_2}=\inf_{\substack{u,v \in V(G) \\ u\sim v}}\kappa^{G}(u,v) 
	\le \inf_{\substack{y_{1},y_{2}\in V_{2}\\y_{1}\sim y_{2}}} \frac{D_{2}}{D_{1}+D_{2}}\kappa^{G_{2}}(y_{1},y_{2}) \le \frac{D_2}{D_1+D_2}\cdot\frac{2}{L_2},$$
	which is equivalent to $\frac{D_1}{L_1} \le \frac{D_2}{L_2}$.
	
	Therefore, we can conclude that $\frac{D_1}{L_1} = \frac{D_2}{L_2}$, and all the inequalities above are sharp, that is $G_1$ and $G_2$ are Bonnet-Myers sharp as well.
	\\
	\\
	To prove (ii) implies (i): we simply plug into \eqref{eq:cart_expand}
	$$ \inf_{\substack{x_{1},x_{2}\in V_{1}\\x_{1}\sim x_{2}}} \kappa^{G_{1}}(x_{1},x_{2})=\frac{2}{L_1} \quad \textup{and} \quad \inf_{\substack{y_{1},y_{2}\in V_{2}\\y_{1}\sim y_{2}}} \kappa^{G_{2}}(y_{1},y_{2})=\frac{2}{L_2} $$
	and use the assumption that $\frac{D_1}{L_1}=\frac{D_2}{L_2}$. As a result, we obtain $\inf_{\substack{u,v \in V(G) \\ u\sim v}}\kappa^{G}(u,v)=\frac{2}{L_1+L_2}$.
\end{proof}

\begin{rem} \label{rem:lichcartsharp}
	In contrast to the necessary and sufficient condition for Bonnet-Myers sharpness in Theorem \ref{thm:cartprod}, much less is required for the Cartesian product $G=G_1\times G_2$ to be Lichnerowicz sharp. In fact, $G$ is Lichnerowicz sharp already if $G_1$ is Lichenerowicz sharp and $G_2$ is an arbitrary graph with its curvature lower bound large enough, as explained in the following argument.

	Let $\lambda_{1}^{G_{1}}, \lambda_{1}^{G_{2}}, \lambda_{1}^{G}$ be the smallest positive eigenvalues of the Laplacians on $G_{1},G_{2},$ $G$. We have
	
	\begin{align} \label{eqn:lich_cart}
	\inf_{\substack{u,v \in V(G) \\ u\sim v}}\kappa^{G}(u,v) & = \min\left\{\inf_{\substack{x_{1},x_{2}\in V_{1}\\x_{1}\sim x_{2}\\ y \in V_{2}}} \kappa^{G}((x_{1},y),(x_{2},y)), \inf_{\substack{y_{1},y_{2}\in V_{2}\\y_{1}\sim y_{2}\\ x \in V_{1}}} \kappa^{G}((x,y_{1}),(x,y_{2}))\right\}
	\nonumber\\
	& = \min\left\{\inf_{\substack{x_{1},x_{2}\in V_{1}\\x_{1}\sim x_{2}}} \frac{D_{1}}{D_{1}+D_{2}}\kappa^{G_{1}}(x_{1},x_{2}), \inf_{\substack{y_{1},y_{2}\in V_{2}\\y_{1}\sim y_{2}}} \frac{D_{2}}{D_{1}+D_{2}}\kappa^{G_{2}}(y_{1},y_{2})\right\}
	\nonumber \\
	& \le \min\left\{\frac{D_{1}}{D_{1}+D_{2}}\lambda_{1}^{G_{1}}, \frac{D_{2}}{D_{1}+D_{2}}\lambda_{2}^{G_{2}}\right\} = \lambda_{1}^{G}.
	\end{align}
where the inequality comes from Lichnerowicz' Theorem on each graph $G_i$: $ \inf \kappa^{G_i} \le \lambda_1^{G_i} $. In order to obtain the equality in \eqref{eqn:lich_cart}, a sufficient condition is $ \inf \kappa^{G_1} = \lambda_1^{G_1} $ (i.e. $G_1$ is Lichnerowicz sharp) and $\frac{D_1}{D_2}\inf \kappa^{G_1} \le \inf \kappa^{G_2}$.
\end{rem}

\section{Examples of Bonnet-Myers sharp graphs}
\label{sec:examples}

Here we present various examples of Bonnet-Myers sharp graphs and
study their properties. Interestingly, the $\mu$-graphs in each of the following examples are cocktail party graphs and the $1$-spheres are
strongly regular.

Note that a finite simple graph $G=(V,E)$ is called \emph{strongly regular}
with parameters $(\nu,k,\lambda,\mu)$ if $G$ is not a complete graph
and the following holds true:
\begin{itemize}
\item $V$ has cardinality $\nu$,
\item every vertex has degree $k$,
\item each pair of adjacent vertices has precisely
  $\lambda$ common neighbours,
\item each pair of non-adjacent vertices has precisely
  $\mu$ common neighbours.
\end{itemize}
(In contrast to the usual definition, we also consider a set of $n$
isolated points to be a strongly regular graph with parameters
$(\nu,k,\lambda,\mu) = (n,0,*,0)$ where $*$ can be any integer.) We
say that a strongly regular graph $G$ with these parameters is
${\rm{srg}}(\nu,k,\lambda,\mu)$.

\subsection{Hypercubes $Q^n$, $n \ge 1$} \label{subsec:ex_hypcubes}

The hypercube $Q^n$ can be viewed as the graph whose vertices are
elements of $\{0,1\}^n$, and two vertices $x,y \in \{0,1\}^n$ are
adjacent if and only if their Hamming distance is one. 
In \cite{LLY11} the authors showed that the hypercube $Q^{n}$ has
constant curvature $\frac{2}{n}$. Since $Q^{n}$ has diameter
$n$, it follows that it is Bonnet-Myers sharp.

It is obvious that every $\mu$-graph of $Q^n$ consists of two isolated
points and that every $1$-sphere or $Q^n$ consists of $n$ isolated
points. Moreover, hypercubes are self-centered and the antipole of the
vertex $(x_1,\dots,x_n) \in \{0,1\}^n$ is given by
$(1-x_1,\dots,1-x_n)$. 

We will show in Section \ref{sec:BMextrdiam} that the hypercubes are
the only Bonnet-Myers sharps graph where their valency is equal to
their diameter.

\subsection{Cocktail party graphs $CP(n)$, $n \ge
  3$} \label{subsec:ex_cocktailparty}

The cocktail party graph $CP(n)$ is defined to have vertex set
$\{u_{1},\ldots,u_{n},v_{1},\ldots,v_{n}\}$ where all pairs of
vertices are adjacent unless they share the same subscript. Note that
$CP(n)$ has diameter $L=2$ and is regular with valency $D=2n-2$. $CP(n)$
is self-centered and $u_i$ and $v_i$ are antipoles of each other.

\begin{lemma}
  Let $x\sim y$ be an edge in the cocktail party graph $CP(n)$. Then
  $\kappa(x,y) = 1$, which implies that $CP(n)$ is
  Bonnet-Myers sharp.
\end{lemma}

\begin{proof}
  Since $CP(n)$ is edge-transitive it suffices to show that
  $\kappa(u_{1},u_{2}) = 1$, which is equivalent to showing that
  $\kappa_{\frac{1}{2n-1}}(u_{1},u_{2}) = \frac{2n-2}{2n-1}$. Note
  that $\mu_{u_{1}}(v_{1}) = 0$ and $\mu_{u_{1}}$ equals
  $\frac{1}{2n-1}$ otherwise. Likewise $\mu_{u_{2}}(v_{2}) = 0$ and
  $\mu_{2}$ equals $\frac{1}{2n-1}$ otherwise. Therefore the only mass
  that must be transported is a mass of size $\frac{1}{2n-1}$ from
  $v_{2}$ to $v_{1}$ over a distance of 1. Therefore
  $$ \kappa_{\frac{1}{2n-1}}(u_{1},u_{2}) = 1 - W_{1}(\mu_{u_{1}},\mu_{u_{2}}) 
  = 1 - \frac{1}{2n-1} = \frac{2n-2}{2n-1}, $$
  as required.
\end{proof}

It is easily checked that every $\mu$-graph of $CP(n)$ as well as any
induced $1$-sphere is isomorphic to $CP(n-1)$ and therefore
${\rm{srg}}(2n-2,2n-4,2n-6,2n-4)$.

\subsection{Johnson graphs $J(2n,n)$, $n \ge 3$} \label{subsec:ex_johnson}

The Johnson graphs are a family of graphs that can be seen as a
generalisation of the complete graphs. See \cite{CLP2018} where their
Bakry-\'Emery curavture is calculated and compared to the curvature of
the complete graphs.

The vertices of the Johnson graph $J(n,k)$ are all the subsets of
$\{1,\ldots,n\}$ with $1 \le k \le n-1$ elements. Two vertices $u$ and $v$ are
connected by an edge if $|u\cap v| = k-1$. Observe that $J(n,1)$ is
isomorphic to the complete graph $K_n$ on $n$ vertices.

The Johnson graph $J(n,k)$ is $D=k(n-k)$-regular and has diameter
$L=\min\{k,(n-k)\}.$ The smallest non-zero eigenvalue $\lambda_{1}$ of
the Laplacian on $J(n,k)$ is $\frac{n}{k(n-k)}$.

\begin{lemma}
  Let $n,k\in\mathbb{N}$, $1 \leq k \leq n-1$. Let $x,y$ be to
  adjacent vertices in $J(n,k)$. Then
  $$ \kappa(x,y) = \frac{n}{k(n-k)}. $$
  Therefore $J(n,k)$ is Lichnerowicz sharp. Furthermore, $J(2n,n)$ is
  Bonnet-Myers sharp.
\end{lemma}

\begin{proof}
  Without loss of generality, due to edge-transitivity, we may take
  $x = \{1,\ldots, k\}$ and $y = \{2,\ldots,k+1\}.$ Observe that
  $$
  N_{xy} = \left\{\{2,\ldots,k\}\cup\{i\}:i\in\{k+2,\ldots,n\} \right\} 
  \, \cup \, \left\{ \{1,\ldots,k+1\}\setminus\{i\}: i \in
    \{2,\ldots,k\} \right\},
  $$
  and
  \begin{eqnarray*}
  N_{x} \setminus \left(N_{xy} \cup\{y\}\right) &=& 
  \{(\{1,\ldots,k\}\setminus\{i\})\cup\{j\}:i\in\{2,\ldots k\},j\in\{k+2,\ldots,n\}\}, 
  \\
  N_{y} \setminus \left(N_{xy}\cup\{x\}\right) &=& 
 \{(\{2,\ldots,k+1\}\setminus\{i\})\cup\{j\}:i\in\{2,\ldots k\},j\in\{k+2,\ldots,n\}\}.
  \end{eqnarray*}
  Note that $|N_{xy}| = n-2$ and there is an obvious perfect matching
  between $N_{x} \setminus (N_{xy}\cup \{y\})$ and
  $N_{y} \setminus (N_{xy}\cup\{x\})$. Therefore, by Proposition
  \ref{prop:curvcalc0}, we have
  $$\kappa(x,y) = \frac{n}{k(n-k)}.$$
  In the particular case $J(2n,n)$, we obtain
  $$ 
  \kappa(x,y) = \frac{2n}{n(2n-n)} = \frac{2}{n} = \frac{2}{\min\{n,2n-n\}}, 
  $$
  showing that $J(2n,n)$ is Bonnet-Myers sharp.
\end{proof}

The graphs $J(2n,n)$ are self-centered and the antipole of the vertex
$A \subset \{1,\dots,2n\}$ is given by $\{1,\dots,2n\} \backslash A$.
Let us also investigate the structures of the $\mu$-graphs and
$1$-spheres of $J(2n,n)$. Since Johnson graphs are
distance-transitive, it suffices to consider the $\mu$-graph of
$x = \{1,\dots,n\}$ and $z = \{3,\dots,n+2\}$. Then we have
$$ N_{xz} = \{ \{1,3,\dots,n,n+1\}, \{1,3,\dots,n,n+2\},
\{2,3,\dots,n,n+1\}, \{2,3,\dots,n,n+2\} \}, $$ and the induced
subgraph is a quadrangle, that is the $\mu$-graph of $x$ and $z$ is
isomorphic to $CP(2)$. Finally, let us consider
$$ S_1(x) = \{ y_{ij}:=(\{1,\dots,n\} \backslash \{i\}) \cup \{j\}: i \in 
\{1,2,\dots,n\}, j \in \{n+1,n+2,\dots,2n\} \}. $$ There is a natural
identification of the vertices in $S_1(x)$ with the elements in the
set $\{1,2,\dots,n\} \times \{1,2,\dots,n\}$ via
$y_{ij} \mapsto (i,j-n)$. Note that in the induced subgraph $S_1(x)$
we have $y_{ij} \sim y_{kl}$ if and only if ($i=k$ and $j \neq l$) or
($i \neq k$ and $j=l$). This shows that the induced subgraph $S_1(x)$
is isomorphic to the Cartesian product $K_n \times K_n$, where the
vertex set of $K_n$ is identified with the set $\{1,2,\dots,n\}$. Note
that $K_n \times K_n$ is strongly regular with parameters
$(n^2,n,n-2,2)$.

\subsection{Demi-cubes $Q^{2n}_{(2)}$, $n \ge 3$} \label{subsec:ex_demicubes}

The halved cube graph, $\frac{1}{2}Q^{n}$, is the distance two
sub-graph of the hypercube $Q^{n}$, that is the graph with the
vertices of $Q^{n}$ formed by connecting any pair of vertices at
distance exactly two in the hypercube $Q^n$. The halved cube has two
isomorphic connected components and we shall denote either of these
components by $Q^{n}_{(2)}.$ The graph $Q^{n}_{(2)}$ is known as the
$n$-dimensional demi-cube.

Note that $Q^n_{(2)}$ is a regular graph of valency $D={n \choose
  2}=\frac{n(n-1)}{2}$ and diameter $L=\lfloor \frac{n}{2} \rfloor$.
Note also that the adjacency matrices of $Q^n$, $Q^n_{(2)}$ and
$\frac{1}{2}Q^n$ are related by
$$ A_{\frac{1}{2}Q^n} = \begin{pmatrix} A_{Q^n_{(2)}} & 0 \\ 0 & A_{Q^n_{(2)}} \end{pmatrix} = \frac{1}{2} \left( A_{Q^n}^2 - n {\rm{Id}} \right). $$
Since the eigenvalues of $Q^n$ are given by $n-2i$, $i=0,\dots,n$, the
second largest eigenvalue of $A_{Q^n_{(2)}}$ is
$$ \theta_1 = \frac{1}{2}\left( (n-2)^2-n \right) = \frac{1}{2}(n-4)(n-1) $$
and the smallest non-zero eigenvalue $\lambda_1$ of the Laplacian 
$\frac{1}{D} A_{Q^n_{(2)}} - {\rm{Id}}$ is
$$ \lambda_1 = 1 - \frac{1}{D}\theta_1 = \frac{4}{n}. $$

\begin{lemma}
Let $n\geq 2$ and $x,y$ be two adjacent vertices in $Q^n_{(2)}$. Then
$$\kappa(x,y) =  \frac{4}{n}.$$
Therefore $Q^n_{(2)}$ is Lichnerowicz sharp and $Q^{2n}_{(2)}$ is
Bonnet-Myers sharp.
\end{lemma}

\begin{proof}
  We may view the vertices of $Q^n_{(2)}$ as elements of
  $\{0,1\}^n$ that contain an even number of ones.  Two vertices
  are connected by an edge if their Hamming distance is equal to
  $2$. Let $e_i \in \{0,1\}^n$ be the $i$-th standard vector with
  precisely one non-zero entry at position $i$. Without loss of
  generality we may take $x = (0,\dots,0)$ and $y = e_1+e_2$.

  Then the common neighbours of $x$ and $y$ are given by $e_1+e_j$ and
  $e_2+e_j$ with $3 \le j \le n$, that is, we have
  $|N_{xy}| = 2(n-2)$.  Moreover, the vertices in
  $N_{x} \setminus (N_{xy}\cup \{y\})$ are given by
  $x_{ij}=e_i+e_j$, $3 \le i < j \le n$ and, similarly, the vertices
  in $N_{y} \setminus (N_{xy}\cup \{x\})$ are given by
  $y_{ij}=e_1+e_2+e_i+e_j$, $3 \le i < j \le n$. Obviously, the pairing
  $x_{ij} \sim y_{ij}$ provides a perfect matching between these sets of
  vertices. Therefore, by Proposition \ref{prop:curvcalc0}, we have
  $$ \kappa(x,y) = \frac{4}{n} = \lambda_1.$$
  This shows that $Q^n_{(2)}$ is Lichnerowicz sharp. For the graph $Q^{2n}_{(2)}$
  we have 
  $$ \kappa(x,y) = \frac{4}{2n} = \frac{2}{n} = \frac{2}{L}, $$
  that is, $Q^{2n}_{(2)}$ is Bonnet-Myers sharp.
\end{proof}

To identify the $\mu$-graphs of $Q^{2n}_{(2)}$, we can choose without
loss of generality the distance two vertices $x= (0,\dots,0)$ and
$z=e_1+e_2+e_3+e_4$. Let $y_{ij} := e_i + e_j$ with
$1 \le i < j \le 2n$.  Then the common neighbours of $x$ and $z$ are
the $6$ vertices $y_{12}$, $y_{13}$, $y_{14}$, $y_{23}$, $y_{24}$, and
$y_{34}$. Moreover, the vertex $y_{ij}$ is adjacent to all others in
the $\mu$-graph of $x$ and $z$, except for the vertex $y_{kl}$ with
$\{k,l\} = \{1,2,3,4\} \backslash \{i,j\}$. This shows that the
$\mu$-graph of $x$ and $z$ is isomorphic to $CP(3)$.

Next, let us consider the induced $1$-sphere $S_1(x)$. Its vertices
are given by $y_{ij}$, and $y_{ij}$ and $y_{kl}$ are adjacent if and only if
$| \{i,j\} \cap \{k,l\} | = 1$. This shows that $S_1(x)$ is isomorphic
to the strongly regular graph $J(2n,2)$ with parameters
$(n(2n-1),2(2n-2),2n-2,4)$.

Finally, observe that $Q^{2n}_{(2)}$ is self-centered: Identifying the vertices
of $Q^{2n}_{(2)}$ with vectors in $\{0,1\}^{2n}$, the antipole of $(x_1,\dots,x_{2n}) \in \{0,1\}^{2n}$ is $(1-x_1,\dots,1-x_{2n})$.

\subsection{The Gosset graph} \label{subsec:ex_gosset}

The Gosset graph is a regular graph of valency $D=27$ with 56 vertices
and diameter $L=3$. Its adjaceny matrix can be found in
\begin{center}
\url{http://www.distanceregular.org/graphs/gosset.html}
\end{center}

The Gosset graph can be understood as follows: Take two copies of
$K_8$, say $G, G'$ both isomorphic to $K_8$. Denote the vertex sets of
$G$ and $G'$ by $\{1,2,\dots,8\}$. The edges in $G$ and $G'$ can be
identified with sets $\{i,j\}$, $1 \le i < j \le 8$. However, to
distinguish between edges in $G$ and edges in $G'$, we denote them by
sets $\{i,j\}$ and $\{i,j\}'$, respectively. There are
${8 \choose 2} = 28$ edges in each of the graphs $G$ and $G'$, and each
edge represents a vertex of the Gosset graph. Pairs of edges in the
same copy are neighbours (as vertices in the Gosset graph) if and only if  they
have a vertex in common, that is, $\{i,j\} \sim \{k,l\}$ if and only if
$| \{i,j\} \cap \{k,l\} | = 1$.  Pairs of edges $\{i,j\}$ and
$\{k,l\}'$ are neighbours (as vertices in the Gosset graph if and only if
$\{i,j\} \cap \{k,l\} = \emptyset$. With this explicit description,
we can prove the following:

\begin{lemma}
Let $x,y$ be two adjacent vertices of the Gosset graph. Then
$$ \kappa(x,y) =  \frac{2}{3}, $$
and the Gosset graph is Bonnet-Myers sharp.
\end{lemma}

\begin{proof}
  Since the Gosset graph is distance-transitive we only need to
  calculate the curvature of one edge, say,
  $e=\{ x=\{1,2\},y=\{2,3\} \}$. Note that $e$ lies in precisely $m=16$
  triangles: There are $6$ common neighbours of $x$ and $y$ in $G$,
  namely, $\{1,3\}$, $\{2,4\}$, $\{2,5\}$, $\{2,6\}$, $\{2,7\}$,
  $\{2,8\}$, and there are $10$ common neighbours of $x$ and $y$ in
  $G'$, namely, $\{4,5\}'$, $\{4,6\}'$, $\{4,7\}'$, $\{4,8\}'$,
  $\{5,6\}'$, $\{5,7\}'$, $\{5,8\}'$, $\{6,7\}'$, $\{6,8\}'$,
  $\{7,8\}'$. Moreover, there is a perfect matching between the $10$
  neighbours of $x$ not in $B_1(y)$ and the $10$ neighbours of $y$ not
  in $B_1(x)$: The $10$ neighbours of $x$ not in $B_1(y)$ are
  $$ z_1=\{1,4\}, z_2=\{1,5\}, z_3=\{1,6\}, z_4=\{1,7\}, z_5=\{1,8\},$$
  and
  $$ z_6=\{3,4\}', z_7=\{3,5\}', z_8=\{3,6\}', z_9=\{3,7\}', z_{10}=\{3,8\}'. $$
  Similarly, the $10$ neighbours of $y$ not in $B_1(x)$ are
  $$ w_1=\{3,4\}, w_2=\{3,5\}, w_3=\{3,6\}, w_4=\{3,7\}, w_5=\{3,8\},$$
  and
  $$ w_6=\{1,4\}', w_7=\{1,5\}', w_8=\{1,6\}', w_9=\{1,7\}', w_{10}=\{1,8\}', $$
  and we match $z_j \sim w_j$, $j=1,2,\dots,10$. Applying Proposition
  \ref{prop:curvcalc0} again, we conlude that
  $$ \kappa(x,y) = \frac{2+m}{D} = \frac{18}{27} = \frac{2}{L}, $$
  finishing the proof.
\end{proof}

It is known that the induced $1$-spheres of the Gosset graph are
isomorphic to the Schl\"afli graph which is ${\rm{srg}}(27,16,10,8)$.
Moreover, the Gosset graph is self-centered since the vertices $\{i,j\}$
and $\{i,j\}'$ are antipoles of each other. Let us, finally, identify
the $\mu$-graphs. Again, by distance-transitivity of the Gosset graph,
it suffices to consider the $\mu$-graph of $x=\{1,2\}$ and
$z=\{3,4\}$. The common neighbours of $x$ and $y$ in the Gosset graph
and corresponding to edges in $G$ are $\{1,3\}$, $\{1,4\}$, $\{2,3\}$
and $\{2,4\}$. The common neighbours of $x$ and $y$ corresponding to
edges in $G'$ are $\{5,6\}'$, $\{5,7\}'$, $\{5,8\}'$, $\{6,7\}'$,
$\{6,8\}'$, $\{7,8\}'$. Together, these represent $10$ vertices of the
Gosset graph and the vertex $\{i,j\} \subset \{1,\dots,4\}$ is
adjacent to each of them except to itself and to the vertex
$\{1,\dots,4\} \backslash \{i,j\}$. Similarly, the vertex $\{k,l\}'$
with $\{k,l\} \subset \{5,\dots,8\}$ is adjacent to each of these
vertices except to itself and to the vertex $\{i,j\}'$ with
$\{i,j\} = \{5,\dots,8\} \backslash \{k,l\}$. This shows that the
$\mu$-graph of $x$ and $z$ is isomorphic to $CP(5)$.

\subsection{Revisiting our examples}
\label{sec:revisex}

Let us first mention a few common properties of the
$(D,L)$-Bonnet-Myers sharp graphs presented in Subsections
\ref{subsec:ex_hypcubes}-\ref{subsec:ex_gosset}. They are all
\begin{itemize}
\item self-centered, 
\item irreducible (with the exception of the hypercubes $Q^n$,
$n \ge 2$), and
\item distance-regular.
\end{itemize} 
A $D$-regular connected finite graph $G=(V,E)$ of diameter $L$ is
called \emph{distance-regular} if there are integers $b_j, c_j$ such
that for any two vertices $x,y \in V$ with $d(x,y) = j$ we have
$d_x^-(y) = c_j$ and $d_x^+(y) = b_j$. The sequence
$$ (b_0,\dots,b_{L-1};c_1=1,\dots,c_L\} $$
is called the \emph{intersection array} of the distance regular graph
$G$. Moreover, $\mu$ denotes the number of common neighbours of a pair
of vertices at distance $2$, i.e., $\mu=c_2$. Distance regular graphs
can be also defined as those graphs $G = (V,E)$ with the property
that, for any choice of integers $k,l,m \ge 0$, the cardinality of
$B_k(x) \cap B_l(y)$ for $x,y \in V$ with $d(x,y) = m$ depends only on
the integers $k,l,m$.

Further properties of these examples are listed in Table
\ref{table:BMsharp_examples} below. We observe that all the above
examples have \emph{symmetric} intersection arrays, that is, we have
$$ c_j = b_{L-j} \quad \text{for $1 \le j \le L$}. $$ 
Moreover, we always have
\begin{equation} \label{eq:DL-bc} 
b_0 = D, \quad b_1 = D+1-\frac{2D}{L} \quad \text{and} \quad c_2 = 
\frac{2(D-L)}{L(L-1)} + 2,
\end{equation}
and all $\mu$-graphs are isomorphic to the cocktail party graph
$CP(c_2/2)$. Furthermore, all $1$-spheres $S_1(x)$ in these graphs are
strongly regular with parameters $(\nu,k,\lambda,\mu)$. We say that
our examples are \emph{locally} ${\rm{srg}}(\nu,k,\lambda,\mu)$.

The parameters $(\nu,k,\lambda,\mu)$ are already determined by the
size of the $\mu$-graphs via the following general result:

\begin{proposition} \label{prop:cocktail-implies-str}
  Let $G=(V,E)$ be a self-centered $(D,L)$-Bonnet-Myers sharp graph such
  that every $\mu$-graph is isomorphic to the cocktail party graph
  $CP(m)$ with 
  $$ m = \frac{D-L}{L(L-1)} + 1. $$
  Suppose that the $1$-sphere $S_1(x)$ in $G$ is strongly regular. Then $S_1(x)$
  is 
  $$ {\rm{srg}}\left( D, \frac{2D}{L}-2, \frac{D-1}{L-1}-3, 
  \frac{2(D-L)}{L(L-1)}
  \right). $$
  Moreover, the adjacency matrix of $S_1(x)$ has second largest eigenvalue
  equals $\frac{(D-L)(L-2)}{L(L-1)}$.
\end{proposition}

The following proof will refer to Theorems \ref{onespheredegree} and \ref{recursionformulas} which appear later in Section \ref{sec:genfacts} but they do not depend on the current section, and hence the following arguments are still applicable.

\begin{proof}
  Assume that the induced subgraph $S_{1}(x)$ is strongly regular with
  parameters $(\nu,k,\lambda,\mu)$. Clearly $\nu = D$. By Theorem
  \ref{onespheredegree} in the next section, we have $k = \frac{2D}{L}
  - 2$, since every vertex is a pole in the case of a self-centered
  graph.

  We now calculate $d_x^-(z)$ for any $z\in S_{2}(x).$ Since the
  $\mu$-graph of $x,z$ is isomorphic to $CP(m)$ we have
  $$ d_x^-(z) = 2m = \frac{2(D-L)}{L(L-1)} + 2 = av_2^-(x). $$
  Let $y\in S_{1}(x).$ By Theorem \ref{recursionformulas}, we have
  $$ d_x^+(y) = D\left( 1 - \frac{2}{L} \right) + d_x^-(y) =
  D\left( 1 - \frac{2}{L} \right) + 1 =
  \frac{(L-2)D}{L} + 1 = av_1^+(x). $$ 
  Thus, using $av_1^+(x)|S_{1}(x)| = av_2^-(x)|S_{2}(x)|,$ we obtain
  $$ |S_{2}(x)| = D(L-1)\frac{(L-2)D+L}{2(D+L(L-2))}.$$ 
  Let $\mathcal{T}^-$ be the set of all triangles containing the
  vertex $x$, $\mathcal{T}^0$ be the set of all triangles in $S_1(x)$
  and and $\mathcal{T}^+$ be the set of all triangles containing two
  vertices in $S_1(x)$ and one vertex in $S_2(x)$. For $z\in S_{2}(x)$
  we have that the number of triangles in $\mathcal{T}^+$ containing
  $z$ is equal to the number of edges in $CP(m),$ which is
  $2m(m-1)$. Thus
  \begin{equation} \label{eq:T+} 
  |\mathcal{T}^+| = |S_{2}(x)|\cdot 2m(m-1) = 
  \frac{D}{L(L-1)}((L-2)D+L)\left(\frac{D}{L}-1\right). 
  \end{equation}
  Let $E(S_{1}(x))$ be the set of edges in $S_{1}(x).$ Since every
  edge in $G$ is contained in precisely $\frac{2D}{L}-2$ triangles we
  have
  $$ |E(S_{1}(x))|\cdot \left(\frac{2D}{L} - 2\right) = |\mathcal{T}^-| + 
  3 |\mathcal{T}^0| + |\mathcal{T}^+|, $$ 
  since every triangle in $\mathcal{T}^- \cup \mathcal{T}^+$ shares
  precisely one edge with $S_1(x)$ and every triangle in $\mathcal{T}^0$
  shares all three edges with $S_1(x)$. Moreover, since $\lambda$ agrees
  with the number of triangles in ${\mathcal T}^0$ containing a fixed edge
  in $E(S_1(X))$, we have
  $$ |E(S_{1}(x))|\cdot \lambda = 3 |\mathcal{T}^0|. $$
  Thus
  \begin{equation} \label{eq:calc_lambda} 
  |E(S_{1}(x))|\cdot \lambda = |E(S_{1}(x))|\cdot 
  \left(\frac{2D}{L}-2\right)- |\mathcal{T}^-|-|\mathcal{T}^+|. 
  \end{equation}
  Note that $|\mathcal{T}^-|$ is equal to the number of edges in
  $S_1(x)$, that is
  $$ |\mathcal{T}^-| = |E(S_{1}(x))| = \frac{|S_1(x)|}{2} \cdot k =
  \frac{D}{2}\left(\frac{2D}{L}-2\right) =
  D\left(\frac{D}{L}-1\right). $$ 
  Plugging this into \eqref{eq:calc_lambda} and using \eqref{eq:T+}
  leads to
  $$ \lambda = \frac{D-1}{L-1} - 3. $$
  Finally, we compute $\mu$ by using
  $\mu = \frac{k(k-\lambda - 1)}{\nu-k-1}$ (see \cite[p. 116]{BH}). This gives
  $$\mu = 2\frac{D-L}{L(L-1)}.$$

  The second largest adjacency eigenvalue of the induced strongly
  regular subgraph $S_1(x)$ is given by (see \cite[Theorem 9.1.2]{BH})
  $$ 
  \frac{1}{2}\left( \lambda - \mu + \sqrt{(\lambda-\mu)^2+4(k-\mu)} \right). 
  $$
  It is straightforward to check that
  $$ (\lambda-\mu)^2 + 4(k-\mu) = 
  \left( \frac{L^2 + D(L-2)}{L(L-1)} \right)^2, $$ 
  which implies 
  $$ \frac{1}{2}\left( \lambda - \mu + \sqrt{(\lambda-\mu)^2+4(k-\mu)} \right)
  = \frac{(D-L)(L-2)}{L(L-1)}. $$
\end{proof}

Applying Proposition \ref{prop:cocktail-implies-str} in our examples
and using the relation \eqref{eq:DL-bc} between the parameters $(D,L)$
and $(b_0,b_1,c_2)$, we conclude that all our distance-regular
examples of Bonnet-Myers sharp graphs are locally strongly regular
with parameters
$$ {\rm{srg}}\left(b_0,b_0-b_1-1,\frac{b_0(c_2-2)}{b_0-b_1+1}-3,c_2-2 \right). $$

We know from Theorem \ref{thm:lichn} that every Bonnet-Myers sharp graph is
also Lichnerowicz sharp. Table \ref{table:BMsharp_examples} contains
also information about the multiplicity $\dim E_{\lambda_1}$ of the
Lichnerowicz eigenvalue $\lambda_1 = \kappa(x,y) = \frac{2}{L}$ for
all $x,y \in V$, $x \neq y$. Note that Lichnerowicz sharpness means that the second largest eigenvalue $\theta_1$ of the adjacency matrix agrees with $b_1 -1 = D - \frac{2D}{L}$ and that there is a classification of all distance-regular graphs with second largest adjacency matrix eigenvalue $\theta_1$ equal to $b_1-1$ (see \cite[Theorem 4.4.11]{BCN89}). For the reader's convenience, we provide here the statement of this theorem.

\begin{theorem}[\cite{BCN89}] \label{thm:4.4.11} Let $G=(V,E)$ be a
  distance-regular graph with second largest eigenvalue
  $\theta_1=b_1-1$. Then at least one of the following holds:
\begin{itemize}
	\item[{\rm (i)}] $G$ is a strongly regular graph with smallest eigenvalue -2;
	\item[{\rm (ii)}] $\mu=1$, i.e., $G$ has numerical girth at least $5$;
	\item[{\rm (iii)}] $\mu=2$, and $G$ is a Hamming graph, a Doob graph, or a locally Petersen graph;
	\item[{\rm (iv)}] $\mu=4$, and $G$ is a Johnson graph;
	\item[{\rm (v)}] $\mu=6$, and $G$ is a demi-cube;
	\item[{\rm (vi)}] $\mu=10$, and $G$ is the Gosset graph.
\end{itemize}
\end{theorem}

This classification contains all our examples as well as many others
which are not Bonnet-Myers sharp. In Section \ref{sec:curvandeig}, we will
identify all Lichnerowicz sharp graphs in this classification. It can
be checked from this classification that all \emph{distance regular}
Bonnet-Myers sharp graphs are just the ones given in our examples. In
Sections \ref{sec:transpgeod} and \ref{sec:antBMstrspher}, however, we
will follow a different route and prove a much stronger result for
Bonnet-Myers sharp graphs: a full classification of all
\emph{self-centered} Bonnet-Myers sharp graphs.

\begin{table}[h!]
\centering
\begin{tabular}{l|l|l|l|l|l|l} 
$G$ & $(D,L)$ & $|V|$ & $\dim E_{\lambda_1}$ & $\mu$-graph & $S_1(x)$ &
intersection array \\[.1cm] \hline &&&&&& \\[-.2cm] 

$Q^n$ & $(n,n)$ & $2^n$ & $n$ & $CP(1)$ & $n$ points & $c_j=j$ \\[.1cm] 

$CP(n)$ & $(2n-2,2)$ & $2n$ & $n$ & $CP(n-1)$ & $CP(n-1)$ & 
$c_2=2n-2$ \\[.1cm] 

$J(2n,n)$ & $(n^2,n)$ & $2n \choose n$ & $2n-1$ & $CP(2)$ & $K_n \times K_n$
& $c_j = j^2$ \\[.1cm] 

$Q^{2n}_{(2)}$ & $(2n^2-n,n)$ & $2^{2n-1}$ & $2n$ & $CP(3)$ & $J(2n,2)$ & $c_j=j(2j-1)$ \\[.1cm] 

Gosset & $(27,3)$ & $56$ & $7$ & $CP(5)$ & Schl\"afli & $(c_2,c_3)=(10,27)$
\end{tabular}
\caption{$(D,L)$-Bonnet-Myers sharp graphs from Subsections 
\ref{subsec:ex_hypcubes}--\ref{subsec:ex_gosset}} 
\label{table:BMsharp_examples}
\end{table}

\section{General facts about Bonnet-Myers sharp graphs}
\label{sec:genfacts}

\subsection{A useful inequality and its applications}

Henceforth, $\Delta = \Delta_G$ denotes the normalized Laplacian on a
graph $G=(V,E)$ defined in \eqref{eq:Deltanorm}. We start with the following
lemma.

\begin{lemma} \label{lem:W1_Delta}
	Let $G=(V,E)$ be a finite connected $D$-regular graph and
	$u,v \in V$, $p \in [0,1]$. Then for any function $f:V\rightarrow \IR$,
	$$ \sum_{z \in V} f(z) \left( \mu_u^p(z)-\mu_v^p(z) \right) = f(u) - f(v) + (1-p)\Delta f(u) - 
	(1-p)\Delta f(v). $$
	In particular, if $f \in \textrm{\rm{1}--{\rm Lip}}(V)$, then
	$$ W_1(\mu_u^p,\mu_v^p) \ge f(u) - f(v) + (1-p)\Delta f(u) - 
	(1-p)\Delta f(v), $$
        with equality iff $f$ is an optimal Kantorovich potential transporting
        $\mu_u^p$ to $\mu_v^p$.
\end{lemma}

\begin{proof}
	The first statement is a simple calculation relating $\mu_u^p$ and $\mu_u^p$ to the Laplaction $\Delta$:
	\begin{eqnarray*}
		\sum_{z \in V} f(z) \left( \mu_u^p(z)-\mu_v^p(z) \right)
		&=& \left[ p f(u) + \frac{1-p}{D} \sum_{z \sim u} f(z) \right] - 
		\left[ p f(v) + \frac{1-p}{D} \sum_{w \sim v} f(w) \right] \\
		&=& f(u) + (1-p)\Delta f(u) - f(v) - (1-p)\Delta f(v).  
	\end{eqnarray*}
	
	In particular, the second statement follows immediately from Theorem \ref{Kantorovich} (Kantorovich Duality).
\end{proof}

The following result from \cite{MW17} will prove very useful in our investigations.

\begin{theorem}[\cite{MW17}]\label{ineq}
  Let $G = (V,E)$ be a finite connected $D$-regular graph with diameter $L$. Let
  $x,y\in V$ with $d(x,y) = L$, and $z$ be a vertex lying on a geodesic
  from $x$ to $y$ and $f \in \textrm{\rm{1}--\rm{Lip}}(V)$ satisfy 
  $f(y)-f(x) = L$. Then
  $$\Delta f(z) \leq 1 -\kappa(x,z)d(x,z).$$
\end{theorem}

A very similar result was stated in \cite[Theorem 4.1]{MW17} for the
specific function $f = d(x,\cdot)$. The proof is a straightforward
consequence of the alternative definition \eqref{eq:kappally_alt} of
Ollivier Ricci curvature. For the reader's convenience, we provide a
direct proof not making use of \eqref{eq:kappally_alt}.

\begin{proof}
  Observe that the negative function $-f$ lies in
  $\textrm{{1}--{\rm Lip}}(V)$, too. Choosing $p = \frac{1}{D+1}$,
  Lemma \ref{lem:W1_Delta} implies
  $$ W_{1}(\mu_{x},\mu_{z}) \ge f(z) - f(x) -\frac{D}{D+1} \Delta f(x) + \frac{D}{D+1} \Delta f(z). $$
  Thus
  $$\Delta f(z) \leq \frac{D+1}{D}W_{1}(\mu_{x},\mu_{z}) +
  \frac{D+1}{D}(f(x)-f(z))+\Delta f(x).$$ 
  Note that, since $f \in \textrm{\rm{1}--{\rm Lip}}(V)$ and since
  $f(y) - f(x) = L = d(y,x)$, $f(z) - f(x) = d(x,z)$ and
  $\Delta f(x) \leq 1.$ Therefore
  \begin{align*}
    \Delta f(z) & \leq \frac{D+1}{D}(W_{1}(\mu_{x},\mu_{z})-d(x,z))+1
    \\
                & = 1 - \kappa(x,z)d(x,z).
  \end{align*}
\end{proof}

\begin{lemma}\label{geosharp}
  Let $G = (V,E)$ be a $(D,L)$-Bonnet-Myers sharp graph and let
  $x,y\in V$ with $d(x,y) = L$.
  Let $z,w$ be two different vertices lying on a geodesic from $x$ to
  $y$ with $d(x,z)+d(z,w)+d(w,y) = L$. Then
  $$\kappa(z,w) = \frac{2}{L}.$$
\end{lemma}

\begin{proof}
  We have to show that
  $\kappa_{\frac{1}{D+1}}(z,w) = \frac{D}{D+1}\frac{2}{L}$. The
  triangle inequality tells us that
  \begin{align*}
    W_{1}(\mu_{x},\mu_{y})  \leq & W_{1}(\mu_{x},\mu_{z})+W_{1}(\mu_{z},\mu_{w})+W_{1}(\mu_{w},\mu_{y})
    \\
    = & d(x,z)(1-\kappa_{\frac{1}{D+1}}(x,z))+d(z,w)(1-\kappa_{\frac{1}{D+1}}(z,w))+d(w,y)(1-\kappa_{\frac{1}{D+1}}(w,y))
    \\
    = & L - d(x,z)\kappa_{\frac{1}{D+1}}(x,z)-d(z,w)\kappa_{\frac{1}{D+1}}(z,w)-d(w,y)\kappa_{\frac{1}{D+1}}(w,y).
  \end{align*}
  Bonnet-Myers sharpness means that
  $$\kappa_{\frac{1}{D+1}}(u,v) \geq \frac{D}{D+1}\frac{2}{L},$$
  for all $u,v\in V, u\neq v$. Assume that
  $\kappa_{\frac{1}{D+1}}(z,w) > \frac{2D}{D+1}\frac{1}{L}$. Then
  $$ W_{1}(\mu_{x},\mu_{y}) < L - 
  \frac{D}{D+1}\frac{2}{L}(d(x,z)+d(z,w)+d(w,y)) = L - \frac{2D}{D+1},
  $$
  and so
  $$\kappa(x,y) > \frac{2}{L},$$
  which is a contradiction to \eqref{eq:kapzw_ineq}. Thus
  $$\kappa(z,w) = \frac{2}{L}.$$
\end{proof}

Note that we can choose $z=x$ or $w = y$ in the statement of the lemma above.
This, together with Theorem \ref{ineq}, leads to the following result:

\begin{theorem}\label{Ftrick}
  Let $G = (V,E)$ be a $(D,L)$-Bonnet-Myers sharp graph with diameter
  $L$. Let $x,y\in V$ with $d(x,y) = L$, and $z$ be a vertex lying on a
  geodesic from $x$ to $y$, and
  $f\in \textrm{\rm{1}--{\rm Lip}}$ satisfy $f(y)-f(x) = L$. Then
  \begin{equation} \label{eqn:Ftrick}
  \Delta f(z) = 1 -\frac{2d(x,z)}{L}.
  \end{equation}
  
\end{theorem}

\begin{proof}
From Lemma \ref{geosharp}, $\kappa(x,z) = \kappa(y,z) = \frac{2}{L}.$ We obtain, from Theorem \ref{ineq}
$$\Delta f(z) \leq  1 -\frac{2d(x,z)}{L},$$
and
$$\Delta (-f(z)) \leq  1 -\frac{2d(y,z)}{L} = 1 -\frac{2(L-d(x,z))}{L} = \frac{2d(x,z)}{L} - 1 .$$
Thus 
$$\Delta f(z) = 1 -\frac{2d(x,z)}{L},$$
as required. 
\end{proof}

\begin{theorem}\label{lieongeos}
  Let $G = (V,E)$ be a $(D,L)$-Bonnet-Myers sharp graph and $x,y\in V$
  with $d(x,y) = L$. Then any vertex
  $u\in V$ lies on a geodesic from $x$ to $y$, that is, we have
  $[x,y] = V$.
\end{theorem}

\begin{proof}
  Let $f : = d(x,\cdot),\: g:= L - d(\cdot, y)$. We must show that
  $f = g$. Note that
  $$ f(z)-g(z) = d(x,z) + d(z,y) - L \geq 0. $$
  Thus $f\geq g$. Suppose $f \neq g$. Then there exists a vertex $z$
  closest to $x$ with $f(z)>g(z)$ with $z \neq x$. Hence there exists
  a vertex $w\sim z$ with $d(x,w)< d(x,z)$. By our assumptions we have
  $f(w) = g(w)$. Therefore $w$ is on a geodesic from $x$ to $y$. Thus,
  by Theorem \ref{Ftrick}, we have $\Delta f(w) = \Delta g(w)$.
  However $f(z)>g(z)$ and so
  \begin{eqnarray*}
    D \Delta f(w) &=& \sum_{u\sim w} (f(u)-f(w)) \\
                  &=& f(z) - f(w) + \sum_{\substack{u\sim w \\ u\neq z}}(f(u)-f(w)) \\
                  &>& g(z) - g(w) + \sum_{\substack{u\sim w \\ u\neq z}}(g(u)-g(w)) \\
                  &=& D \Delta g(w)
  \end{eqnarray*}
  which is a contradiction. Thus $f = g$, completing the proof.
\end{proof}

\begin{corollary} 
  Let $G = (V,E)$ be Bonnet-Myers sharp. Then every vertex in $V$ has
  at most one antipole.
\end{corollary}

\begin{proof}
  Assume $x \in V$ has two different antipoles $y_1$ and $y_2$. Then
  $y_1$ lies on a geodesic from $x$ to $y_2$, by Theorem
  \ref{lieongeos} and, since $y_1 \neq y_2$,
  $$ d(x,y_1) < d(x,y_2) = {\rm diam}(G), $$
  which contradicts to the assumption that $y_1$ is an antipole of $x$.
\end{proof}

\begin{theorem}\label{onespheredegree}
	Let $G= (V,E)$ be $(D,L)$-Bonnet-Myers sharp graph and $x \in V$ be a pole. Then we have for every edge $\{x,y\} \in E$:
	\begin{itemize}
		\item[(a)] The edge $\{x,y\}$ lies in precisely $\frac{2D}{L}-2$ triangles, or, in other words, the induced subgraph $S_1(x)$ is $\left( \frac{2D}{L}-2 \right)$-regular;
		\item[(b)] There is a perfect matching between the sets $N_x \backslash (N_{xy} \cup \{y\})$ and
		$N_y \backslash (N_{xy} \cup \{x\})$;
		\item[(c)] There is an optimal transport plan $\pi$ transporting $\mu_x$ to $\mu_y$ which is based on triangles and a perfect matching (see Definition \ref{def:plan_tpm}) with the cost $$\textup{cost}(\pi)=\frac{1}{D+1}\left(D+1-\frac{2D}{L}\right).$$
	\end{itemize}
\end{theorem}

\begin{proof}
	Let $x\in V,$ and define $f(w) = d(x,w)- \frac{L}{2}.$ By \eqref{eqn:Ftrick}, we have $\Delta f +\frac{2}{L} f = 0.$ Let $y\in S_{1}(x).$ We need to show that $d_{x}^{0}(y) = \frac{2D}{L}-2.$ Now since $f(z)-f(y) = 0$ if $z\sim y,$ $z\in S_{1}(x),$
	\begin{align*}
	0 = \Delta f(y) + \frac{2}{L} f(y) & = \frac{1}{D}\left[ (f(x)-f(y))+ \sum_{z\sim y, z \in S_{2}(x)}(f(z)-f(y))\right]+\frac{2}{L} f(y)
	\\
	& =\frac{1}{D}(d_{x}^{+}(y)-1)+\frac{2}{L}\left(1-\frac{L}{2}\right)
	\\
	& = \frac{-1}{D} + \frac{1}{D}d_{x}^{+}(y)- 1 +\frac{2}{L}.
	\end{align*}
	Rearranging gives
	$$d_{x}^{+}(y) = D+1 - \frac{2D}{L}$$
	and, therefore,
	$$d_{x}^{0}(y) = D - d^{+}_{x}(y)-d^{-}_{x}(y)= D - (D+1 -\frac{2D}{L}) -1 = \frac{2D}{L} - 2.$$
	
	We now show that there is an optimal transport plan transporting $\mu_{x}$ to $\mu_{y}$ which is based on triangles and a perfect matching of the neighbours of $x$ and $y$. We will prove this indirectly. Suppose that there is no optimal transport plan based on triangles and a perfect matching. Thus, in any optimal transport plan, there must be some mass in $N_{xy}\setminus N_x$ that travels with the distance more than $1$, so the total cost is
	$$W_{1}(\mu_{x},\mu_{y}) > \frac{1}{D+1}(D+1-\frac{2D}{L}),$$
	and thus
	$$\kappa(x,y) = \frac{D+1}{D}\kappa_{\frac{1}{D+1}}(x,y) < \frac{2}{L},$$
	which contradicts to the fact that $G$ is Bonnet-Myers sharp (see Lemma \ref{geosharp}).
	
	Therefore, there is an optimal transport plan $\pi$ based on triangles and a perfect matching with the cost:
	$$ \textup{cost}(\pi)=W_{1}(\mu_{x},\mu_{y})= \frac{1}{D+1}(D+1-\frac{2D}{L}). $$
\end{proof}

We now give relations between the in, out and spherical degrees of vertices inside Bonnet-Myers sharp graphs.

\begin{theorem}\label{recursionformulas}
	Let $G= (V,E)$ be a $(D,L)$-regular Bonnet-Myers sharp graph and $x \in V$ be a pole. Let $y\in S_{k}(x),$ where $k\in \mathbb{N}.$ Then
	\begin{eqnarray}
	d_{x}^{+}(y) - d_{x}^{-}(y) &=& D\left(1-\frac{2k}{L}\right), \label{eq:rec1} \\
	2d_{x}^{+}(y) + d_{x}^{0}(y) &=& 2D\left(1-\frac{k}{L}\right), \label{eq:rec2} \\
	2d_{x}^{-}(y) + d_{x}^{0}(y) &=& \frac{2kD}{L}. \label{eq:rec3}
	\end{eqnarray}
\end{theorem}

\begin{proof}
	Define $f(w) = d(x,w)- \frac{L}{2}.$ By \eqref{eqn:Ftrick}, we have $\Delta f +\frac{2}{L} f = 0.$ Thus
	\begin{align*}
	\Delta f(y) + \frac{2}{L} f(y) & = \frac{1}{D}\left[  \sum_{z\sim y}(f(z)-f(y))\right]+\frac{2}{L} f(y)
	\\
	& =\frac{1}{D}\left[  d_{x}^{+}(y)- d_{x}^{-}(y)\right]+\frac{2}{L} \left(k-\frac{L}{2}\right).
	\end{align*}
	Rearranging gives 
	$$d_{x}^{+}(y) - d_{x}^{-}(y) = D\left(1-\frac{2k}{L}\right).$$
	The rest of the formula are obtained by using $D = d_{x}^{+}(y) + d_{x}^{0}(y) + d_{x}^{-}(y) = D $ and algebraic manipulation.
\end{proof}

We end this subsection with the proof of Theorem \ref{thm:DL_rel}.
\begin{theoremdivis}
Any $(D,L)$-Bonnet-Myers sharp graph satisfies $L \le D$. Moreover $L$ must divide $2D$. 
\end{theoremdivis}

\begin{proof}[Proof of Theorem \ref{thm:DL_rel}]
Let $x$ be a pole of $G$. Choosing $k=1$ in Theorem \ref{recursionformulas}, we conclude 
from \eqref{eq:rec3} that $L$ must divide $2D$. Therefore, in the case $L > D$, we must have
$L=2D$. Choosing $k=L-1$ in \eqref{eq:rec2} would lead to
$$ 2d_{x}^+(y) + d_x^0(y) = 2D \left( 1 - \frac{L-1}{L} \right) = 2D - (L-1) = 1, $$
which would imply $d_{x}^+(y) = 0$ and $d_x^0(y) = 1$, which cannot be since some $y \in S_{L-1}(x)$ must be a neighbour of the antipole of $x$. Therefore, we have ruled out $L=2D$ and we conclude $L \le D$.
\end{proof}

\subsection{Self-centered Bonnet-Myers sharp graphs}
\label{sec:self-centBMsh}
This subsection provides two immediate consequences of the results from the previous subsection under the extra assumption of self-centeredness, i.e. every vertex is a pole.

\begin{theorem}\label{constcurv}
  Let $G=(V,E)$ be a self-centered $(D,L)$-Bonnet-Myers sharp
  graph. Then we have for any pair $z,w$ of different vertices
  $$ \kappa(z,w) = \frac{2}{L}, $$
  that is, $G$ has constant Ollivier-Ricci curvature $\frac{2}{L}$.
\end{theorem}

\begin{proof}
Let $z'$ be the antipole of $z$ in $G$. Then $w$ lies on a geodesic from $z$ to $z'$, by Theorem \ref{lieongeos} and Lemma \ref{geosharp} yields
$$ \kappa(z,w) = \frac{2}{L}. $$
\end{proof}

\begin{corollary} \label{cor:numtriangle}
Let $G=(V,E)$ be a self-centered $(D,L)$-Bonnet-Myers sharp graph. Then we have for every edge $\{x,y\} \in E$:
\begin{itemize}
	\item[(a)] The edge $\{x,y\}$ lies in precisely $\frac{2D}{L}-2$ triangles, or, in other words, the induced subgraph $S_1(x)$ is $\left( \frac{2D}{L}-2 \right)$-regular;
	\item[(b)] There is a perfect matching between the sets $N_x \backslash (N_{xy} \cup \{y\})$ and
	$N_y \backslash (N_{xy} \cup \{x\})$;
	\item[(c)] There is an optimal transport plan $\pi$ transporting $\mu_x$ to $\mu_y$ which is based on triangles and a perfect matching with the cost $$\textup{cost}(\pi)=\frac{1}{D+1}\left(D+1-\frac{2D}{L}\right).$$
\end{itemize}
\end{corollary}

\begin{proof}
It follows immediately from Theorem \ref{onespheredegree}, since every vertex of a self-centered graph is a pole.
\end{proof}

\section{Curvatures and eigenvalues}
\label{sec:curvandeig}

The following result agrees with Theorem \ref{thm:lichn} and it
provides additional information about the choice of a suitable
eigenfunction.

\begin{theoremlichn}
  Every $(D,L)$-Bonnet-Myers sharp graph $G = (V,E)$ is Lichnerowicz
  sharp with Laplace eigenfunction $f = d(x,\cdot) - \frac{L}{2}$,
  where $x$ is a pole of $G$.
\end{theoremlichn}


\begin{proof}
  Let ${\rm diam}(G) = L$ and $x,y\in V$ satisfy $d(x,y) = L$. Let
  $f = d(x,\cdot)-\frac{L}{2}$. By Theorem \ref{lieongeos} every vertex $z\in V$
  lies on a geodesic from $x$ to $y$. Thus, by Lemma \ref{geosharp},
  $\kappa(x,z) = \frac{2}{L}$. Then, by Theorem \ref{Ftrick}, we have
  $$ \Delta f +\frac{2}{L}f
  = 0. $$
  Therefore $\lambda_{1}\leq \frac{2}{L}$. By the Discrete Lichnerowicz Theorem \ref{thm:Lich}
  and Bonnet-Myers sharpness, we have
  $$ \lambda_{1}\geq \inf_{\substack{u,v\in V\\ u\neq v}} \kappa(u,v) 
  = \frac{2}{L}. $$ 
  Thus $\lambda_{1} = \frac{2}{L} = \inf_{u,v\in V, u\neq v} \kappa(u,v)$,
  completing the proof.
\end{proof}

We also have the following general relation between
eigenfunctions, curvature and Kantorovich potentials:

\begin{theorem}
  Let $G=(V,E)$ be a finite connected $D$-regular graph. Let
  $f: V \rightarrow \IR$ be a Laplace eigenfunction, which is also an optimal Kantorovich potential
  transporting $\mu_u^p$ to $\mu_v^p$ for some $u,v \in V$, $u \neq v$,
  $p \in (\frac{1}{2},1)$. Then $\lambda=\kappa_{LLY}(u,v)$ with
  $\kappa_{LLY}$ defined in \eqref{eq:kappaLLY}.
\end{theorem}

\begin{proof}
  Since $f$ is an optimal Kantorovich potential transporting $\mu_u^p$
  to $\mu_v^p$ and $\Delta f+\lambda f=0$, Lemma \ref{lem:W1_Delta}
  yields
  $$ W_1(\mu_u^p,\mu_v^p) = \sum_{z \in V} f(z) \left( \mu_u^p(z)-\mu_v^p(z) 
  \right) = (1-(1-p)\lambda) ( f(u) - f(v)). $$ 
  Moreover, since $p>\frac{1}{2}$, every optimal transport plan $\pi$
  from $\mu_u^p$ to $\mu_v^p$ must satisfy $\pi(u,v)>0$, which then
  implies by complementary slackness that $f(u)-f(v)=d(u,v)$ (see,
  e.g., \cite[Lemma 3.1]{Idle}, which states this fact for neighbours
  $x,y \in V$, but it is also true for arbitrary pairs of different
  vertices).
  
  Substituting $f(u)-f(v)=d(u,v)$ in the above equation yields
  $$ \kappa_p(u,v) = 1- \frac{W_1(\mu_u^p,\mu_v^p)}{d(u,v)}
  =(1-p)\lambda, $$
  which implies
  $$ \frac{\kappa_p(u,v)}{1-p} = \lambda. $$
  Since $p \in (\frac{1}{2},1)$, we conclude from \cite[Corollary 3.4]{CK2018} 
  $$ \kappa_{LLY}(u,v) = \lim_{p \to 1} \frac{\kappa_p(u,v)}{1-p} = 
  \frac{\kappa_p(u,v)}{1-p}, $$ 
  finishing the proof.
\end{proof}

The following result can be derived via similar arguments (but a
different logic in its proof):

\begin{theorem} 
  Let $G=(V,E)$ be Lichnerowicz sharp with a Laplace eigenfunction
  $f \in \textrm{\rm{1}--{\rm Lip}}(V)$ associated to the eigenvalue
  $\lambda_1 = \inf_{x \sim y} \kappa(x,y)$. Then, for any pair of
  different vertices $u,v \in V$ with $f(u)-f(v) = d(u,v)$, $f$ is an
  optimal Kantorovich potential transporting $\mu_u$ to $\mu_v$.
\end{theorem}

\begin{proof}
  Assume $f \in \textrm{\rm{1}--{\rm Lip}}(V)$, $\Delta f + \lambda_1 f = 0$ and
  $p = \frac{1}{D+1}$. Lemma \ref{lem:W1_Delta} yields
  $$ W_1(\mu_u,\mu_v) \ge (1-(1-p)\lambda_1) ( f(u) - f(v)), $$
  with equality iff $f$ is an optimal Kantorovich potential
  transporting $\mu_u$ to $\mu_v$. This is equivalent to
  $$ \kappa_p(u,v) = 1- \frac{W_1(\mu_u,\mu_v)}{d(u,v)}
  \le (1-p)\lambda_1 \frac{f(u)-f(v)}{d(u,v)} = (1-p)\lambda_1, $$
  using $f(u)-f(v) = d(u,v)$. This, in turn, is equivalent to
  \begin{equation} \label{eq:kappalambda1} 
  \kappa(u,v) = \frac{\kappa_p(u,v)}{1-p} \le \lambda_1. 
  \end{equation}
  Our assumption $\lambda_1 = \inf_{x \sim y} \kappa(x,y)$ then implies 
  equality in \eqref{eq:kappalambda1} and, therefore, $f$ is
  an optimal Kantorovich potential transporting $\mu_u$ to $\mu_v$.
\end{proof}



Finally, let us identify all Lichnerowicz sharp graphs within an
interesting family of distance regular graphs. More precisely, as
mentioned in Section \ref{sec:examples}, there is a classification of all
distance-regular graphs with second largest adjacency eigenvalue
$\theta_1=b_1-1$ (see Theorem \ref{thm:4.4.11}). This class of
graphs comprises all strongly regular graphs with smallest adjacency
eigenvalue $-2$.  In this subclass, we have the following Lichnerowicz
sharp graphs.

\begin{theorem} \label{thm:lichstrreg}
  The Lichnerowicz sharp strongly regular graphs with smallest
  adjacency eigenvalue $-2$ are precisely the following ones: The
  cocktail party graphs $CP(n)$, $n \ge 2$, the lattice graphs
  $L_2(n) \cong K_n \times K_n$, $n \ge 3$, the triangular graphs
  $T(n) \cong J(n,2)$, $n \ge 5$, the demi-cube $Q^5_{(2)}$, and the
  Schl\"afli graph.
\end{theorem}

\begin{proof}
  The theorem follows directly from Table
  \ref{table:strongly_reg_Lichsharp_examples} below, where the
  curvatures $\inf_{x \sim y} \kappa(x,y)$ were determined with the
  help of the curvature calculator \cite{CKLLS17} at
  \begin{center}
  \url{http://www.mas.ncl.ac.uk/graph-curvature/}
  \end{center}
  Note that Chang stands for any one of the three Chang graphs. For
  the classification of all strongly regular graphs with smallest
  adjacency eigenvalue $-2$, see \cite[Theorem 9.2.1]{BH}.

\begin{table}[h!]
\centering
\begin{tabular}{l|l|l|l|l}
$G$ & $(\nu,k,\lambda,\mu)$ & $\theta_1$ & $\lambda_1$ & $\displaystyle{\inf_{x \sim y} \kappa(x,y)}$ \\[.1cm] 
\hline &&&& \\[-.2cm]

$CP(n)$ & $(2n,2n-2,2n-4,2n-2)$ & $0$ & $1$ & $1$ \\[.1cm]

$K_n \times K_n$ & $(n^2,2(n-1),N-2,2)$ & $n-2$ & $\frac{n}{2(n-1)}$ & $\frac{n}{2(n-1)}$ \\[.1cm]

Shrikhande & $(16,6,2,2)$ & $2$ & $\frac{2}{3}$ & $\frac{1}{3}$ \\[.1cm]

$J(n,2)$ & $\left({n \choose 2},2(n-2),n-2,4\right)$ & $n-4$  & $\frac{n}{2(n-2)}$ & $\frac{n}{2(n-2)}$ \\[.1cm]

Chang & $(28,12,6,4)$ & $4$ & $\frac{2}{3}$ & $\frac{1}{3}$ \\[.1cm]

Petersen & $(10,3,0,1)$ & $1$ & $\frac{2}{3}$ & $0$ \\[.1cm]

$Q^5_{(2)}$ & $(16,10,6,6)$ & $2$ & $\frac{4}{5}$ & $\frac{4}{5}$ \\[.1cm]

Schl\"afli & $(27,16,10,8)$ & $4$ & $\frac{3}{4}$ & $\frac{3}{4}$ \\[.1cm]
\end{tabular}
\caption{Strongly regular graphs with smallest eigenvalue equals $-2$, with their smallest positive Laplace eigenvalue $\lambda_1$ and the infimum of their Ollivier Ricci curvatures}
\label{table:strongly_reg_Lichsharp_examples}
\end{table}
\end{proof}

In the classification Theorem \ref{thm:4.4.11}, we can disregard all examples with $\mu=1$ (that is, all vertices at distance $2$ have precisely one neighbour in common), since none of them can be
Lichnerowicz sharp due to the following result:

\begin{theorem} \label{thm:mu1_notlich}
  A distance-regular graph with second largest adjacency eigenvalue $\theta_1=b_1-1$ and $\mu=1$ cannot be Lichnerowicz sharp.
\end{theorem}

\begin{proof}
  Let $G=(V,E)$ be a distance-regular graph of vertex degree $D$ and
  satisfying $\mu=1$, and $x,z \in V$ with $d(x,z) = 2$. We denote the
  unique common neighbour of $x$ and $z$ by $y$. Then we have
  $0 \le d_x^0(z)=:\alpha \le D-1$ and $b_1 = d_x^+(z) =
  D-1-\alpha$. The second largest adjacency eigenvalue
  is then $\theta_1=b_1-1=D-2-\alpha$ and, consequently, the smallest positive Laplace eigenvalue is
  $$ \lambda_1 = 1 - \frac{D-2-\alpha}{D} = \frac{2+\alpha}{D} > 0. $$
  Lichnerowicz' Theorem tells us that
  $$ \inf_{u \sim v} \kappa(u,v) = \inf_{u \neq v} \kappa(u,v) \le \lambda_1 = \frac{2+\alpha}{D}. $$
  Let us now estimate $\kappa(x,z)$. We have
  $$ W_1(\mu_x,\mu_z) \ge \frac{1}{D+1}\left( 2 + 2(D-1-\alpha) + \alpha \right) 
  = \frac{2D-\alpha}{D+1}, $$ and, therefore,
  $$ \kappa_{\frac{1}{D+1}}(x,z) 
  = 1 - \frac{W_1(\mu_x,\mu_z)}{2} \le \frac{1+\frac{\alpha}{2}}{D+1}. $$ 
  This implies that
  $$ \inf_{u \sim v} \kappa(u,v) \le \kappa(x,z) = \frac{D+1}{D} \kappa_{\frac{1}{D+1}}(x,z) \le 
  \frac{1+\frac{\alpha}{2}}{D} = \frac{\lambda_1}{2} < \lambda_1. $$
  This shows that $G$ cannot be Lichnerowicz sharp.
\end{proof}

Using the previous two results, the following theorem provides a complete
classification of all Lichnerowicz sharp distance-regular graphs with
second largest adjacency eigenvalue $\theta_1=b_1-1$:

\begin{theorem} The Lichnerowicz sharp distance-regular graphs with
  second largest adjacency eigenvalue $\theta_1=b_1-1$ are precisely the
  following ones: 
  \begin{enumerate}
  \item the cocktail party graphs $CP(n)$ (also Bonnet-Myers sharp);
  \item the Hamming graphs $H(n,d) = (K_n)^d$ (only Bonnet-Myers sharp if $n=2$, that is $H(n,d) = Q^d$);
  \item the Johnson graphs $J(n,k)$ (only Bonnet-Myers sharp if $n=2k$);
  \item the demi-cubes $Q^n_{(2)}$ (only Bonnet-Myers sharp if $n$ is even);
  \item the Schl\"afli graph (not Bonnet-Myers sharp);
  \item the Gosset graph (also Bonnet-Myers sharp).
  \end{enumerate}
\end{theorem}

\begin{proof}
  The theorem is an immediate consequence of the classification
  Theorem \ref{thm:4.4.11}, Theorems \ref{thm:lichstrreg} and
  \ref{thm:mu1_notlich}, and Table
  \ref{table:Lichsharp_examples} below. As before, the curvatures
  $\inf_{x \sim y} \kappa(x,y)$ were determined with the help of the
  curvature calculator \cite{CKLLS17} at
  \begin{center}
  \url{http://www.mas.ncl.ac.uk/graph-curvature/}
  \end{center}
  Note that the Doob graphs are given by
  ${\rm Doob}^{n,m} = K_4^n \times {\rm Shk}^m$ with $n,m \ge 1$,
  where ${\rm Shk}$ denotes the Shrikhande graph, and the
  $(7,2)$-Kneser, Conway-Smith graph and Hall graph are the three
  locally Petersen graphs.
 
\begin{table}[h!]
\centering
\begin{tabular}{l|l|l|l|l|l|l}
$G$ & $|V|$ & $D$ & $L$ & $\theta_1=b_1-1$ & $\lambda_1$ & $\displaystyle{\inf_{x \sim y} \kappa(x,y)}$ \\[.1cm] 
\hline &&&&& \\[-.2cm]

$(K_n)^d$ & $n^d$ & $d(n-1)$ & $d$ & $n(d-1)-d$ & $\frac{n}{d(n-1)}$ & $\frac{n}{d(n-1)}$ \\[.1cm]

${\rm Doob}^{n,m}$ & $4^{n+2m}$ & $3(n+2m)$ & $n+2m$ & $3(n+2m)-4$ & $\frac{4}{3(n+2m)}$ & $\frac{2}{3(n+2m)}$ \\[.1cm]

$(7,2)$-Kneser & $21$ & $10$ & $2$ & $3$ & $\frac{7}{10}$ & $\frac{1}{2}$ \\[.1cm]  

Conway-Smith & $63$ & $10$ & $4$ & $5$ & $\frac{1}{2}$ & $-\frac{1}{10}$ \\[.1cm]

Hall & $65$ & $10$ & $3$ & $5$ & $\frac{1}{2}$ & $-\frac{1}{10}$ \\[.1cm]

$J(n,k)$ & $n \choose k$ & $k(n-k)$ & $\min(k,n-k)$ & $k(n-k)-n$ & $\frac{n}{k(n-k)}$ & $\frac{n}{k(n-k)}$ \\[.1cm]

$Q^n_{(2)}$ & $2^{n-1}$ & $n \choose 2$ & $\lfloor \frac{n}{2}\rfloor$ & $\frac{(n-4)(n-1)}{2}$ & $\frac{4}{n}$ & $\frac{4}{n}$ \\[.1cm]

Gosset & $56$ & $27$ & $3$ & $9$ & $\frac{2}{3}$ & $\frac{2}{3}$

\end{tabular}
\caption{Distance-regular graphs with second largest eigenvalue $\theta_1=b_1-1$ not yet considered and taken from Theorem \ref{thm:4.4.11}}
\label{table:Lichsharp_examples}
\end{table}
\end{proof}

\section{Transport geodesics of self-centered Bonnet-Myers sharp graphs}
\label{sec:transpgeod}

This section together with the next one is dedicated to the proof that
the examples in Subsections \ref{subsec:ex_hypcubes}-\ref{subsec:ex_gosset}
and suitable Cartesian products of them are the only
self-centered Bonnet-Myers sharp graphs. Of crucial importance in this proof are transport geodesic techniques. In view of this result, it is
natural to ask the following:

\begin{question} 
Are there any Bonnet-Myers sharp graphs which are not self-centered?
\end{question}

We assume that all Bonnet-Myers sharp graphs are self-centered, but
this is currently still an open problem. Let us now start to introduce
the relevant tools to achieve the above mentioned goal.

\subsection{Concatenation of transport maps}
\label{sec:concat_tramap}

Let $G=(V,E)$ be a simple, connected, $D$-regular graph and $\mu_0,\mu_1$ be probability measures on $V$. A transport plan $\pi \in \Pi(\mu_0,\mu_1)$ is induced by 
a \emph{transport map} $T: {\rm supp}(\mu_0) \to {\rm supp}(\mu_1)$ if 
$\mu_1(T(x)) = \mu_0(x)$ for all $x \in V$ and
$$ \pi(x,y)  = \begin{cases} \mu_0(x), & \text{if $x \in {\rm supp}(\mu_0)$ and $y = T(x)$}, \\ 0, & \text{otherwise.} \end{cases} $$
We define the \emph{cost} of a transport map $T$ as the cost of its induced transport plan $\pi: V \times V \to [0,1]$:
$$ {\rm cost}(T) := {\rm cost}(\pi) = \sum_{x \in {\rm supp}(\mu_0)} d(x,T(x)) \mu_0(x). $$
$T$ is called an \emph{optimal transport map} from $\mu_0$ to $\mu_1$ if its induced plan $\pi \in \Pi(\mu_0,\mu_1)$ is an optimal transport plan. The existence of optimal transport maps for given probability measures $\mu_0, \mu_1$ is known as the \emph{Monge Problem}. 

Let $T_1, T_2$ be transport maps from $\mu_0$ to $\mu_1$ and from $\mu_1$ to $\mu_2$, respectively. These transport maps can be concatenated to a transport map from $\mu_0$ to $\mu_2$ via 
$$ T_2 \circ T_1: {\rm supp}(\mu_0) \to {\rm supp}(\mu_2). $$
The following fact about concatenation will be useful henceforth. 

\begin{proposition}[transport geodesic]\label{prop:transgeod}
Let $G=(V,E)$ be a simple, connected $D$-regular graph and $\mu_0,\mu_1\dots,\mu_k$ be probability measures on $V$. For $1 \le j \le k$, let $T_j$ be a transport map from $\mu_{j-1}$ to $\mu_j$, and $T^j$ be the concatenated map
$$ T^j := T_j \circ T_{j-1} \circ \cdots \circ T_1: {\rm supp}(\mu_0) \to
{\rm supp}(\mu_j). $$
Then we have
\begin{equation} \label{eq:concat}
{\rm cost}(T^k) \le \sum_{j=1}^k {\rm cost}(T_j).
\end{equation}

Assume that we have equality in \eqref{eq:concat}.  
Then, for each $z \in {\rm supp}(\mu_0)$, the sequence of vertices
$$ z, T^1(z), T^2(z), \cdots, T^k(z) $$
lies on a geodesic from $z$ to $T^k(z)$, that is,
\begin{equation} \label{eq:eqdisttramap}
d(z,T^k(z)) = d(z,T^1(z)) + d(T^1(z),T^2(z)) + \cdots + d(T^{k-1}(z),T^k(z)). 
\end{equation}

Such sequence of vertices $z, T^1(z), T^2(z), \cdots, T^k(z)$ is hence called a \textbf{\textit{transport geodesic}}. 
\end{proposition}

\begin{proof}
	Setting $T^0 = {\rm Id}_{{\rm supp}(\mu_0)}$, we have
\begin{align*}
{\rm cost}(T^k) &= \sum_{z \in {\rm supp}(\mu_0)} d(z,T^k(z)) \mu_0(z)\\
&\stackrel{\Delta}{\le} \sum_{z \in {\rm supp}(\mu_0)} \left( \sum_{j=1}^k d(T^{j-1}(z),T^j(z)) \right) \mu_0(z)\\
&= \sum_{j=1}^k \left( \sum_{z \in {\rm supp}(\mu_0)}  d(T^{j-1}(z),T_j \circ T^{j-1}(z)) \mu_0(z)  
\right)\\ 
&= \sum_{j=1}^k \left( \sum_{x \in {\rm supp}(\mu_{j-1})}  d(x,T_j(x)) \mu_{j-1}(x)   \right) = \sum_{j=1}^k {\rm cost}(T_j),
\end{align*}
with equality if and only if \eqref{eq:eqdisttramap} for all $z \in {\rm supp}(\mu_0)$.
\end{proof}

\subsection{Transport geodesics of a Self-centered Bonnet-Myers sharp graph} \label{subsection: transgeod}

In this subsection and henceforth, we always assume that our $(D,L)$-Bonnet-Myers sharp graph $G=(V,E)$ has the extra condition of self-centeredness.

Let us start with a full-length (i.e. of length $L$) geodesic $g$, and denote the vertices along this geodesic by
$$g: \qquad x_0 \sim x_1 \sim x_2 \sim x_3 \sim \cdots \sim x_L. $$

For every $1\le j \le L$, since $x_{j-1}\sim x_j$ consider an optimal transport map $T_j$ from $\mu_{x_{j-1}}$ to $\mu_{x_{j}}$ based on triangles and a perfect matching (see Theorem \ref{cor:numtriangle}(c)), that is:
$$T_j: B_1(x_{j-1}) \rightarrow B_1(x_{j})$$ is a bijective function and satisfies 
\begin{enumerate}
	\item $x=T_j(x)$ if $x\in B_1(x_{j-1})\cap B_1(x_{j})$, and
	\item $x\sim T_j(x)$ if $x\in B_1(x_{j-1})\setminus B_1(x_{j})$.
\end{enumerate}
For simplicity, we will write 1. and 2. together as $x\simeq T_j(x)$, where the symbol $\simeq$ means ``adjacent or equal to''. 

Moreover, for each $z \in B_1(x_0)$, we define $z(0):=z$ and for $1 \le j \le L$,
$$ z(j) := T^j(z) := T_j \circ \cdots \circ T_1(z) \in B_1(x_j).$$

Note that, in particular, we have
$x_0(1) = x_0(0) = x_0$ by condition 1.

\begin{proposition} \label{prop:transgeod_BMsharp}
Let $G=(V,E)$ be a self-centered $(D,L)$-Bonnet-Myers sharp. Given a full-length geodesic $g$ and maps $T_j$ and $T^j$ (for $1\le j\le L$) defined as above. 
Then for every $z\in B_1(x_0)$, the sequence of vertices $$z(0) \simeq z(1) \simeq z(2)\simeq \cdots \simeq z(L)$$ is a transport geodesic. Since this transport geodesic follows closely the geodesic $g$, we call it a {\bf{\textit{transport geodesic along $\boldsymbol{g}$}}} and denote it by $g_z$.
\end{proposition}

\begin{rem} \label{rem: pm_not_uniq} The definition of a transport
  geodesic $g_z$ depends on a full-length geodesic $g$ and sets of
  transport maps $\{T_j\}_{j=1}^L$. Each $T_j$ is a priori not
  uniquely defined, since the definition of $T_j$ is based on
  triangles and a perfect matching, the latter of which is not
  necessarily unique. We will see later (cf. Remark \ref{rem:
    pm_uniq}) that in fact the maps $\{T_j\}_{j=1}^L$ are already
  uniquely determined by $g$ in the case of self-centered Bonnet-Myers
  sharp graphs.
\end{rem}

\begin{proof}
Note that $T^L$ induces a transport plan from $\mu_{x_0}$ to $\mu_{x_L}$, and together with Theorem \ref{cor:numtriangle}(c), we have
$$ W_1(\mu_{x_0},\mu_{x_L}) \le {\rm cost}(T^L) \le \sum_{j=1}^L {\rm cost}(T_j) = L \cdot\frac{1}{D+1}\left(D+1-\frac{2D}{L}\right)
= L - \frac{2D}{D+1}. $$
On the other hand,
$$ \frac{2}{L} \ge \kappa(x_0,x_L) = \frac{D+1}{D} \kappa_{\frac{1}{D+1}}(x_0,x_L) =
\frac{D+1}{D} \left( 1 - \frac{W_1(\mu_{x_0},\mu_{x_L})}{L} \right) $$
implies that
$$ L - \frac{2D}{D+1} \le W_1(\mu_{x_0},\mu_{x_L}). $$
Bringing these inequalities together, we conclude that
\begin{equation} \label{eq:cost_fullplan}
{\rm cost}(T^L) = \sum_{j=1}^L {\rm cost}(T_j)=L - \frac{2D}{D+1}.
\end{equation}

Then by Proposition \ref{prop:transgeod}, for every $z\in B_1(x_0)$, the sequence of vertices $z,T^1(z),T^2(z),\cdots T^L(z)$ is a transport geodesic. This sequence is indeed the same as $$z(0) \simeq z(1) \simeq z(2)\simeq \cdots \simeq z(L)$$ since by definition $z(j)=T^j(z)$ and $z(j-1)\simeq T_j(z(j-1))=z(j)$ for all $j$.
\end{proof}

\begin{proposition} \label{prop: ext_transgeod}
Given the same setup as in Proposition \ref{prop:transgeod_BMsharp}. Then for every $z\in B_1(x_0)$, the corresponding transport geodesic $g_z$, namely $$g_z:\quad z(0) \simeq z(1) \simeq \cdots \simeq z(L)$$ has the length
\begin{equation*}
\ell(g_z):=d(z(0),z(L))=\begin{cases}
L-1 \quad\textup{, if } z(0)=x_0 \textup{ or } z(L)=x_L\\
L-2 \quad\textup{, otherwise}.
\end{cases}
\end{equation*}
As an immediate consequence, every geodesic $g_z$ can be extended to the geodesic
$${\rm ext}(g_z):\quad x_0\simeq z(0) \simeq z(1) \simeq \cdots \simeq z(L)\simeq x_L.$$
\end{proposition}

\begin{proof}
Since $z(0)\in B_1(x_0)$ and $z(L)\in B_1(x_L)$, triangle inequality gives 
\begin{align*}
L=d(x_0,x_L) &\le d(x_0,z(0))+d(z(0),z(L))+d(z(L),x_L) \\
&=\mathbbm{1}_{\{x_0\not=z(0)\}}+\ell(g_z)+\mathbbm{1}_{\{x_L\not=z(L)\}}.
\end{align*}

Note also that $z(0)=x_0$ and $z(L)=x_L$ cannot happen simultaneously. Otherwise, it means that $\ell(g_{x_0})=L$ which would imply that the geodesic $g_{x_0}$ contains all distinct vertices $x_0(0)\sim x_0(1) \sim \cdots \sim x_0(L)$, contradicting to the repetition $x_0=x_0(0)=x_0(1)$. Therefore, 
\begin{equation} \label{eq:estimate_gz}
\ell(g_z)\ge\begin{cases}
L-1 \quad\textup{, if } z(0)=x_0 \textup{ or } z(L)=x_L\\
L-2 \quad\textup{, otherwise}.
\end{cases}
\end{equation}
and
\begin{equation} \label{eq:estimate_gz_sum}
\sum\limits_{z \in B_1(x_0)} \ell(g_z) \ge 2(L-1)+(D-1)(L-2).
\end{equation}

On the other hand, from \eqref{eq:cost_fullplan}, $T^L$ has the cost of $L - \frac{2D}{D+1}$. It follows that
\begin{align*}
\frac{1}{D+1} \sum_{z \in B_1(x_0)} \ell(g_z) 
&=\sum_{z \in B_1(x_0)} d(z,T^L(z)) \frac{1}{D+1}\\
&={\rm cost}(T^L)=L - \frac{2D}{D+1}\\
&=\frac{1}{D+1} \bigg(2(L-1)+(D-1)(L-2)\bigg).
\end{align*}
which implies that the equality holds true in \eqref{eq:estimate_gz_sum}, and also in \eqref{eq:estimate_gz} as desired.
\end{proof}

\subsection{Antipoles of intervals in a self-centered Bonnet-Myers sharp graph}

We still assume that our graph $G=(V,E)$ is a self-centered $(D,L)$-Bonnet-Myers sharp graph. Henceforth
we will use the following notation related to intervals: Given an interval
$[x,y] \subset V$ in $G$ and a vertex $z \in [x,y]$, we call a vertex
$\overline{z} \in [x,y]$ an \emph{antipole of $z$ w.r.t. $[x,y]$} if
$d(x,y) = d(z,\overline{z})$. Note that antipoles were already
introduced for graphs and this definition simply means that $z$ and
$\overline{z}$ are antipoles of the induced subgraph of $[x,y]$. We
now focus on identifying antipoles w.r.t. intervals via the method of transport
geodesics.

\begin{theorem} \label{thm:uniq_ant}
Let $G=(V,E)$ be a self-centered $(D,L)$-Bonnet-Myers sharp, and given a full-length geodesic $g$:
$$g: \quad x_0 \sim x_1 \sim \cdots \sim x_L.$$ Then for any $2\le k \le L$, $x_1$ has a unique antipole w.r.t. the interval $[x_0,x_k]$, which we will then denote as $\textup{ant}_{[x_0,x_k]}(x_1)$. In fact, we show that 
$$ \textup{ant}_{[x_0,x_k]}(x_1)=x_0(k)=T^k(x_0) \in B_1(x_k) $$ 
for any fixed $\{T_j,T^j\}_{j=1}^L$ defined in Subsection \ref{subsection: transgeod}.
\end{theorem}

\begin{proof}
First, fix a set of transport maps $\{T_j,T^j\}_{j=1}^L$ associated to $g$. Suppose that there exists $z\in [x_0,x_k]$ which is an antipole of $x_1$ w.r.t. $[x_0,x_k]$, that is $z\in [x_0,x_k]$ and $d(x_1,z)=d(x_0,x_k)=k$. Since $x_1\sim x_0$ and $d(x_1,z)=k$, we have $d(x_0,z)\ge k-1$. Since $z\in [x_0,x_k]$ and $z\not=x_k$, we must have $d(x_0,z)= k-1$ and $d(z,x_k)=1$. Since $z\in B_1(x_k)$, there is a unique $a\in B_1(x_0)$ such that $a(k)=z$, that is $a=(T^k)^{-1}(z)$ (because $T^k$ is a bijective map). By Proposition \ref{prop: ext_transgeod}, $$ x_0\simeq a(0) \simeq a(1) \simeq ... \simeq \equalto{a(k)}{z}$$ 
is part of the geodesic ${\rm ext}(g_a)$, so it is also a geodesic. Therefore it satisfies
\begin{equation} \label{eqn: x0_z}
k-1=d(x_0,z)=d(x_0,a(0))+d(a(0),a(1))+d(a(1),z).
\end{equation}
On the other hand, since $d(x_1,z)=k$ and $a(1) \in B_1(x_1)$, triangle inequality gives $d(a(1),z)\ge k-1$. Equation (\ref{eqn: x0_z}) then implies $x_0=a(0)=a(1)$, which means $a=x_0$. Thus $z=a(k)=x_0(k)$. So far we have shown that, for every $2\le k \le L$, $x_0(k)$ is the only candidate for an antipole of $x_1$ w.r.t. $[x_0,x_k]$. It remains to show that $x_0(k)$ is in fact the antipole of $x_1$ w.r.t $[x_0,x_k]$.

In particular, when $k=L$, the antipole of $x_1$ w.r.t. $[x_0,x_L]=V$ exists by the assumption that $G$ is self-centered. Denote this antipole by $\overline{x}_1$. By the previous argument, $x_0(L)$ must be $\overline{x}_1$, $d(x_1,x_0(L))=L$.

Consider the transport geodesic $g_{x_0}$:
$$g_{x_0}: \quad \equalto{x_0(0)}{x_0}= x_0(1)\simeq x_0(2) \simeq \cdots \simeq  \equalto{x_0(L)}{\overline{x}_1}$$

Since $g_{x_0}$ has length $L-1$ (by Proposition \ref{prop: ext_transgeod}), all the ``$\simeq$'' in $g_{x_0}$ must be strict ``$\sim$''. Thus $g_{x_0}$ can be written as 
$$g_{x_0}: \quad \equalto{x_0(0)}{x_0}= x_0(1)\sim x_0(2) \sim \cdots \sim  \equalto{x_0(L)}{\overline{x}_1}.$$
Moreover, since $g_{x_0}$ has length $L-1$ and $d(x_1,\overline{x}_1)=L$, the geodesic $g_{x_0}$ can then be extended (by adding $x_1$ to the left) to another geodesic $g'$: $$g': \quad x_1\sim \equalto{x_0(0)}{x_0}= x_0(1)\sim x_0(2) \sim \cdots \sim  \equalto{x_0(L)}{\overline{x}_1}.$$ 
Consequently, we can read off from the geodesic $g'$ that for every $k\in\{2,...,L\}$
\begin{enumerate}
	\item $d(x_1,x_0(k))=k$,
	\item $x_0(k)\in [x_0,x_k]$, because $k=d(x_0,x_k)\le d(x_0,x_0(k))+d(x_0(k),x_k)\le (k-1)+1$.
\end{enumerate}
Therefore, $x_0(k)$ is the unique antipole of $x_1$ w.r.t. $[x_0,x_k]$ as desired.
\end{proof}

Let us first discuss an immediate consequence of Theorem
\ref{thm:uniq_ant}. Note that the theorem implies that there is a
well-defined antipole map
$$ \textup{ant}_{[x,y]}: [x,y] \cap B_1(x) \to [x,y] \cap B_1(y). $$
Existence and uniqueness of antipoles for neighbours of $x$
w.r.t. $[x,y]$ implies the following result:

\begin{corollary} \label{cor: ant_biject} Let $G = (V,E)$ be a
  self-centered Bonnet-Myers sharp graph, $x, y\in V$ be two different
  vertices. Then the antipole map
  $$ \textup{ant}_{[x,y]}: [x,y] \cap B_1(x) \to [x,y] \cap B_1(y) $$ 
  is bijective and, consequently,
  $$ \left| [x,y] \cap B_1(x)\right| = \left| [x,y] \cap B_1(y) \right|. $$ 
\end{corollary}

\begin{rem} \label{rem:ant_toggle}
Let $x,y \in V$ be two different vertices and $x' \in [x,y] \cap B_1(x)$
with its antipole $y' = \textup{ant}_{[x,y]}(x')$. Observe that then
$x,y \in [x',y']$ and $y = \textup{ant}_{[x',y']}(x)$. 
\end{rem}

Another immediate consequence of Theorem \ref{thm:uniq_ant} is the following corollary.
\begin{corollary} \label{cor: mu_CP}
Let $G = (V,E)$ be a self-centered Bonnet-Myers sharp graph. Then all $\mu$-graphs of $G$ are cocktail party graphs. 
\end{corollary}

\begin{proof}
  Let $x,y \in V$ with $d(x,y)=2$ and $z \in N_{xy}$. Since $G$ is
  self-centered Bonnet-Myers sharp, $x=x_0$ has an antipole $x_L \in V$, and we can find a geodesic $g$ from $x_0$ to $x_L$ passing through $z=x_1$ and $y=x_2$ by Theorem \ref{lieongeos}:
$$ g: \quad \equalto{x_0}{x} \sim \equalto{x_1}{z} \sim \equalto{x_2}{y} \sim \cdots \sim x_L. $$
Applying Theorem \ref{thm:uniq_ant} with $k=2$, we conclude that there
is a unique $z' \in N_{xy}$ satisfying $d(z,z')=2$. This shows that the
$\mu$-graph of $x$ and $y$ is a cocktail party graph.
\end{proof}

\begin{rem}
  The fact that all $\mu$-graphs of $G$ are cocktail party graphs
  allows us to naturally introduce a \emph{switching map}, defined as
  follows. Consider a pair $x,y\in V$ with $d(x,y)=2$. Then the
  switching map $\sigma_{xy}:N_{xy}\to N_{xy}$ is defined by
  $\sigma_{xy}(z):={\rm ant}_{[x,y]}(z)$ and satisfies
  $\sigma_{xy}^2 = {\rm Id}_{N_{xy}}$.
\end{rem}

\begin{rem} \label{rem: pm_uniq} Recall from Remark \ref{rem: pm_not_uniq}
  that for general Bonnet-Myers sharp graphs the perfect matchings
  defining the maps $T_j$ are not necessarily unique. However, under the additional condition of self-centeredness, the fact that all $\mu$-graphs of $G$ are
  cocktail party graphs implies the uniqueness of these perfect matchings and the
  associated transport maps $T_j$. Therefore, the definition of a
  transport geodesic $g_z$ depends only on the geodesic
  $g$.

  In particular, the transport geodesic $g_{x_0}$ containing all antipoles
  of $x_1$ w.r.t. increasing intervals $[x_0,x_k]$ (see Theorem \ref{thm:uniq_ant}) can be also understood 
  as been generated via the following recursive process of switching maps,
  as illustrated in Figure \ref{fig:transport_geodesic}:
  \begin{eqnarray*}
  x_0(2) &= \sigma_{x_0x_2}(x_1), \\
  x_0(3) &= \sigma_{x_0(2)x_3}(x_2), \\
  &\vdots \\
  x_0(k) &= \sigma_{x_0(k-1)x_k}(x_{k-1}), \\
  & \vdots \\
  x_0(L) &= \sigma_{x_0(L-1)x_L}(x_{L-1}). 
  \end{eqnarray*} 
\end{rem}

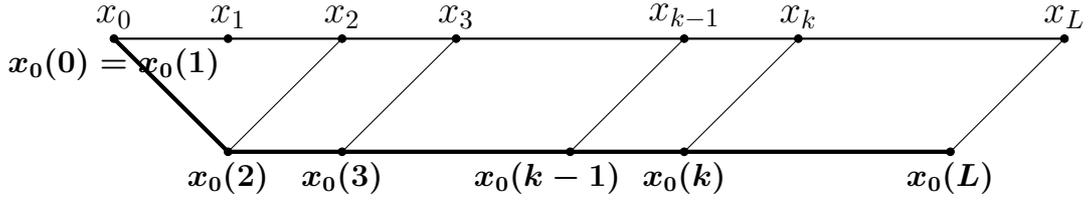
\begin{figure}[h!]
	\centering
	\begin{tikzpicture}[scale=1]
	
	\draw[thick] (0,0)--(12.5,0);
	\draw[ultra thick] (0,0)--(1.5,-1.5)--(11,-1.5);
	\draw[](11,-1.5)--(12.5,0) ;
	\draw[] (3,0)--(1.5,-1.5);
	\draw[] (4.5,0)--(3,-1.5);
	\draw[] (7.5,0)--(6,-1.5);	
	\draw[] (9,0)--(7.5,-1.5);	
	
	\draw[fill] (0,0) circle [radius=0.05];
	\node[above] at (0,0) {\Large $x_0$};	
	\node[below] at (0,0) {\large $\boldsymbol {x_0(0)=x_0(1)}$};	
	
	\draw[fill] (1.5,0) circle [radius=0.05];
	\node[above] at (1.5,0) {\Large $x_1$};
	
	\draw[fill] (3,0) circle [radius=0.05];
	\node[above] at (3,0) {\Large $x_2$};
	
	\draw[fill] (4.5,0) circle [radius=0.05];
	\node[above] at (4.5,0) {\Large $x_3$};
	
	\draw[fill] (7.5,0) circle [radius=0.05];
	\node[above] at (7.5,0) {\Large $x_{k-1}$};
	
	\draw[fill] (9,0) circle [radius=0.05];
	\node[above] at (9,0) {\Large $x_{k}$};	
	
	\draw[fill] (12.5,0) circle [radius=0.05];
	\node[above] at (12.5,0) {\Large $x_L$};
	
	\draw[fill] (1.5,-1.5) circle [radius=0.05];
	\node[below] at (1.5,-1.5) {\large $\boldsymbol{x_0(2)}$};
	
	\draw[fill] (3,-1.5) circle [radius=0.05];
	\node[below] at (3,-1.5) {\large $\boldsymbol{x_0(3)}$};
	
	\draw[fill] (6,-1.5) circle [radius=0.05];
	\node[below] at (5.7,-1.5) {\large $\boldsymbol{x_0(k-1)}$};
	
	\draw[fill] (7.5,-1.5) circle [radius=0.05];
	\node[below] at (7.5,-1.5) {\large $\boldsymbol{x_0(k)}$};
	
	\draw[fill] (11,-1.5) circle [radius=0.05];
	\node[below] at (11,-1.5) {\large $\boldsymbol{x_0(L)}$};
	
	\end{tikzpicture}
	\caption
	{Transport geodesic $g_{x_0}$ (along $g$) shown in \textbf{bold}} 
	\label{fig:transport_geodesic}
\end{figure}

\section{Self-centered Bonnet-Myers sharp implies strongly spherical}
\label{sec:antBMstrspher}

The ultimate goal of this section is to prove that all self-centered
Bonnet-Myers sharp graphs are strongly spherical (as stated in Theorem
\ref{thm_spherical} below). Let us recall the definition of self-centeredness,
antipodal, and strongly spherical (introduced in Definition \ref{def:strspher} and in Subsection \ref{sec:graph_notation}). For a finite connected graph $G=(V,E)$:\\
  $\bullet$ $G$ is \text{self-centered} if for every $x\in V$ there exists $\overline{x}\in V$ such that $d(x,\overline{x})=\textup{diam}(G).$\\
  $\bullet$ $G$ is \text{antipodal} if for every $x\in V$ there exists $\overline{x}\in V$ such that $[x,\overline{x}]=V$. The vertex $\overline{x}$ is then called an \emph{antipode} of $x$.\\
  $\bullet$ $G$ is \text{strongly spherical} if $G$ is antipodal, and the induced subgraph of every interval of $G$ is antipodal.

\begin{rem} It is important to notice the distinction between the notions ``antipole'' and ``antipode''. Here are basic facts about antipodes:
\begin{itemize}
\item Antipodes are also antipoles: Let $\overline{x}$ be an antipode
  of $x$ in $G$, that is, $[x,\overline{x}]=V$. We choose arbitrary
  $y,z \in V$ such that $d(y,z) = \textup{diam}(G)$. Then we have by
  $y,z \in [x,\overline{x}]$ and the triangle inequality
\begin{multline*}
  \textup{diam}(G) \ge d(x,\overline{x}) = \frac{1}{2} (d(x,y) +
  d(y,\overline{x})) +
  \frac{1}{2} (d(x,z) + d(z,\overline{x})) \\
  = \frac{1}{2} (d(y,x) + d(x,z)) + \frac{1}{2} (d(y,\overline{x}) +
  d(\overline{x},z)) \ge d(y,z) = \textup{diam}(G).
\end{multline*}
\item Antipodes are necessarily unique: Assume $\overline{x}_1$ and $\overline{x}_2$
are antipodes of $x$. Then $\overline{x}_2$ lies on a geodesic
from $x$ to $\overline{x}_1$. Since $d(x,\overline{x}_1)=d(x,\overline{x}_2)$,
this implies $\overline{x}_1=\overline{x}_2$.
\end{itemize}
\end{rem}

For the reader's convenience, let us start with a brief overview of
the proof that self-centered Bonnet-Myers sharp graphs are strongly
spherical. Note first that every self-centered Bonnet-Myers sharp graph
coincides with the interval of any pair of antipoles
(by Theorem \ref{lieongeos}). Therefore, it suffices to prove that every interval $[x,y]$, $x,y \in V$, of a self-centered Bonnet-Myers sharp graph
$G=(V,E)$ is antipodal.  This proof is divided into the
following four steps:

\begin{description}
\item[Step 1:] Let $x' \in [x,y] \cap S_1(x)$ with antipole
  $y' = \textup{ant}_{[x,y]}(x')$. We prove for every $z \in [x,y] \cap B_1(x)$
that $z \in [x',y']$ (see Theorem \ref{thm_spherical_step1}).
\item[Step 2:] Let $x' \in [x,y] \cap S_1(x)$ with antipole
  $y' = \textup{ant}_{[x,y]}(x')$. We prove for every $z \in [x,y]$
that $z \in [x',y']$ (see Theorem \ref{thm_spherical_step2}).
\item[Step 3:] Let $x' \in [x,y] \cap S_1(x)$ with antipole
  $y' = \textup{ant}_{[x,y]}(x')$. We prove that $[x,y] = [x',y']$ (see
Corollary \ref{cor_clockwise_spherical}).
\item[Step 4:] Let $x' \in [x,y]$. We prove that there exists $y' \in [x,y]$
such that $[x,y] = [x',y']$ (see Theorem \ref{thm_spherical}).
\end{description}

Let us now start to prove each of these steps in order. Recall that
the existence of antipoles of vertices in $[x,y] \cap B_1(x)$ w.r.t. $[x,y]$ is
guaranteed by Corollary \ref{cor: ant_biject}.

\begin{theorem}\label{thm_spherical_step1}
Let $G=(V,E)$ be self-centered Bonnet-Myers sharp. Let $x,y \in V$ be two different vertices, and consider any $x'\in [x,y]\cap S_1(x)$ with its antipole $y'=\textup{ant}_{[x,y]}(x')$. Then every $z \in [x,y]\cap B_1(y)$ satisfies $z\in[x',y']$.
\end{theorem}

We start with the set-up and introduce particular sets $A,A_1,A_2,Z,Z_1,Z_2$ and a function $F$ which will be important for the proof of the above theorem.

Let $k=d(x,y)$. We re-label the vertices as $x=x_0$ and $y=x_k$ and
$x'=x_1$ and $y'=\overline{x}_1$, as illustrated in Figure \ref{fig:intervalxy}. Keep in mind that
$\overline{x}_1=\textup{ant}_{[x_0,x_k]}(x_1)$ and $x_0\sim x_1$ and
$x_k\sim \overline{x}_1$.

\begin{figure}[h!]
	\centering
	\begin{tikzpicture}[scale=1]
		
	\draw[thick] (0,0) to [out=80,in=180] (4,2);
	\draw[thick] (0,0) to [out=-80,in=180] (4,-2);
	\draw[thick] (4,2) to [out=0,in=100] (8,0);
	\draw[thick] (4,-2) to [out=0,in=-100] (8,0);
	\draw[dashed] (0,0) to [out=10,in=175] (6,0.3);
	\draw[thick] (6,0.3) to [out=-5,in=170] (8,0);
	
	\draw[fill] (0,0) circle [radius=0.05];
	\node[left] at (0,0) {\LARGE $x$};
	\node[left] at (0.2,-0.5) {\Large $\boldsymbol{x_0}$};
	
	\draw[fill] (8,0) circle [radius=0.05];
	\node[right] at (8,0) {\LARGE $y$};
	\node[right] at (8,-0.5) {\Large $\boldsymbol{x_k}$};
	
	\draw[fill] (1.5,1.65) circle [radius=0.05];
	\node[left] at (1.5,1.75) {\LARGE $x'$};
	\node[left] at (1.6,1.15) {\Large $\boldsymbol{x_1}$};
	
	\draw[fill] (6.5,-1.65) circle [radius=0.05];
	\node[below right] at (6.5,-1.65) {\Large $y'=\textup{ant}_{[x,y]}(x')$};
	\node[below right] at (6.5,-2.35) {\Large $\boldsymbol{\overline{x}_1=\textup{\bf ant}_{[x_0,x_k]}(x_1)}$};
	
	\draw[fill] (6,0.3) circle [radius=0.05];
	\node[above] at (6,0.3) {\LARGE $z$};
	
	
	\end{tikzpicture}
	\caption
	{The interval $[x,y]$ with the re-labelled vertices in \textbf{bold}, and $z\in [x_0,x_k]\cap B_1(x_k)$.
	} 
	\label{fig:intervalxy}
\end{figure}
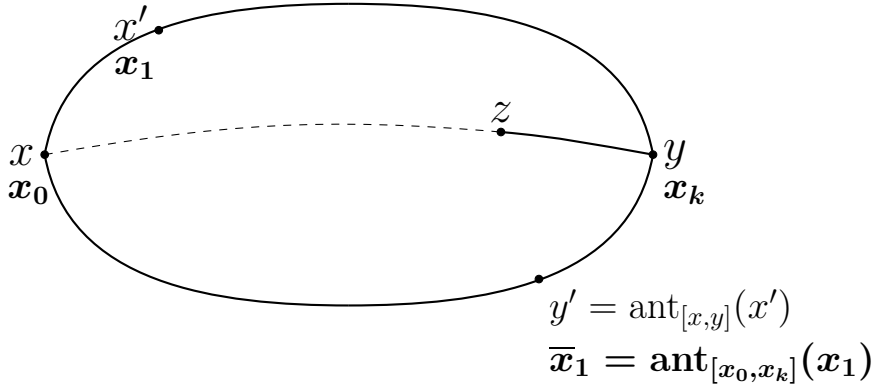
We define the following sets 
\begin{align*}
& A:=[x_0,x_k]\cap B_1(x_0), && Z:=[x_0,x_k]\cap B_1(x_k),\\
& A_1:=A\cap S_1(x_1)\setminus\{x_0\}, && Z_1:=Z\cap S_1(\overline{x}_1)\setminus\{x_k\},\\
& A_2:=A\cap S_2(x_1), && Z_2:=Z\cap S_2(\overline{x}_1).
\end{align*}
Note that the sets $A$ and $Z$ can be partitioned into $$A=\{x_0,x_1\}\sqcup A_1 \sqcup A_2 \qquad \textup{and} \qquad Z=\{x_k,\overline{x}_1\}\sqcup Z_1 \sqcup Z_2.$$

Now fix an arbitrary full-length geodesic $g$ from $x_0$ to $x_L$ (the antipole of $x_0$) which passes through $x_1$ and $x_k$ (this can be done since $x_1\in [x_0,x_k]$ and $x_k\in [x_0,x_L]$ by Theorem \ref{lieongeos}), namely
$$g: \quad x_0\sim x_1 \sim x_2 \sim \cdots \sim x_k \sim x_{k+1} \sim \cdots \sim x_L.$$ 
Consider the transport map $T^k:B_1(x_0)\rightarrow B_1(x_k)$ introduced in
Subsection \ref{subsection: transgeod}. Recall that $T^k$ is
bijective. Then define a function $F:Z\rightarrow A$ to be $F(z):=(T^k)^{-1}(z)$ for all $z\in Z\subset B_1(x_k)$. Lemma
\ref{lemma_FZ} below guarantees that $F(Z) \subseteq A$, hence $F$ is well-defined.

In order to conclude Theorem \ref{thm_spherical_step1}, we need to
prove that $\forall z\in Z:\ z\in [x_1,\overline{x}_1]$, which is
divided into Lemma \ref{lemma_FZ2} (dealing with the case
$z \in Z_2 \sqcup \{x_k,\overline{x_1}\}$) and Lemma \ref{lemma_FZ1}
(dealing with the case $z \in Z_1$).

\begin{lemma} \label{lemma_FZ}
$F(Z) \subseteq A$ and $F: Z \to A$ is bijective.
\end{lemma}

\begin{lemma} \label{lemma_FZ2}
$F(Z_2\sqcup\{x_k,\overline{x}_1\})=A_2\sqcup \{x_0,x_1\}$ and $\forall z\in Z_2\sqcup\{x_k,\overline{x}_1\}:\ z\in [x_1,\overline{x}_1]$.
\end{lemma}

\begin{lemma} \label{lemma_FZ1}
$F(Z_1)=A_1$ and $\forall z\in Z_1:\ z\in [x_1,\overline{x}_1]$.
\end{lemma}

Now we will prove the above three lemmas in order, and then conclude Theorem \ref{thm_spherical_step1}.

\begin{proof}[Proof of Lemma \ref{lemma_FZ}]
First we show that $F(z)\in A$ for all $z\in Z$. Let $a:=F(z)\in B_1(x_0)$, that is $a=a(0)$ and $z=a(k)$. By Proposition \ref{prop: ext_transgeod} we know that $$ x_0\simeq a(0) \simeq a(1) \simeq ... \simeq a(k)$$ is a geodesic (as a part of ${\rm ext}(g_a)$). Moreover, since $a(k)=z\in [x_0,x_k]$, this geodesic can be extended to another geodesic $\gamma$, namely 
\begin{equation} \label{eq:geodgamma}
\gamma:\quad  x_0\simeq \equalto{a(0)}{a} \simeq a(1) \simeq ... \simeq \equalto{a(k)}{z}\simeq x_k.
\end{equation}
Therefore, $a$ must lie in the interval $[x_0,x_k]$, which means $a\in A$
and we have $F(Z) \subseteq A$.   

Next, note that the function $F:Z\rightarrow A$, which is a
restriction of $(T^k)^{-1}$, must be injective (because $T^k$ is
bijective). Note also that $|A|=|Z|$ because of Corollary \ref{cor:
  ant_biject}. Therefore, $F$ must be bijective.
\end{proof}


\begin{proof}[Proof of Lemma \ref{lemma_FZ2}]

  A main feature of the following proof is to show
  $A_2 \sqcup \{x_0,x_1\} \subseteq F(Z_2 \sqcup
  \{x_k,\overline{x}_1\})$. For that reason we start with an element
  $a\in A_2\sqcup \{x_0,x_1\}$. Then there exists a uniqe $z \in Z$ with
$F(z)=a$. Consequently, $z=a(k)$ and $z \in [x_0,x_k]$. Consider the following two cases.

\smallskip

\underline{Case $a=x_0$:} From Theorem \ref{thm:uniq_ant},  we have $a=x_0=(T^k)^{-1}(\overline {x}_1)=F(\overline {x}_1)$, so $z=\overline {x}_1$ and
$z \in [x_1,\overline{x}_1]$.

\smallskip

\underline{Case $a\in A_2 \sqcup \{x_1\}$:} As in the proof of
Lemma \ref{lemma_FZ}, we have the following geodesic $\gamma$ of length $k$ (referred to the one in \eqref{eq:geodgamma}):
$$ \gamma: \quad  x_0\sim \equalto{a(0)}{a} \simeq a(1) \simeq ... \simeq \equalto{a(k)}{z} \simeq x_k.
$$
From this geodesic $\gamma$ and an observation that
$$ d(x_0,a(1)) = \begin{cases} 1, & \text{if $a=x_1$ (and therefore also $a(1)=x_1$),} \\
  2, & \text{if $a \in A_2$ (and therefore $a(1)\neq a(0)$),} \end{cases} $$ we conclude
\begin{equation} \label{eq:da1xk} 
d(a(1),x_k) = \begin{cases} k-1, & \text{if $a=x_1$,} \\
k-2, & \text{if $a \in A_2$.} \end{cases} 
\end{equation}
Now we extend the geodesic
$$ a(1) \simeq ... \simeq \equalto{a(k)}{z} \simeq x_k
$$
to
$$ x_1 \simeq a(1) \simeq ... \simeq \equalto{a(k)}{z} \simeq x_k \sim \overline{x}_1, 
$$
which is, again a geodesic because of \eqref{eq:da1xk} (and recall that $d(x_1,\overline{x}_1) = k$). We can then read off from the above geodesic that $z = x_k$ or $z \in S_2(\overline{x}_1) \cap [x_0,x_k] = Z_2$ and $z \in [x_1,\overline{x}_1]$.

We conclude from both cases that $A_2\sqcup \{x_0,x_1\}\subseteq F(Z_2\sqcup\{x_k,\overline{x}_1\})$. Since $F$ is bijective, it follows that $|A_2| \le |Z_2|$. 
By switching the roles between $x_0$ and $x_k$ and between the antipoles $x_1$ and $\overline{x}_1$ w.r.t. $[x_0,x_k]$, we obtain the opposite inequality $|Z_2| \le |A_2|$. Therefore, we have $|Z_2| = |A_2|$, and thus $A_2\sqcup \{x_0,x_1\} = F(Z_2\sqcup\{x_k,\overline{x}_1\})$, as desired. 

Consequently, if we consider any $z\in Z_2\sqcup\{x_k,\overline{x}_1\}$, then $a\in A_2\sqcup \{x_0,x_1\}$ falls into one of the above cases, in which we have shown $z\in [x_1,\overline{x}_1]$.
\end{proof}

\begin{proof}[Proof of Lemma \ref{lemma_FZ1}]
  Since $F: Z \to A$ is bijective and
  $F(Z_2\sqcup\{x_k,\overline{x}_1\}) = A_2\sqcup \{x_0,x_1\}$, we
  conclude $F(Z_1) = A_1$.

Moreover, consider $z\in Z_1$. It follows that $z\sim \overline{x}_1$ and $d(x_1,z)\le k-1$, because $z\not=\overline{x}_1=\textup{ant}_{[x_0,x_k]}(x_1)$. Therefore
$$d(x_1,z)+d(z,\overline{x}_1) = d(x_1,z)+1 \le (k-1)+1=k,$$ which means $z\in[x_1,\overline{x}_1]$.
\end{proof}

\begin{proof}[Proof of Theorem \ref{thm_spherical_step1}]
  Recalling the original set-up and notation, we only need to show
  that $z \in [x_1,\overline{x}_1]$.  This follows immediately from
  Lemma $\ref{lemma_FZ2}$ and Lemma \ref{lemma_FZ1}.
\end{proof}

The next theorem generalized Theorem \ref{thm_spherical_step1} by removing the restriction $z\in B_1(y)$.

\begin{theorem}\label{thm_spherical_step2}
Let $G=(V,E)$ be self-centered Bonnet-Myers sharp. Let $x,y \in V$ be two different vertices, and consider any $x'\in [x,y]\cap S_1(x)$ with its antipole $y'=\textup{ant}_{[x,y]}(x')$. Then every $z \in [x,y]$ satisfies $z\in[x',y']$.
\end{theorem}

\begin{proof}[Proof of Theorem \ref{thm_spherical_step2}]
Let $d_1=d(x,y)$ and $d_2=d(z,y)$ (note that $0\le d_2\le d_1$). We will prove the statement of the theorem by induction on $d_1$ and $d_2$.

\smallskip

\underline{Base step:} For any value of $d_1$, the cases $d_2=0,1$ are both covered by Theorem \ref{thm_spherical_step1}.

\smallskip

\underline{Inductive step:} Assume that the statement is true for $d_1=k-1$ and all $d_2$, and assume that the statement is true for $d_1=k$ and $d_2=j-1$ for some $2\le j\le k-1$. Now consider $d(x,y)=k$ and $z\in [x,y]\cap S_j(y)$. Choose an arbitrary $z_1\in [z,y]\cap S_1(y)$. Hence $x,z,z_1,y$ lies in a geodesic , see Figure \ref{fig:intervalxy_induction}. 
In particular, $z\in [x,z_1]$.

\begin{figure}[h!]
	\centering
	\begin{tikzpicture}[scale=1]
	
	\draw[thick] (0,0) to [out=80,in=180] (4,2);
	\draw[thick] (0,0) to [out=-80,in=180] (4,-2);
	\draw[thick] (4,2) to [out=0,in=100] (8,0);
	\draw[thick] (4,-2) to [out=0,in=-100] (8,0);
	\draw[dashed] (0,0) to [out=10,in=175] (6,0.3);
	\draw[thick] (4,0.4) to [out=0,in=170] (8,0);
	
	\draw[fill] (0,0) circle [radius=0.05];
	\node[left] at (0,0) {\LARGE $x$};
	
	\draw[fill] (8,0) circle [radius=0.05];
	\node[right] at (8,0) {\LARGE $y$};
	
	\draw[fill] (0.8,1.3) circle [radius=0.05];
	\node[above] at (0.8,1.3) {\LARGE $x'$};
	
	\draw[fill] (7.2,-1.3) circle [radius=0.05];
	\node[right] at (7.2,-1.4) {\LARGE $y'$ \Large $=\textup{ant}_{[x,y]}(x')$};
	
	\draw[fill] (4,0.4) circle [radius=0.05];
	\node[above] at (4,0.4) {\LARGE $z$};
	
	\draw[fill] (7,0.18) circle [radius=0.05];
	\node[above] at (7,0.18) {\LARGE $z_1$};
	
	
	\end{tikzpicture}
	\caption
	{The interval $[x,y]$ and $z\in [x,y]$ with $d(z,y)=j\ge2$ and $z_1\in [z,y]\cup S_1(y)$.
	} 
	\label{fig:intervalxy_induction}
\end{figure}
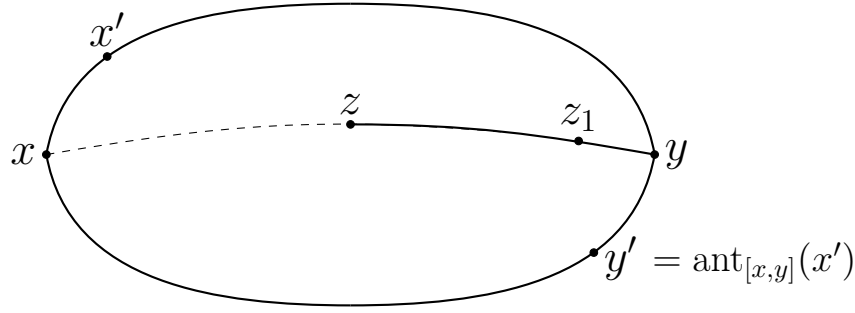

Now consider the following three cases whether $d(z_1,y')$ is 0, 1, or 2.

\smallskip

\underline{Case $z_1=y'$:} It follows immediately that $z\in [x,z_1]=[x,y'] \subseteq [x',y']$ where the last inclusion is due to $x\in[x',y']$.

\smallskip

\underline{Case $z_1 \sim y'$:} Since $z_1 \in [x,y]$, by Theorem \ref{thm:uniq_ant} there is a unique $a_1=\textup{ant}_{[x,y]}(z_1) \in [x,y]$. Since $a_1 \in [x,y] \cap B_1(x)$ and
$z_1 \in [x,y] \cap B_1(y)$, by Theorem \ref{thm_spherical_step1}, 
$z_1, a_1 \in[x',y']$.

The fact that $a_1,z_1\in[x',y']$ and that $d(a_1,z_1)=d(x,y)=d(x',y')$ altogether implies that $a_1$ must be the unique antipole $\textup{ant}_{[x',y']}(z_1)$ by Corollary \ref{cor: ant_biject} since $z_1 \sim y'$. This is illustrated in Figure
\ref{fig:induction_case2}. By Remark \ref{rem:ant_toggle}, it implies that $y'=\textup{ant}_{[a_1,z_1]}(x')$.

Observe also that $z\in [x,z_1]\subset [a_1,z_1]$ with $d(z,z_1)=j-1$.

We are now in a position to apply the induction hypothesis for the
interval $[a_1,z_1]$ (instead of $[x,y]$) and $z \in [a_1,z_1]$ with
$d(z,z_1)=j-1$. Note that $d(a_1,z_1) = k$. Note also that
$x' \in [a_1,z_1] \cap S_1(a_1)$ and
$y' = \textup{ant}_{[a_1,z_1]}(x')$. Then the induction hypothesis
implies $z \in [x',y']$, finishing this case.



\begin{figure}[h!]
	\centering
	\begin{tikzpicture}[scale=1]
	
	\draw[] (0,0) to [out=80,in=180] (4,2);
	\draw[] (0,0) to [out=-80,in=180] (4,-2);
	\draw[] (4,2) to [out=0,in=100] (8,0);
	\draw[] (4,-2) to [out=0,in=-100] (8,0);
	\draw[dashed] (0,0) to [out=10,in=175] (6,0.3);
	\draw[] (4,0.4) to [out=0,in=170] (8,0);
	\draw[] (0,0) to [out=-30,in=175] (1, -0.3);
	
	\draw[very thick] (0.8,1.3) to [out=-120,in=135] (1, -0.3) to [out=-40,in=-160] (7.2,-1.3) to [out=60,in=-45] (7,0.18) to [out=135,in=20] (0.8,1.3);
	
	\draw[fill] (0,0) circle [radius=0.05];
	\node[left] at (0,0) {\LARGE $x$};
	
	\draw[fill] (8,0) circle [radius=0.05];
	\node[right] at (8,0) {\LARGE $y$};
	
	\draw[fill] (0.8,1.3) circle [radius=0.05];
	\node[above] at (0.8,1.3) {\LARGE $x'$};
	
	\draw[fill] (7.2,-1.3) circle [radius=0.05];
	\node[right] at (7.2,-1.4) {\LARGE $y'$ \Large $=\textup{ant}_{[x,y]}(x')$};
	
	\draw[fill] (4,0.4) circle [radius=0.05];
	\node[above] at (4,0.4) {\LARGE $z$};
	
	\draw[fill] (7,0.18) circle [radius=0.05];
	\node[above] at (7,0.18) {\LARGE $z_1$};
	
	\draw[fill] (1,-0.3) circle [radius=0.05];
	\node[below right] at (1,-0.3) {\LARGE $a_1$\large $=\textup{ant}_{[x,y]}(z_1)$};
	
	
	\end{tikzpicture}
	\caption
	{Picture for Case $z_1\sim y'$. The \textbf{bold} cycle represents the fact that $a_1$ and $z_1$ are antipoles w.r.t. not only $[x,y]$ but also $[x',y']$.}
	\label{fig:induction_case2}
\end{figure}
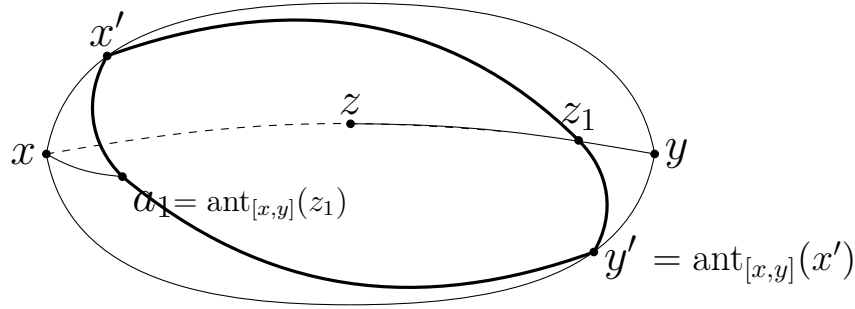

\smallskip

\underline{Case $d(z_1, y')=2$:} 
Since $z_1 \in [x,y] \cap B_1(y)$, by Theorem \ref{thm_spherical_step1}, we have $z_1\in[x',y']$. The condition $d(z_1, y')=2$ then implies that $d(x',z_1)=d(x',y')-2=k-2$. 
It follows that $$d(x,x')+d(x',z_1)+d(z_1,y)=1+(k-2)+1=k=d(x,y)$$ which means
that $x'$ and $z_1$ lie on a geodesic from $x$ to $y$. Let us denote this
geodesic by $g^*$:
$$g^*: \quad x\sim x'\sim\cdots\sim z_1\sim y.$$ 
In particular, $x'\in [x,z_1]$. Then $y'':=\textup{ant}_{[x,z_1]}(x')$ exists by
Corollary \ref{cor: ant_biject}. The situation is illustrated in Figure \ref{fig:induction_case3}.

\begin{figure}[h!]
	\centering
	\begin{tikzpicture}[scale=1]
	
	\draw[] (0,0) to [out=80,in=180] (4,2);
	\draw[] (0,0) to [out=-80,in=180] (4,-2);
	\draw[] (4,2) to [out=0,in=100] (8,0);
	\draw[] (4,-2) to [out=0,in=-100] (8,0);
	\draw[dashed] (0,0) to [out=10,in=175] (6,0.3);
	\draw[] (4,0.4) to [out=0,in=170] (8,0);
	
	\draw[very thick] (0,0) to [out=80,in=-140] (0.8,1.3) to [out=20,in=135] (7,0.18) to [out=-90,in=30] (6.2,-0.8) to [out=-170,in=-40] (0,0);
	
	\draw[fill] (0,0) circle [radius=0.05];
	\node[left] at (0,0) {\LARGE $x$};
	
	\draw[fill] (8,0) circle [radius=0.05];
	\node[right] at (8,0) {\LARGE $y$};
	
	\draw[fill] (0.8,1.3) circle [radius=0.05];
	\node[above] at (0.8,1.3) {\LARGE $x'$};
	
	\draw[fill] (7.2,-1.3) circle [radius=0.05];
	\node[right] at (7.2,-1.5) {\Large $y'=\textup{ant}_{[x,y]}(x')$};
	
	\draw[fill] (4,0.4) circle [radius=0.05];
	\node[above] at (4,0.4) {\LARGE $z$};
	
	\draw[fill] (7,0.18) circle [radius=0.05];
	\node[above] at (7,0.18) {\LARGE $z_1$};
	
	\draw[fill] (6.2,-0.8) circle [radius=0.05];
	\node[right] at (6.2,-0.8) {\Large $y''$};
	
	
	\end{tikzpicture}
	\caption
	{Picture for Case $d(z_1,y')=2$. The \textbf{bold} cycle represents $y''=\textup{ant}_{[x,z_1]}(x')$. } 
	\label{fig:induction_case3}
\end{figure}
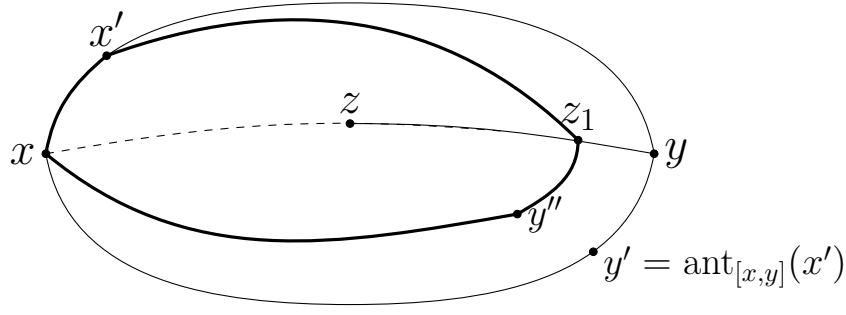

Next we apply the induction hypothesis for the
interval $[x,z_1]$ (instead of $[x,y]$) and $z \in [x,z_1]$ with
$d(z,z_1)=j-1$. Note that $d(x,z_1) = k-1$. Note also that
$x' \in [x,z_1] \cap S_1(x)$ and
$y'' = \textup{ant}_{[x,z_1]}(x')$. Then the induction hypothesis
implies $z \in [x',y'']$.

So far we have that $d(x',z) + d(z,y'') = d(x',y'') = d(x,z_1) = k-1$. It remains to show that $d(y'',y') = 1$ which would imply
$$ k = d(x',y') \le d(x',z) + d(z,y'') + d(y'',y') = (k-1)+1 = k, $$
that is $z \in [x',y']$, as desired.  

To prove $d(y'',y')=1$ we use transport geodesic
techniques. Therefore, we relabel the vertices of the geodesic $g^*$
and extend $g^*$ to a full-length geodesic $g$ in $G$ starting from $x=x_0$ 
as follows:
$$ g: \quad \equalto{x}{x_0}\sim \equalto{x'}{x_1}\sim \cdots \sim \equalto{z_1}{x_{k-1}}\sim \equalto{y}{x_k} \sim x_{k+1} \sim \cdots \sim x_L,$$ 
and consider the transport geodesic along $g$ starting at $x_0$.

Theorem \ref{thm:uniq_ant} guarantees that
$x_0(m)=\textup{ant}_{[x_0,x_m]}(x_1)$ for all $2\le m\le L$. In
particular, we have $y''= \textup{ant}_{[x_0,x_{k-1}]}(x_1) = x_0(k-1)$
and $y' = \textup{ant}_{[x_0,x_k]}(x_1) = x_0(k)$. Therefore, $y'' = x_0(k-1)$
and $y' = x_0(k)$ must be adjacent vertices (as illustrated in Figure \ref{fig:transport_geodesic_g}), thus completing the proof.

\begin{figure}[h!]
	\centering
	\begin{tikzpicture}[scale=1]
	
		\draw[thick] (0,0)--(13.5,0);
	\draw[ultra thick] (0,0)--(1.5,-1.5)--(7.5,-1.5);
	\draw[dotted] (7.5,-1.5)--(12,-1.5);
	\draw[dotted](12,-1.5)--(13.5,0) ;
	\draw[] (3,0)--(1.5,-1.5);
	\draw[] (4.5,0)--(3,-1.5);
	\draw[] (7.5,0)--(6,-1.5);	
	\draw[] (9,0)--(7.5,-1.5);	
	
	\draw[fill] (0,0) circle [radius=0.05];
	\node[above] at (0,0) {\Large $x_0$};	
	
	\draw[fill] (1.5,0) circle [radius=0.05];
	\node[above] at (1.5,0) {\Large $x_1$};
	
	\draw[fill] (3,0) circle [radius=0.05];
	\node[above] at (3,0) {\Large $x_2$};
	
	\draw[fill] (4.5,0) circle [radius=0.05];
	\node[above] at (4.5,0) {\Large $x_3$};
	
	\draw[fill] (7.5,0) circle [radius=0.05];
	\node[above] at (7.5,0) {\Large $x_{k-1}$};
	
	\draw[fill] (9,0) circle [radius=0.05];
	\node[above] at (9,0) {\Large $x_{k}$};	
	
	\draw[fill] (13.5,0) circle [radius=0.05];
	\node[above] at (13.5,0) {\Large $x_L$};
	
	\draw[fill] (1.5,-1.5) circle [radius=0.05];
	\node[below] at (1.5,-1.5) {\large $\bf x_0(2)$};
	
	\draw[fill] (3,-1.5) circle [radius=0.05];
	\node[below] at (3,-1.5) {\large $\bf x_0(3)$};
	
	\draw[fill] (6,-1.5) circle [radius=0.05];
	\node[below] at (5.7,-1.5) {\large $\boldsymbol{\equalto{x_0(k-1)}{y''}}$};
	
	\draw[fill] (7.5,-1.5) circle [radius=0.05];
	\node[below] at (7.5,-1.5) {\large $\boldsymbol{\equalto{x_0(k)}{y'}}$};
	
	\draw[] (0,0)--(10.5,0);
	
	\draw[] (0,0)--(1.5,-1.5)--(7.5,-1.5);

	\end{tikzpicture}
	\caption
	{Transport geodesic $g_{x_0}$ and the antipoles of $x_1$ w.r.t. increasing intervals $[x_0,x_m]$.}
	\label{fig:transport_geodesic_g}
\end{figure}
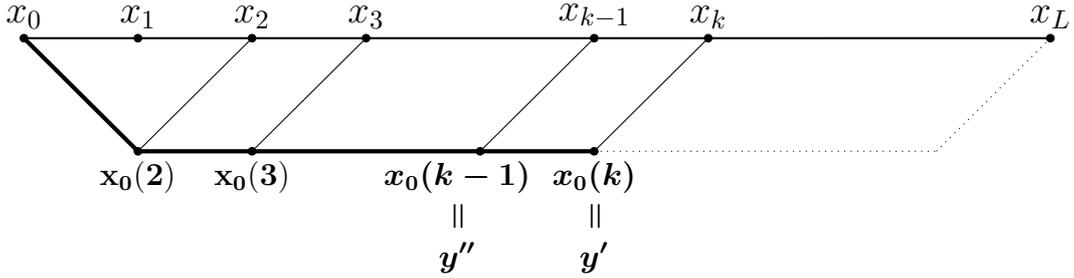

\end{proof}

An immediate but important consequence of the above theorem is the following corollary.

\begin{corollary}\label{cor_clockwise_spherical}
Let $G=(V,E)$ be self-centered Bonnet-Myers sharp. Let $x,y \in V$ be two different vertices, and consider any $x'\in [x,y]\cap S_1(x)$ with its antipole $y'={\rm ant}_{[x,y]}(x')$. Then $[x',y']=[x,y]$.
\end{corollary}

\begin{proof}[Proof of Corollary \ref{cor_clockwise_spherical}]
Theorem \ref{thm_spherical_step2} can be rephrased as $[x,y]\subseteq [x',y']$. 
Since $y=\textup{ant}_{[x',y']}(x)$ by Remark \ref{rem:ant_toggle}, we can interchange the roles of $x,y$ and $x',y'$ to obtain the opposite inclusion $[x',y']\subseteq [x,y]$. Therefore $[x',y']=[x,y]$, as desired.
\end{proof}

Now, we are ready to conclude the ultimate result of this section by using Corollary \ref{cor_clockwise_spherical} inductively.

\begin{theorem}\label{thm_spherical}
Let $G=(V,E)$ be self-centered Bonnet-Myers sharp. Then for two different vertices $x,y \in V$, the induced subgraph of the interval $[x,y]$ is antipodal. Therefore $G$ is strongly spherical. 
\end{theorem}

\begin{proof}[Proof of Theorem \ref{thm_spherical}]
Let $x' \in [x,y]$. The existence of a vertex $y' \in [x,y]$ satisfying
$$ [x,y] = [x',y'] $$
is proved via induction on $d(x,x')$. 

\smallskip

\underline{Base step:} The case $d(x,x')=0$ is trivial and $d(x,x')=1$
is covered by Corollary \ref{cor_clockwise_spherical}.

\smallskip

\underline{Inductive step:} We assume the statement of the Theorem is
true for all $d(x,x') \le m-1$ with $2 \le m \le {\rm{diam}}(G)$. Let
$x' \in [x,y]$ with $d(x,x') = m$. We choose a vertex
$x_1 \in [x,y] \cap S_1(x)$ such that $x_1,x'$ lie on a geodesic from
$x$ to $y$. By the induction hypothesis, there exists $y_1 \in [x,y]$
such that
$$ 
[x,y] = [x_1,y_1]. 
$$
Since $d(x_1,x') = m-1$, the induction hypothesis again implies the existence
of $y' \in [x_1,y_1] = [x,y]$ such that
$$
[x_1,y_1] = [x',y'] = [x,y]. 
$$
This finishes the proof. 
\end{proof}

\section{Classification of Bonnet-Myers sharp graphs with extremal
  diameters}
\label{sec:BMextrdiam}

In this section we show that there are no Bonnet-Myers sharp graphs of
extremal diameter (that is $L=2$ and $L=D$) which are not
self-centered. In other words, the only Bonnet-Myers sharp graphs
with diameter $L=2$ are cocktail party graphs and the only Bonnet-Myers
sharp graphs with diameter $L=D$ are hypercubes. 

\subsection{Characterisation of sharpness for $L = 2$}

\begin{theorem}
  Let $G=(V,E)$ be a $(D,2)$-Bonnet Myers sharp graph. Then $G$ is
  isomorphic to a cocktail party graph $CP(n)$ for $n \ge 2$.
\end{theorem}

\begin{proof}
  Let us first show that all Bonnet-Myers sharp graphs of diameter
  $L=2$ are necessarily self-centered: if there were a
  $(D,2)$-Bonnet-Myers sharp graph which is not self-centered, it
  would have a vertex adjacent to all other vertices and, by
  $D$-regularity, would have to be the complete graph $K_{D+1}$, which
  does not have diameter $2$. Now we employ the classification Theorem
  \ref{thm:main} for self-centered Bonnet-Myers sharp graphs and
  conclude that all Bonnet-Myers sharp graphs of diameter $L=2$ are
  cocktail party graphs $CP(n)$, $n \ge 2$.
\end{proof}

\subsection{Characterisation of sharpness for $L = D$}

We now show that the only Bonnet-Myers sharp graphs with diameter
equal to their degree are the hypercubes.

\begin{lemma}\label{l:match}
  Let $G=(V,E)$ be a $D$-regular graph.  Suppose an edge
  $\{x,y\} \in E$ is contained in no triangle and satisfies
  $\kappa(x,y) \geq \frac{2}{D}$. Then, every pair of adjacent edges
  $w\sim x\sim y$ is contained in a 4-cycle.
\end{lemma}

\begin{proof}
  This follows immediately from Proposition \ref{prop:curvcalc0}. 
\end{proof}

We recall the definitions of the small sphere structure and the
non-clustering property from \cite{LMP2} which have been the key
concepts to prove the rigidity result under sharpness of Bonnet-Myers
in the Bakry-\'Emery $\infty$-curvature (see \cite{LMP2}).

\begin{definition}\label{def: SSP NCP}
  Let $G=(V,E)$ be a $D$-regular graph and let $x \in V$.
  \begin{enumerate}
  \item[(SSP)] We say $x$ satisfies the \emph{small sphere property}
    (SSP) if
    \begin{align*}
      |S_2(x)| \leq {D \choose 2}.
    \end{align*}
  \item[(NCP)] We say $x$ satisfies the \emph{non-clustering property}
    (NCP) if, whenever $d_x^-(z) = 2$ holds for all $z \in S_2(x)$,
    one has that for all distinct $y_1,y_2\in S_1(x)$ there is at most one
    $z \in S_2(x)$ satisfying $y_1\sim z \sim y_2$.
  \end{enumerate}
\end{definition}

\begin{rem}
  For the arguments below, it is useful to understand structural
  properties of the hypercube $Q^n$. We view the vertices of $Q^n$ as
  the elements of $\{0,1\}^n$ which are connected if their Hamming
  distance is equal to $1$, and assume without loss of generality that
  $x=(0,0,...,0)$. Then for $1\le k \le n$, 
  $$ S_k(x) = \{(a_i)_{i} \in \{0,1\}^n | \ \sum_i a_i=k \}, $$ 
  which gives $\#S_k(x) = {D \choose k}$. In particular, $Q^n$
  satisfies (SSP).

  Moreover, for distinct $y_1,y_2\in S_k(x)$ we always have
  $d(y_1,y_2) \ge 2$. In the case $d(y_1,y_2)=2$, the entries
  of $y_1$ and $y_2$ differ in precisely two places and, consequently,
  $y_1,y_2$ has precisely one common neighbour in $S_{k-1}(x)$ (if $k \ge 1$)
  and one common neighbour in $S_{k+1}(x)$ (if $k \le n-1$). Therefore,
  $Q^n$ satisfies also (NCP).
\end{rem}

\begin{lemma}\label{l:SSPNCP}
  Let $G=(V,E)$ be a $D$-regular graph.  If $x \in V$ belongs to no
  triangle and if $\kappa(x,y) \geq 2/D$ for all $y\sim x$, then $x$
  satisfies $(SSP)$ and $(NCP)$.
\end{lemma}

\begin{proof}
  We first show $|S_2(x)| \leq {D \choose 2}$.  Let $z \in S_2(x)$ and
  let $y\in V$ s.t. $x \sim y \sim z$. Due to Lemma~\ref{l:match}, $z$
  is connected to at least two vertices from $S_1(x)$. By double
  counting, we have
  $$ 2 |S_2(x)| \le \sum_{z \in S_2(x)} d_x^-(z) = \sum_{y \in S_1(x)} d_x^+(y)
  \le (D-1) |S_1(x)| = (D-1)D, $$ which implies
  $|S_2(x)| \leq {D \choose 2}$. Therefore $x$ satisfies (SSP). 

  Next we prove (NCP) at $x$: For all distinct $y_1, y_2 \in S_1(x)$
  the pair of adjacent edges $y_1 \sim x \sim y_2$ is contained in a
  $4$-cycle by Lemma~\ref{l:match}, which means that there is
  $z \in S_2(x)$ with $y_1 \sim z \sim y_2$. Since
  $|S_2(x)| \le {D \choose 2}$, there is at most one such $z$ for each such pair
  $y_1, y_2 \in S_1(x)$.
\end{proof}

Now we state the main theorem of this section.

\begin{theorem}
  Let $G=(V,E)$ be $(D,L)$-Bonnet-Myers sharp with $L=D$. Then $G$ is
  the hypercube $Q^D$.
\end{theorem}

\begin{proof}
  Let $x$ be a pole. We write $B_N:= B_N(x)$ and $S_N:=S_N(x)$. 

  By Theorem~\ref{onespheredegree}(a), $x$ is not contained in any
  triangles, therefore, $B_1(x)$ is isomorphic to the $1$-ball in
  $Q^D$.

  Now suppose, the $N$-ball $B_N$ is isomorphic to the $N$-ball of the
  hypercube with $1 \leq N < D$. We want to show that the $(N+1)$-ball
  $B_{N+1}$ is then isomorphic to the $(N+1)$-ball of the hypercube,
  which would prove the theorem by induction.  

  By \eqref{eq:rec2} in Theorem~\ref{recursionformulas} and the fact
  that $d_x^0(z) = 0$ for all $z \in S_N$ (because of the structure of
  the $N$-ball in the hypercube), we observe $d_x^+(z)= D-N$ for all
  $z\in S_N$. Let $M:=\{\{v,w\}\subset S_N: d(v,w)=2\}$. Since the
  $N$-ball $B_N$ is isomorphic to the $N$-ball of the hypercube, we
  have 
  \begin{equation} \label{eq:Mest}
  |M| \geq {D \choose N-1}{D - N + 1 \choose 2}.
  \end{equation}
  Again, due to the structure of the $N$-ball in the hypercube, any
  pair $\{v,w\} \in M$ cannot have any common neighbours in $S_N$ and
  can have at most one common neighbour in $S_{N-1}$. Therefore, due to
  Lemma~\ref{l:match}, since $\kappa \geq \frac{2}{D}$, there exists
  $p:M \to S_{N+1}$ satisfying $v \sim p(\{v,w\}) \sim w$ for all
  $\{v,w\} \in M$.

  By \eqref{eq:rec3} in Theorem~\ref{recursionformulas}, every
  $z \in S_{N+1}$ satisfies $d_x^-(z) \le N+1$. We classify the
  vertices in $S_{N+1}$ by their backwards degree.  Let $a_s$ be the number
  of $z\in S_{N+1}$ with $d_x^-(z)=s$. Remark $a_s = 0$ for $s > N+1$.
  Therefore, the set $E(S_{N+1},S_N)$ of all edges joining $S_N$ and
  $S_{N+1}$, satisfies
  $$
  |E(S_{N+1},S_N)| = \sum_{s\leq N+1} s a_s.
  $$
  If $d_x^-(z) = s$ for some $z \in S_{N+1}$, then there are at most
  ${s \choose 2}$ pairs $\{v,w\} \in M$ with $p(\{v,w\})=z$. Thus,
  \begin{equation} \label{eq:ME_est}
   |M| \leq \sum_{N + 1 \geq s \geq 2} a_s {s\choose 2} \leq \frac{N}2
  \sum_{s\leq N+1} sa_s = \frac N 2 |E(S_{N+1},S_N)|.
   \end{equation}
   Note that the second inequality in \eqref{eq:ME_est} is an equality
   iff $a_s = 0$ for all $s < N+1$. Therefore, using \eqref{eq:Mest} and
   \eqref{eq:ME_est},
   \begin{align*}
     |E(S_{N+1},S_N)| \geq \frac 2 N |M| \geq \frac 2 N {D \choose N-1}{D - N + 1 \choose 2} = {D \choose N} (D-N)= |E(S_{N+1},S_N)|
   \end{align*}
   where the last equality follows since $|S_N| = {D \choose N}$ and
   since every $z \in S_N$ satisfies $d_x^+(z) = D-N$. Therefore, we
   have sharpness everywhere which means $a_s=0$ if $s \neq N+1$,
   i.e., $d_x^-(z)=N+1$ for all $z \in S_{N+1}$.  This implies from
   \eqref{eq:rec3} in Theorem \ref{recursionformulas} that
   \begin{equation} 
     \label{eq:SN+1-} d_x^0(z)=0 \quad \text{and} \quad
     d_x^+(z) = D-N-1 \quad \text{for all $z \in S_{N+1}$}
   \end{equation}
   and $|S_{N+1}| = {D \choose N+1}$.

   Since $B_N$ is isomorphic to the $N$-ball of the hypercube, any
   $y \in B_{N-1}$ is not contained in a triangle of $G$. Thus, we can
   apply Lemma~\ref{l:SSPNCP} and conclude that $(SSP)$ and $(NCP)$
   are satisfied for all $y \in B_{N-1}$. Using this fact and
   \eqref{eq:SN+1-}, we can apply \cite[Theorem~6.2]{LMP2} (with
   $k=N+1$) and conclude that $B_{N+1}$ is isomorphic to the
   $(N+1)$-ball of the hypercube $Q^D$. 

   Note the following slight subtlety in this last argument:
   \cite[Theorem~6.2]{LMP2} requires \emph{bipartiteness}of $G$, which is a
   priori not known. Instead, we apply this theorem to a modification
   of $G$. This can be done by the following gluing process of the
   induced graph $B_{N+1}(x)$ with an $(L-N-1)$-ball of the hypercube
   $Q^D$: Let $B'_{L-N-1}(x')$ be an $(L-N-1)$-ball of a hypercube
   $Q^D$ centered at $x'$ and $S'_{L-N-1}(x')$ be the corresponding
   $(L-N-1)$-sphere. Since $|S_{N+1}(x)| = {D \choose N+1} = |S'_{L-N-1}(x')|$
   and 
   $$ d_x^+(z) = D-N-1 = d_{x'}^-(z') \quad \text{and} \quad 
   d_x^0(z)= 0 = d_{x'}^0(z') $$ 
   for all $z \in S_{N+1}(x)$ and $z' \in S'_{L-N-1}(x')$, we can glue
   these two graphs via a bijective identification of the vertex sets
   of $S_{N+1}(x)$ and $S'_{L-N-1}(x')$. This guarantees that the new
   graph is bipartite and $D$-regular. 

   The proof is now finished by the induction principle.
\end{proof}

\section{Bonnet-Myers sharp graphs and Bakry-{\'E}mery curvature}
\label{sec:BM-BE-curvature}

\subsection{Bakry-\'Emery curvature}
\label{sec:BE-curvature}

Bakry-\'Emery curvature is a notion based on a fundamental identity in
Riemannian Geometry, called \emph{Bochner's Formula}, involving the
Laplace-Beltrami operator. This definition allows to introduce
Bakry-\'Emery curvature also on other spaces with a well-defined
Laplacian. The (normalized) Laplacian in our particular discrete
setting of a graph $G$ was given in \eqref{eq:Deltanorm}. In this
section, we will recall some fundamental properties which will be
relevant for relating Bonnet-Myers sharpness in the sense of Ollivier
Ricci curvature and Bakry-\'Emery curvature. More general details
about Bakry-\'Emery curvature can be found in \cite{CLP2018}. We
start with Bakry-\'Emery's $\Gamma$-calculus:

\begin{definition}[$\Gamma$ and $\Gamma_{2}$
  operators]\label{defn:GammaGamma2}
  Let $G=(V,E)$ be a finite simple graph.  For any two functions
  $f,g: V\to \mathbb{R}$, we define
  $\Gamma(f,g): V\rightarrow \mathbb{R}$ and
  $\Gamma_2(f,g): V\rightarrow \mathbb{R}$ by
  \begin{align*}
    2\Gamma(f,g)&:=\Delta(fg)-f\Delta g-g\Delta f;\\
    2\Gamma_2(f,g)&:=\Delta\Gamma(f,g)-\Gamma(f,\Delta g)-\Gamma(\Delta f,g).
  \end{align*}
  We write $\Gamma(f):=\Gamma(f,f)$ and
  $\Gamma_2(f,f):=\Gamma_2(f)$, for short.
\end{definition}

\begin{definition}[Bakry-\'Emery curvature]\label{defn:BEcurvature}
  Let $G=(V,E)$ be a finite simple graph. Let
  $\mathcal{K}\in \mathbb{R}$ and
  $\mathcal{N}\in (0,\infty)\cup\{\infty\}$. We say that a vertex
  $x\in V$ satisfies the \emph{curvature-dimension inequality}
  $CD(\mathcal{K},\mathcal{N})$ if, for any $f:V\to \mathbb{R}$, we
  have
  \begin{equation}\label{eq:CDineq}
    \Gamma_2(f)(x)\geq \frac{1}{\mathcal{N}}
    (\Delta f(x))^2+\mathcal{K}\Gamma(f)(x).
  \end{equation}
  We call $\mathcal{K}$ a lower Ricci curvature bound of $G$ at $x$,
  and $\mathcal{N}$ a dimension parameter. The graph $G=(V,E)$
  satisfies $CD(\mathcal{K},\mathcal{N})$ (globally), if all its
  vertices satisfy $CD(\mathcal{K},\mathcal{N})$. Let
  $\mathcal{K}_{G,x}(\infty)$ be the largest real number such that the
  vertex $x$ satisfies $CD(\mathcal{K}_{G,x}(\infty),\infty)$.
\end{definition}

We now recall results from \cite{CLP2018} which we will need for the
rest of this section. Note that the Bakry-\'Emery curvature
$\mathcal{K}_{G,x}$ in \cite{CLP2018} is based on the non-normalized
Laplacian which, in the case of $D$-regular graphs, can be easily
translated into the normalized setting presented here. Henceforth, we
will denote the Bakry-\'Emery curvature associated to the normalized
Laplacian by $\mathcal{K}^{\rm n}_{G,x}$ (for $D$-regular graphs, we
have $\mathcal{K}^{\rm n}_{G,x} = \frac{1}{D} \mathcal{K}_{G,x})$.

Let $G = (V,E)$ be a $D$-regular graph. Theorem 3.1 of \cite{CLP2018}
tells us that
\begin{equation}\label{upperbound}
  \mathcal{K}^{\rm n}_{G,x}(\infty) \leq \frac{2}{D} + \frac{\#_{\Delta}(x)}{D^{2}} = \frac{3+D-av_1^+(x)}{2D},
\end{equation} 
for every $x\in V$, where $av_1^+(x)$ was defined in
\eqref{eq:averdeg}.

We say, as in \cite{CLP2018}, that a $D$-regular graph $G = (V,E)$ is
\emph{$\infty$-curvature sharp} at $x \in V$ if \eqref{upperbound} holds true
with equality.

We now recall a method from \cite{CLP2018} that allows us to check if
a graph $G=(V,E)$ is $\infty$-curvature sharp at a vertex $x$.
Let $x \in V$ be an \emph{$S_{1}$-out regular} vertex, that is,
$d_{x}^{+}(y)$ is constant for all $y \sim x$.
Let $\{ y_1,\dots, y_d \}$ be the vertices of $S_1(x)$. We now define
two relevant (weighted) Laplacians $\Delta_{S_1(x)}$ and $\Delta_{S_1'(x)}$
on functions $f: S_1(x) \to \IR$ as follows: 
$$ \Delta_{S_1(x)}f(y_i) = \sum_{y_j: y_j \sim y_i} (f(y_j) - f(y_i)), $$
that is, $\Delta_{S_1(x)}$ be the non-normalized Laplacian of the
induced subgraph $S_1(x)$. Let $S_1'(x)$ be the graph with the same
vertex set $\{ y_1,\dots, y_d \}$ and an edge between $y_i$ and $y_j$
iff $|\{ z \in S_2(x) \mid y_i \sim z \sim y_j \}| \ge 1$, where
$\sim$ describes adjacency in the original graph $G$. We introduce the
following weights $w_{y_i y_j}'$ on the edges of $S_1'(x)$:
$$ w_{y_i y_j}' = \sum_{z \in S_2(x)} \frac{w_{y_i z}w_{z y_j}}{d_x^-(z)}. $$
Where $w_{u v} = 1$ if $u\sim v$ and $0$ otherwise.

The corresponding weighted Laplacian is then given by
$$ \Delta_{S_1'(x)} f(y_i) = \sum_{j: j \neq i} w_{y_i y_j}' (f(y_j) - f(y_i)). $$
Let $S_1''(x) = S_1(x) \cup S_1'(x)$, i.e., the vertex set of
$S_1''(x)$ is $\{ y_1,\dots, y_d \}$ and the edge set is the union
of the edge sets of $S_1(x)$ and $S_1'(x)$. Then the sum
$\Delta_{S_1(x)} + \Delta_{S_1'(x)}$ can be understood as the weighted
Laplacian $\Delta_{S_1''(x)}$ on $S_1''(x)$ with weights
$w'' = w + w'$. Note that all our Laplacians $\Delta$ are defined on
functions on the vertex set of $S_1(x)$. Let
$\lambda_{1}(\Delta_{S_1''(x)})$ denote the smallest non-zero
eigenvalue of $\Delta_{S_1''(x)}.$

Theorem 9.1 of \cite{CLP2018} tells us that an $S_1$-out regular
vertex $x$ in a $D$-regular graph $G$ is $\infty-$curvature sharp if
and only if $\lambda_{1}(\Delta_{S_1''(x)})\geq \frac{D}{2}$.

On a different note, we also provide the following general result on Cartesian product, which will be useful in the next subsection.

\begin{lemma} \label{lem:BEcart}
  Let $G_i = (V_i,E_i)$, $i=1,2$, be two connected, simple $D_i$-regular graphs
with diameters $L_i$, respectively. Assume we have
\begin{equation} \label{eq:Kiest}
\mathcal{K}_{G_i,x_i}^{\rm n}(\infty) \le \frac{1}{D_i} + \frac{1}{L_i}
\end{equation}
at $x_i \in V_i$, $i=1,2$. Then we have
\begin{equation} \label{eq:Kprodest}
\mathcal{K}_{G_1 \times G_2,(x_1,x_2)}^{\rm n}(\infty) \le
\frac{1}{D_1+D_2} + \frac{1}{L_1+L_2}.
\end{equation}
Moreover, if \eqref{eq:Kiest} holds with equality for $i=1,2$ and we
have $\frac{D_1}{L_1} = \frac{D_2}{L_2}$, then \eqref{eq:Kprodest}
holds also with equality.
\end{lemma}

\begin{proof}
	Let $G_i$ be $D_i$-regular with diameter $L_i$, $i = 1,2$ and $x_i
	\in V_i$ be the vertices satisfying
	$$ {\mathcal K}^{\rm n}_{G_i,x_i}(\infty) \le \frac{1}{D_i} +
	\frac{1}{L_i}. $$
	Then we have, using \cite[equation (7.26)]{CLP2018},
	\begin{eqnarray*} 
		{\mathcal K}^{\rm n}_{G_1 \times G_2,(x_1,x_2)}(\infty) &=& \frac{1}{D_1+D_2}
		\min_{i=1,2} D_i {\mathcal K}^{\rm n}_{G_i,x_i}(\infty) \\
		&\le&  \frac{1}{D_1+D_2} \min_{i=1,2} \left( 1+ \frac{D_i}{L_i} \right) \\
		&\le& \frac{1}{D_1+D_2} \left(1 + \frac{D_1+D_2}{L_1+L_2} \right) \\
		&=&  \frac{1}{D_1+D_2} + \frac{1}{L_1+L_2}.
	\end{eqnarray*}
	It is easy to see that in the case of equality in \eqref{eq:Kiest} for
	$i=1,2$, the same calculation leads to equality in \eqref{eq:Kprodest}.
\end{proof}

\subsection{The Bakry-\'Emery curvature of Bonnet-Myers sharp graphs}
\label{sec:BEcurvBMsharp}

As a consequence of Lemma \ref{lem:BEcart}, the following proposition show that the $\infty$-curvature sharpness is also preserved under taking Cartesian products of Bonnet-Myers sharp graphs (of the same ratios $\frac{D_i}{L_i}$).

\begin{proposition} \label{prop:cartprodBEcurv}
	Let $G_i=(V_i,E_i)$, $i=1,2$, be two $(D_i,L_i)$-Bonnet-Myers sharp graphs with
	${\mathcal K}^{\rm n}_{G_i,x_i}(\infty) = \frac{1}{D_i}+ \frac{1}{L_i}$ at $x_i \in V_i$. Assume furthermore that  $\frac{D_1}{L_1} = \frac{D_2}{L_2}$.
	Then the Cartesian product $G_1 \times G_2$ is also Bonnet-Myers sharp with
	$$ {\mathcal K}^{\rm n}_{G_1\times G_2,(x_1,x_2)}(\infty) = \frac{1}{D_1+D_2}+ \frac{1}{L_1+L_2}. $$
\end{proposition}

\begin{proof}
	The condition $\frac{D_1}{L_1} = \frac{D_2}{L_2}$ guarantees that the
	Cartesian product $G_1 \times G_2$ is, again, Bonnet-Myers sharp. The
	statement about the Bakry-\'Emery $\infty$-curvature at $(x_1,x_2) \in
	V_1 \times V_2$ follows immediately from Lemma \ref{lem:BEcart}.
\end{proof}

For Bonnet-Myers sharp graphs, we have the following $\infty$-curvature estimate at poles.

\begin{theorem} \label{thm:curv_relation}
  Let $G=(V,E)$ be a $(D,L)$-Bonnet-Myers sharp graph. Then we have at
  every pole $x \in V$:
  \begin{equation} \label{eq:be-est} {\mathcal K}^{\rm
      n}_{G,x}(\infty) \le \frac{1}{D} + \frac{1}{L}.
  \end{equation}
  Moreover, equality in \eqref{eq:be-est} is equivalent to the fact
  that $G$ is Bakry-\'Emery $\infty$-curvature sharp at $x$.
\end{theorem}

\begin{proof}
  Let $x \in V$ be a pole of $G$. Using \eqref{eq:rec1} in Theorem
  \ref{recursionformulas}, we have for every $y \in S_1(x)$:
  $$ d_x^+(y) = 1 + D \left( 1-\frac{2}{L} \right). $$
  This shows that $x$ is $S_1$-out regular with
  $av_1^+(x) =d_x^+(y) = 1 + D -\frac{2D}{L} $.  We know from
  \eqref{upperbound} that
  \begin{equation} \label{eq:BEcurv_est} 
  {\mathcal K}^{\rm n}_{G,x}(\infty) \le \frac{3+D-av_1^+(x)}{2D} = \frac{1}{2D} \left( 2 + \frac{2D}{L} \right)  = \frac{1}{D} + \frac{1}{L}. 
  \end{equation}
  Equality is equivalent to Bakry-\'Emery $\infty$-curvature
  sharpness.
\end{proof}

Since every graph has a pole, Theorem \ref{thm:curv_relation} immediately implies Theorem \ref{thm:BMsharp-BEeq}.

In case of self-centered Bonnet-Myers sharp graphs, Theorem \ref{thm:curv_relation} can be strengthened, where inequality \eqref{eq:be-est} becomes equality at all vertices, resulting in Theorem \ref{thm:aBM-BEcurv}:

\begin{theorembmbesharp} \label{thm:curv_rel_strong}
  Let $G$ be a self-centered $(D,L)$-Bonnet-Myers sharp graph. Then $G$ is Bakry-\'Emery $\infty$-curvature sharp at all vertices
  $x \in V$ and
  \begin{equation} \label{eq:curv1D1L} 
  {\mathcal K}_{G,x}^{\rm n}(\infty) = \frac{1}{D} + \frac{1}{L}. 
  \end{equation}
\end{theorembmbesharp}

\begin{proof} 
In view of Proposition \ref{prop:cartprodBEcurv}, it suffices to prove this theorem only for the graphs in the list of Theorem \ref{thm:main}.  We therefore start with a graph $G= (V,E)$ in the list of Theorem \ref{thm:main} and prove \eqref{eq:curv1D1L} for every vertex. Without loss of generality, we can assume $L \ge 2$ since $L=1$ implies $G = K_2$ which follows immediately from ${\mathcal K}_{K_2,x}^{\rm n}(\infty)=2$.

  Table \ref{table:BMsharp_examples} in Subsection \ref{sec:revisex}
  confirms that these graphs satisfy all the assumption of Proposition
  \ref{prop:cocktail-implies-str}, that is, all $\mu$-graphs of $G$
  are cocktail party graphs $CP(m)$ with
  $$ m = \frac{D-L}{L(L-1)} + 1, $$
  and all $1$-spheres of $G$ are strongly regular.

  Let $x \in V$. Since every vertex of $G$ is a pole, we
  know from \eqref{eq:BEcurv_est} that
  $$ {\mathcal K}_{G,x}(\infty) \le \frac{1}{D} + \frac{1}{L} $$
  with equality iff $x$ is $\infty$-curvature sharp. We have already
  seen in the proof of Theorem \ref{thm:curv_relation} that $x$ is
  $S_1$-out regular. So it only remains to show
  $\lambda_{1}(\Delta_{S_1''(x)})\geq \frac{D}{2}$.

  By Proposition \ref{prop:cocktail-implies-str}, the strongly regular induced
  $S_{1}(x)$ has parameters
  $$
  (\nu,k,\lambda,\mu) = 
  (D,\frac{2D}{L}-2,\frac{D-1}{L-1}-3,2\frac{D-L}{L(L-1)}). 
  $$ 

  Let $A$ denote the adjacency matrix of $S_{1}(x).$ Then, by Theorem
  \ref{onespheredegree},
  $\Delta_{S_{1}(x)} = A - (\frac{2D}{L}-2{\rm )Id}$. Let $y \sim
  x$. Then, by \eqref{eq:rec1} in Theorem \ref{recursionformulas},
  $d_{x}^{+}(y) = \frac{(L-2)D}{L}+1$. Since every $\mu$-subgraph of
  $G$ is $CP(m)$, we have
  $d_{x}^{-}(z) = 2m = \frac{2}{L-1}(\frac{D}{L}+L-2)$ for all
  $z\in S_{2}(x)$.  

  Let us first calculate the adjacency matrix $A'$ of the weighted
  graph $S_1'(x)$. Recall that the entries $\omega_{yy'}'$ of $A'$ are
  given by
  $$ \omega_{yy'}' = \sum_{z \in S_2(x), y \sim z \sim y'} \frac{1}{d_x^-(z)} = 
  \frac{1}{2m} \left| \{ z \in S_2(x) \mid y \sim z \sim y' \} \right|. $$
  Assume first that $y$ and $y'$ are not neighbours in the induced
  $S_1(x)$. There is a unique antipole of $x$ in
  $\mu(y_1,y_2)$, which is a vertex in $S_2(x)$. Therefore, we have
  $$ y,y' \in S_1(x), y \not\sim y' \quad \Rightarrow \quad \omega_{yy'}' = 
  \frac{1}{2m}. $$ 
  Now assume that $y \sim y'$. The edge $\{y,y'\}$ lies in precisely
  $\frac{2D}{L}-2$ triangles, one of them is $\{x,y,y'\}$ and there
  are precisely $\lambda = \frac{D-1}{L-1}-3$ triangles in the induced
  $S_1(x)$. The rest of triangles are in $1-1$ correspondence to
  vertices $z \in S_2(x)$ with $z \sim y$ and $z \sim y'$. Therefore
  we have
  $$ \left| \{ z \in S_2(x) \mid y \sim z \sim y' \} \right| = 
  \left(\frac{2D}{L} -2\right) - 1 - \left(\frac{D-1}{L-1} -3 \right)
  = \frac{2D}{L} - \frac{D-1}{L-1}. $$ 
  This implies that
  $$  y,y' \in S_1(x), y \sim y' \quad \Rightarrow \quad \omega_{yy'}' = 
  \left( \frac{2D}{L} - \frac{D-1}{L-1} \right) \frac{1}{2m}. $$ 
  Since the adjacency matrix $A^c$ of the complement of the induced
  $S_1(x)$ can be written as $A^c = J - {\rm Id} - A$, where $J$ is
  the all-one matrix, we have for the weighted adjacency matrix $A'$
  of $S_1'(x)$
  $$ A' = \frac{1}{2m} \left( \left( \frac{2D}{L} - \frac{D-1}{L-1} \right) A +
  A^c \right) = \frac{1}{2m} \left( \left( \frac{2D}{L} - \frac{D-1}{L-1} - 
  1 \right) A - {\rm Id} + {\rm J}  \right). $$
  Note that
  \begin{eqnarray*}
    \Delta_{S_1(x)} &=& A - \left( \frac{2D}{L} - 2 \right) {\rm Id}, \\
    \Delta_{S_1'(x)} &=& A' - {\rm diag}(v'), 
  \end{eqnarray*}
  with $v' = A' {\bf 1}$ where ${\bf 1}$ is the all-one vector. Since
  $A$ is the adjacency matrix of a
  $\left(\frac{2D}{L}-2\right)$-regular graph of size $D$, $v'$ is a constant
  vector with all entries equal to
  $$ \frac{1}{2m} \left( \frac{2D}{L} - \frac{D-1}{L-1} - 1 \right) 
  \left( \frac{2D}{L} - 2 \right) - \frac{1}{2m} + \frac{D}{2m}. $$
  Plugging this information into the formula for
  $\Delta_{S_1''(x)} = \Delta_{S_1(x)} + \Delta_{S_1'(x)}$ gives,
  $$\Delta_{S_1''(x)} = \frac{1}{2m}\left(\left(\frac{2D}{L}-\frac{D-1}{L-1}+2m-1\right)\left(A-\left(\frac{2D}{L}-2\right) {\rm Id}\right)-D\cdot {\rm Id}+{\rm J}\right).$$
  Observe that the three matrices $A,{\rm Id},{\rm J}$ pairwise
  commute.

  By Proposition \ref{prop:cocktail-implies-str}, the second largest
  eigenvalue of $A$ is $\frac{(D-L)(L-2)}{L(L-1)}$. Note that the
  eigenvector $w$ of the second largest eigenvalue is orthogonal to
  ${\bf 1}$ and, therefore, ${\rm J}w = 0$. Thus to complete the claim
  it remains to show that
  $$ \lambda_1(\Delta_{S_1''(x)}) = 
  \frac{-1}{2m}\left(\left(\frac{2D}{L}-\frac{D-1}{L-1}+2m-1\right)
  \underbrace{\left(\frac{(D-L)(L-2)}{L(L-1)}
        -\left(\frac{2D}{L}-2\right)\right)}_{= -
      \frac{D-L}{L-1}}-D\right)\geq \frac{D}{2}.$$ 
  Multiplying the whole expression by $2m$, we need to show that
  $$ \left( \left(\frac{2D}{L}-\frac{D-1}{L-1}+2m-1\right) \frac{D-L}{L-1} 
    +D\right) - mD \ge 0, $$ 
  which simplifies, after inserting $m = \frac{D-L}{L(L-1)} + 1$ into
  the expression, to
  $$ \frac{1}{L(L-1)^2} (L(L-2)+D)(D-L) \ge 0, $$
  which is obviously true since $L\leq D$ and $L \geq 2$.
\end{proof}

\subsection{A conjecture about Bakry-\'Emery curvature}

In this subsection, let us revisit the following conjecture mentioned
in the Introduction:

\begin{conjectureBM}
  Let $G=(V,E)$ be a connected, simple $D$-regular graph with diameter
  $L$. We then have
  \begin{equation} \label{eq:conj} \inf_{x \in V }{\mathcal K}^{\rm
      n}_{G,x}(\infty) \le \frac{1}{D} + \frac{1}{L}.
 \end{equation}
\end{conjectureBM}

A simple argument provides the following general estimate. The
challenge of the conjecture is thus to remove the final term in
\eqref{eq:almostconj}.

\begin{theorem}
  Let $G=(V,E)$ be a $D$-regular graph of diameter $L$. Then we have
  \begin{equation} \label{eq:almostconj} 
  \inf_{x \in V} {\mathcal K}^{\rm n}_{G,x}(\infty) \le \frac{1}{D} + \frac{1}{L} + \frac{1}{2D^2} \max_{x \in V} \#_\Delta(x). 
  \end{equation}
\end{theorem}
 
\begin{proof} The proof is a combination of the inequalities
  \eqref{eq:BM_BE} and \eqref{upperbound}.
\end{proof}

Here is a list of examples providing supporting evidence for this conjecture:

\begin{enumerate}
\item All graphs with $D \le L$: This is an immediate consequence of
  $$ \inf_{x \in V} {\mathcal K}_{G,x}^{\rm n}(\infty) \le \frac{2}{L} $$
  proved in \cite[Corollary 2.2]{LMP}.
\item All Bonnet-Myers sharp graphs: This follows immediately from Theorem
  \ref{thm:curv_relation}.
\item All strongly regular graphs: Note that a strongly regular graph $G=(V,E)$
  with parameters $(\nu,D,\lambda,\mu)$ satisfies, as all vertices $x \in V$,
  $$ \#_\Delta(x) = \frac{D \lambda}{2} \le \frac{D (D-2)}{2}, $$
  since $\lambda \le D-2$ ($G$ cannot be the complete graph). Using
  \eqref{upperbound}, this implies
  $$ {\mathcal K}_{G,x}^{\rm n}(\infty) \le \frac{2}{D} + \frac{D-2}{2D} =
  \frac{1}{D} + \frac{1}{2}. $$ 
\item All complete graphs: Note that the complete graph $G=K_n$ has
  degree $D = n-1$ and Bakry-\'Emery $\infty$-curvature (see
  \cite[Example 5.17]{CLP2018})
  $$ {\mathcal K}_{G,x}^{\rm n}(\infty) = \frac{D+3}{2D} \le \frac{1}{D} + 1 $$
  in all vertices $x$.
\item All demi-cube graphs: The even-dimensional demi-cubes $Q^{2n}_{(2)}$ satisfies ${\mathcal K}_{G,x}^{\rm n}(\infty) = \frac{1}{D}+\frac{1}{L}$
for all vertices $x$ (due to Theorem \ref{thm:aBM-BEcurv} as it is self-centered Bonnet-Myers sharp). On the other hand, the odd-dimensional demi-cube $Q^{2n+1}_{(2)}$ has Bakry-\'Emery $\infty$-curvature $$ K_{G,x}^{\rm n}(\infty) \le \frac{3+D-av_1^+(x)}{2D} = \frac{1}{n} = \frac{1}{L} < \frac{1}{D} + \frac{1}{L}, $$ where the upper bound $\frac{1}{D} + \frac{1}{L}$ will never be achieved.
\item All Johnson graphs: The Johnson graph $G=J(n,k)$ has the
  following Bakry-\'Emery $\infty$-curvature (see \cite[Example
  9.7]{CLP2018}) in all vertices $x$
  $$ {\mathcal K}^{\rm n}_{G,x}(\infty) = \frac{n+2}{2k(n-k)} 
  \le \frac{1}{D} + \frac{1}{L}, $$
  with vertex degree $D= k(n-k)$ and diameter $L=\min\{ k,n-k\}$.
\item All triangle-free graphs: Since $\#_\Delta(x)=0$ for all $x\in V$, \eqref{eq:almostconj} implies that $$ \inf_{x \in V} {\mathcal K}^{\rm n}_{G,x}(\infty) \le \frac{1}{D} + \frac{1}{L}.$$
\item Cartesian products: If \eqref{eq:conj} holds for the
  graphs $G_i=(V_i,E_i)$, $i=1,2$, then \eqref{eq:conj} holds also for
  the Cartesian product $G_1 \times G_2$ due to Lemma
  \ref{lem:BEcart}.
\end{enumerate}




\bigskip

{\bf{Acknowlegdements:}} The authors are grateful to David Bourne for
many useful discussions and contributions. All authors would
also like to thank the University of Science and Technology of China,
Hefei, for its hospitality. DC, SL and NP enjoyed the opportunity for
further discussions during the 2017 conference ``Analysis and Geometry
on Graphs and Manifolds'' at the University of Potsdam,
Germany. DC and FM would also like to thank the Max Planck
Institute for Mathematics, Bonn, for the opportunity to participate in
the 2017 event ``Metric Measure Spaces and Ricci Curvature''. Finally, FM wants to thank the German National Merit Foundation for financial support, and SK wants to thank Thai Institute for the Promotion of Teaching Science and Technology for his scholarship.

\end{document}